\documentclass[11pt]{amsart}

\pdfoutput=1

\usepackage{amsmath} 
\usepackage[foot]{amsaddr}

\usepackage{esint,amssymb} 
\usepackage{graphicx}
\usepackage{MnSymbol}
\usepackage{mathtools} 
\usepackage[colorlinks=true, pdfstartview=FitV, linkcolor=blue, citecolor=blue, urlcolor=blue,pagebackref=false]{hyperref}
\usepackage{orcidlink}
\usepackage{microtype}

\usepackage{bm}
\usepackage{dsfont}
\usepackage{mathrsfs}

\newtheorem{proposition}{Proposition}
\newtheorem{theorem}[proposition]{Theorem}
\newtheorem{lemma}[proposition]{Lemma}
\newtheorem{corollary}[proposition]{Corollary}

\theoremstyle{remark}
\newtheorem{remark}[proposition]{Remark}

\theoremstyle{definition}
\newtheorem{definition}[proposition]{Definition}

\numberwithin{equation}{section}
\numberwithin{proposition}{section}
\numberwithin{figure}{section}
\numberwithin{table}{section}

\newcommand{\N}{\mathbb{N}}
\newcommand{\Q}{\mathbb{Q}}
\newcommand{\R}{\mathbb{R}}

\newcommand{\E}{\mathbb{E}}
\renewcommand{\P}{\mathbb{P}}
\newcommand{\EE}{\mathbf{E}}
\newcommand{\PP}{\mathbf{P}}

\renewcommand{\le}{\leqslant}
\renewcommand{\ge}{\geqslant}
\renewcommand{\leq}{\leqslant}
\renewcommand{\geq}{\geqslant}

\renewcommand{\subset}{\subseteq}
\renewcommand{\bar}{\overline}
\renewcommand{\tilde}{\widetilde}

\renewcommand{\hat}{\widehat}

\newcommand{\Ll}{\left}
\newcommand{\Rr}{\right}
\renewcommand{\d}{\mathrm{d}}

\newcommand{\1}{\mathbf{1}}

\newcommand{\mcl}{\mathcal}
\newcommand{\msf}{\mathsf}
\newcommand{\mfk}{\mathfrak}
\newcommand{\msc}{\mathscr}

\newcommand{\al}{\alpha}

\newcommand{\de}{\delta}
\newcommand{\ep}{\varepsilon}
\newcommand{\eps}{\varepsilon}
\newcommand{\si}{\sigma}

\DeclareMathOperator{\on}{\text{ on }}

\DeclareMathOperator{\supp}{supp}

\newcommand{\upa}{\uparrow}

\newenvironment{e}{\begin{equation}}{\end{equation}\ignorespacesafterend}
\newenvironment{e*}{\begin{equation*}}{\end{equation*}\ignorespacesafterend}

\newcommand{\scdot}{{\,\cdot\,}} 

\newcommand{\bmart}{\mathbf{Mart}}

\newcommand{\D}{{D}}
\newcommand{\s}{d}
\newcommand{\sS}{{[D]}}
\newcommand{\spindist}{{\boldsymbol{\pi}}} 
\newcommand{\sq}{{\mathsf{q}}}
\renewcommand{\sp}{{\mathsf{p}}}
\newcommand{\cL}{{\mcl L}}
\newcommand{\one}{\mathds{1}}

\begin{document}

\author{Hong-Bin Chen\,\orcidlink{0000-0001-6412-0800}}
\address[Hong-Bin Chen]{Institut des Hautes Études Scientifiques, Bures-sur-Yvette, France}
\email{\href{mailto: hbchen@ihes.fr}{ hbchen@ihes.fr}}

\author{Victor Issa\,\orcidlink{0009-0009-1304-046X}}
\address[Victor Issa]{Department of Mathematics, ENS Lyon, Lyon, France}
\email{\href{mailto:victor.issa@ens-lyon.fr}{victor.issa@ens-lyon.fr}}

\author{Jean-Christophe Mourrat\,\orcidlink{0000-0002-2980-725X}}
\address[Jean-Christophe Mourrat]{Department of Mathematics, ENS Lyon and CNRS, Lyon, France}
\email{\href{mailto:jean-christophe.mourrat@ens-lyon.fr}{jean-christophe.mourrat@ens-lyon.fr}}

\subjclass[2010]{}

\title[The convex structure of the Parisi formula]{The convex structure of the Parisi formula for multi-species spin glasses}

\begin{abstract}
    We study the free energy of mean-field multi-species spin glasses with convex covariance function. For such models with $D$ species, the Parisi formula is known to be valid, and expresses the limit free energy as a supremum over monotone probability measures on~$\R_+^D$. We show here that one can transform this representation into a supremum over \emph{all} probability measures on~$\R_+^D$ of a \emph{concave} functional. We then deduce that the Parisi formula admits a unique maximizer.
    Using convex-duality arguments, we also obtain a new representation of the free energy as an infimum over martingales in a Wiener space. 

   \bigskip


    \medskip

\end{abstract}

\maketitle

{
  \hypersetup{linkcolor=black}
  \tableofcontents
}

%
%
%
%
%
%
\section{Introduction}

We study mean-field spin-glass models with multiple species. Fixing an integer $\D\ge 1$, we write $[D]=\{1,\dots,\D\}$ to denote the set of different species of spins.
For each $N\in\N$, we give ourselves $(I_{N,\s})_{\s\in\sS}$ a partition of $\{1,\dots,N\}$, and interpret each $I_{N,\s}$ as the set of indices that belong to the $\s$-th species. For any two configurations $\si,\si' \in \R^N$ and $d \in [D]$, we define the overlap for the $d$-th species as
\begin{equation}  
R_{N,d}(\sigma,\sigma') = \frac 1 N \sum_{n \in I_{N,d}} \si_n \si'_{n},
\end{equation}
and we set
\begin{equation}  
R_N(\sigma,\sigma') = (R_{N,d}(\sigma,\sigma'))_{d \in [D]}.
\end{equation}
We give ourselves a function $\xi : \R^D \to \R$ and a centered Gaussian field $(H_N(\sigma))_{\sigma\in [-1,+1]^N}$ with covariance
\begin{align}\label{e.def H_N}
    \E\Ll[ H_N(\sigma)H_N(\sigma')\Rr] = N\xi\Ll(R_N(\sigma,\sigma')\Rr).
\end{align}
Throughout, we assume that the function $\xi$ admits an absolutely convergent power-series expansion and satisfies $\xi(0) = 0$, and we make the key assumption that \emph{the function $\xi$ is convex on $\R_+^D$}. 
The definitions of the free energy and Gibbs measure of the spin glass require that we also choose a reference probability measure. For each species $d \in [D]$, we fix a probability measure~$\spindist_\s$ with support in $[-1,1]$ that is not a Dirac mass, and let
\begin{equation}\label{e.P_N=}
\d P_N(\sigma) = \bigotimes_{d \in [D]} \bigotimes_{n \in I_{N,d}} \d \spindist_\s(\si_n).
\end{equation}
In words, when $\sigma = (\si_1,\ldots, \si_N)$ is sampled according to $P_N$, the coordinates of $\sigma$ are independent, and the law of $\si_n$ is $\spindist_\s$ if  $n$ belongs to~$I_{N,d}$. For each $d \in [D]$, the proportion of spins that belong to the $d$-th species is given by 
\begin{equation}
\label{e.def.lambda_Nd}
\lambda_{N,\s} = |I_{N,\s}|/N, \quad \text{ and we set }  \lambda_N=\Ll(\lambda_{N,\s}\Rr)_{\s\in\sS}.
\end{equation} 
We assume throughout that for some $\lambda_\infty = (\lambda_{\infty,d})_{d \in [D]} \in (0,1)^D$, we have 
\begin{equation}
\label{e.lambda_infty}
    \lim_{N\to\infty} \lambda_N =\lambda_\infty.
\end{equation}
The main object of study  of this paper is the limiting value as $N \to +\infty$ of the free energy defined for every $t \geq 0$ by 
\begin{equation}
\label{e.def.FN.delta0}
   \bar F_N(t,\delta_0) = - \frac{1}{N} \E \log \int \exp \left( \sqrt{2t} H_N(\sigma) - Nt\xi \left( R_N(\sigma,\sigma) \right) \right)\d P_N(\sigma).
\end{equation}
The term $-Nt\xi \left( R_N(\sigma,\sigma) \right)$ is introduced as a convenience to simplify the expression of the limit. One possible way to remove it a posteriori is described in \cite{mourrat2020extending}; we also
point out that this term is constant if each $\spindist_\s$ is supported in $\{-1,+1\}$. In \eqref{e.def.FN.delta0}, the second argument of $\bar F_N$ is a Dirac mass at $0 \in \R^D$, and we will extend this function to more general arguments below.

In order to state the Parisi formula for the limit of \eqref{e.def.FN.delta0}, we introduce further notation. We write $\R_+ = [0,+\infty)$, and let $\mcl P(\R_+^D)$ denote the set of probability measures on $\R_+^D$. We say that $\mu \in \mcl P(\R_+^D)$ is \emph{monotone} when there exists an increasing map $\msf q:[0,1)\to\R_+^D$ such that $ \mu = \mathrm{Law}(\msf q(U))$ with $U$ a uniform random variable in $[0,1]$. Here and throughout, the word ``increasing'' (and likewise for ``decreasing'') is understood in the sense of non-strict inequalities, that is, we are asking here that $\msf q(t) - \msf q(s) \in \R_+^D$  for every $s \le t \in [0,1)$. 
We denote by $\mcl P^\upa(\R_+^D)$ the collection of monotone probability measures on $\R_+^D$.
For every $p \geq 1$, we denote by $\mcl P_p(\R_+^D)$ the set of probability measures with finite $p$-th moment, by $\mcl P_\infty(\R_+^D)$ the set of compactly supported probability measures on $\R_+^D$, and we set $\mcl P^\upa_p(\R_+^D) = \mcl P^\upa(\R_+^D) \cap \mcl P_p(\R_+^D)$. 

For all the models considered here, the Parisi formula is known to be valid \cite{barra2015multi, chen2024free, chenmourrat2023cavity, gue03,pan, pan.multi, Tpaper}. We give two versions of the Parisi formula, the first one in \eqref{e.parisi.usual} being the most classical, and the second one in \eqref{e.parisi.preferred} being the one which will best clarify its relationship with the sequel. For each $\nu \in \mcl P(\R^D_+)$ and $d \in [D]$, we denote by $\nu_d$ the $d$-th marginal of $\nu$; in other words, $\nu_d$ is the image of $\nu$ through the mapping $x \mapsto x_d$.  
The first term in the Parisi formula is given, for each $\nu \in \mcl P_1^\upa(\R_+^D)$, by
\begin{equation}  
\label{e.decomp.psi}
\psi(\nu) = \sum_{d=1}^D \lambda_{\infty,d} \, \psi_d (\nu_d),
\end{equation}
where for each $d \in [D]$, the function $\psi_d$ is the cascade transform of the reference measure $\spindist_\s$ defined precisely in~\eqref{e.psi_d=} (see also Lemma~\ref{l.ext_Parisi_PDE} for an alternative characterization). 
It follows from \cite{auffinger2015parisi} that the functions $\psi_1, \ldots, \psi_D$ are each concave over $\mcl P_1(\R_+)$. The formula \eqref{e.decomp.psi} therefore allows us to extend~$\psi$ into a concave (but not strictly concave) function on all of $\mcl P_1(\R_+^D)$.
For every $y \in \R^D$, we set
\begin{equation}\label{e.theta=}
\theta(y) = y \cdot \nabla \xi(y) - \xi(y),
\end{equation}
\begin{equation}  
\label{e.def.chi.*}
\xi^*(y) = \sup_{x \in \R^D_+} \{x \cdot y -  \xi(x)\},
\end{equation}
and similarly for $(t\xi)^*$ for any $t \ge 0$. We note that for every $t > 0$, we have $(t\xi)^*  = t \, \xi^* \Ll( \scdot/t \Rr)$. 
\begin{theorem}[Parisi formula]
\label{t.parisi.basic}
For every $t \ge 0$, we have 
\begin{align}  
\label{e.parisi.usual}
\lim_{N \to +\infty} \bar F_N(t,\de_0) 
& = \sup_{\mu \in \mcl P^\upa_\infty(\R_+^D)} \Ll\{ \psi((t \nabla \xi)(\mu))- t \int \theta \, \d \mu \Rr\} 
\\
\label{e.parisi.preferred}
& = \sup_{\nu \in \mcl P^\upa_\infty(\R_+^D)} \Ll\{ \psi(\nu) - \int (t\xi)^* \, \d \nu \Rr\} 
\end{align}
where $(t\nabla \xi)(\mu)$ denotes the image of the measure $\mu$ through the mapping $x \mapsto t \nabla \xi(x)$.
\end{theorem}

A measure $\mu \in \mcl P^\upa_\infty(\R_+^D)$ that realizes the supremum in \eqref{e.parisi.usual} is usually called a Parisi measure; this measure can often be interpreted as the asymptotic law of the overlap under the Gibbs measure (see \cite[Corollary~8.7]{chenmourrat2023cavity} or \cite[Theorem~1.4]{chen2024free} for precise statements). The existence of a Parisi measure can be obtained via coercivity and compactness arguments. 
For models with a single species ($D = 1$), it was proved in \cite{auffinger2015parisi} that the Parisi measure is unique. Here we extend this result to all multi-species models.

\begin{theorem}[Uniqueness of Parisi measure]  \label{t.unique_parisi}
For every $t \ge 0$, the supremum in \eqref{e.parisi.preferred} is achieved at exactly one measure, say $\bar \nu \in \mcl P^\upa_\infty(\R^D_+)$; and the supremum in~\eqref{e.parisi.usual} is achieved at all measures $\bar\mu \in \mcl P^\upa_\infty(\R^D_+)$ that satisfy $\bar \nu = (t \nabla \xi)(\bar \mu)$, and only at those measures.
\end{theorem}
As already noted, although the function $\psi$ is initially only defined on $\mcl P^\upa_1(\R^D_+)$, the formula in \eqref{e.decomp.psi} readily allows us to extend it into a concave function defined on all of $\mcl P_1(\R^D_+)$. In order to show Theorem~\ref{t.unique_parisi}, we aim to relax the maximization problem in \eqref{e.parisi.preferred} into one that is posed over all of $\mcl P_\infty(\R^D_+)$. In our extension to this variational problem, the last term in there is replaced by the following optimal-transport cost. For every $\mu,\nu \in \mcl P(\R_+^D)$, we denote by $\Pi(\mu,\nu)$ the set of all probability measures on $\R_+^D \times \R_+^D$ with first marginal $\mu$ and second marginal $\nu$.  For each $t \ge 0$ and $\mu,\nu \in \mcl P(\R_+^D)$, we set 
\begin{align}\label{e.T_t=}
    \mcl T_t(\mu,\nu) = \inf_{\pi\in\Pi(\mu,\nu)}\int (t\xi)^*(y-x)\d \pi(x,y).
\end{align}
In this introduction we have only defined $\bar F_N(t,\delta_0)$ in \eqref{e.def.FN.delta0}, but below we define $\bar F_N(t,\mu)$ for all $(t,\mu) \in \R_+ \times \mcl P^\upa_1(\R^D_+)$, and we state our next results for arbitrary arguments.
\begin{theorem} \label{t.convexoptimization}
For every $(t,\mu) \in \R_+ \times \mcl P^\upa_\infty(\R_+^D)$, we have 
     \begin{e} \label{e.convexoptimization}
         \lim_{N \to +\infty} \bar F_N(t,\mu) = \sup_{\nu \in \mcl P_\infty(\R_+^D)} \left\{ \psi(\nu) - \mcl T_t(\mu,\nu) \right\}.
     \end{e}
     Moreover, the functional inside the supremum in \eqref{e.convexoptimization} is concave in $\nu$. Finally, there is exactly one measure $\nu \in \mcl P_\infty(\R_+^D)$ that is monotone and realizes the supremum in \eqref{e.convexoptimization}. 
\end{theorem}
Using convex-duality arguments inspired by \cite{issa2024hopflike, uninverting}, we then deduce from Theorem~\ref{t.convexoptimization} a representation of the limit free energy as an infimum. 
We denote by $\mathfrak X$ the set of functions $\chi : \R^D_+ \to \R$ of the form $\chi(x) = \sum_{d = 1}^D \lambda_{\infty,d} \chi_d(x_d)$ where $\chi_1, \ldots, \chi_D : \R_+ \to \R$ are convex, $1$-Lipschitz, increasing functions that vanish at the origin. For every $\chi \in \mathfrak X$, $x \in \R^D_+$ and $t \ge 0$, we let 
\begin{align}
\label{e.def.psi.*}
    \psi_*(\chi) & = \inf_{\nu \in \mcl P_1^\upa(\R_+^D)} \left\{ \int \chi \d \nu - \psi(\nu) \right\}
\end{align}
and
\begin{align}
\label{e.def.St1}
    S_t \chi(x) & = \sup_{p \in \R_+^D} \left\{p \cdot x +  t \xi(p) - \chi^*(p)\right\} \\
\label{e.def.St2}
& = \sup_{y \in \R_+^D} \left\{ \chi(y) -(t\xi)^*(y-x) \right\} .
\end{align}
In \eqref{e.def.St1}, the definition of $\chi^*$ is as in \eqref{e.def.chi.*}. The equality between \eqref{e.def.St1} and \eqref{e.def.St2} will be explained below in Lemma~\ref{l.St1=St2}.

\begin{theorem} \label{t.hopflike}
For every $(t,\mu) \in \R_+ \times \mcl P^\upa_\infty(\R_+^D)$, we have 
     \begin{e} \label{e.convexoptimization2}
         \lim_{N \to +\infty} \bar F_N(t,\mu) = \inf_{\chi \in \mathfrak X} \left\{ \int S_t\chi\, \d \mu - \psi_*(\chi) \right\}.
     \end{e}
\end{theorem}
Finally, we describe a more concrete representation of the limit free energy as an infimum over martingales. Let $\msc P = (\Omega, \mcl F, \PP)$ be a probability space with associated expectation $\EE$, and let 
$(\mcl F_1({t}))_{t \ge 0}$, \ldots, $(\mcl F_D({t}))_{t \ge 0}$ be an independent family of filtrations over $\msc P$. We assume that all these $\sigma$-algebras are complete, that is, they all contain every subset of any null-measure set. We also assume that the probability space is rich enough to accommodate for the existence of independent Brownian motions $(B_1(t))_{t \ge 0}$, \ldots, $(B_D(t))_{t \ge 0}$ that are adapted to the filtrations $(\mcl F_1({t}))_{t \ge 0}$, \ldots, $(\mcl F_D({t}))_{t \ge 0}$ respectively.  We denote by $\bmart_1$, \ldots, $\bmart_D$ the spaces of bounded martingales over~$\msc P$ with respect to the filtrations $(\mcl F_1({t}))_{t \ge 0}$, \ldots, $(\mcl F_D({t}))_{t \ge 0}$ respectively, and we set
\begin{equation}
\label{e.def.bmart}
\bmart = \prod_{d = 1}^D \bmart_d.
\end{equation}
For every $\alpha = (\alpha_1,\ldots, \alpha_D) \in \bmart$, we define $\chi_\alpha : \R^D_+ \to \R$ such that for every $x = (x_1, \ldots, x_D) \in \R_+^D$, 
\begin{equation}  
\label{e.def.chi.alpha}
\chi_\alpha(x) = \sum_{d = 1}^D \lambda_{\infty,d} \int_0^{x_d} \EE \Ll[ \alpha_d(t)^2 \Rr] \, \d t.
\end{equation}
We also define, for every $t \ge 0$ and $x \in \R$,
\begin{equation}  
\label{e.def.phid}
\phi_d(t,x) = \log \int \exp \Ll( x \sigma - t |\sigma|^2 \Rr) \, \d \spindist_\s(x),
\end{equation}
and set, for every $t \ge 0$ and $y \in \R$,
\begin{equation}  
\label{e.def.phid*}
\phi_d^*(t,y) = \sup_{x \in \R} \Ll\{ x \cdot y - \phi_d(t,x)\Rr\}.
\end{equation}
\begin{theorem}
\label{t.explicit.uninverted}
For every $t \ge 0$, $\mu \in \mcl P^\upa_\infty(\R^D_+)$, and $T = (T_1, \ldots, T_D) \in \R_+^D$ such that $\supp(\mu) + t \nabla \xi([0,1]^D) \subset \prod_{d= 1}^D [0,T_d]$,  we have
\begin{multline}  \label{e.explicit.uninverted}
\lim_{N \to +\infty} \bar F_N(t,\mu) 
= \inf_{\al \in \bmart} \bigg\{ \sum_{d = 1}^D \lambda_{\infty,d} \, \EE \Ll[ \phi^*_d(T_d,\alpha_d(T_d)) - \sqrt{2} \alpha_d(T_d)B_d(T_d) \Rr]  \\
- \chi_\alpha(T) + \int S_t \chi_\al \d \mu\bigg\} .
\end{multline}
In addition the infimum in \eqref{e.explicit.uninverted} is reached at a unique $\bar \alpha \in \bmart$. This martingale $\bar \alpha$ is such that for every $d \in [D]$ and almost every $t \in [0,T_d]$, we have $\bar \alpha_d(t) = \partial_x \Phi_{\bar \nu_d}(t,X_d(t))$, where $\bar \nu$ is the unique monotone measure that realizes the supremum in \eqref{e.convexoptimization}, $\Phi_{\bar \nu_d}$ is defined in \eqref{e.Phi_nu_eqn} subject to the terminal condition \eqref{e.Psi_nu(nu^-1(1))=}, and $X_d$ is defined in \eqref{e.def.X}.
\end{theorem}

\bigskip

\noindent \textbf{Conjecture.} 
We conjecture that Theorems~\ref{t.hopflike} and \ref{t.explicit.uninverted} remain valid even when the function~$\xi$ is not assumed to be convex on $\R^D_+$, with the understanding that~$S_t\chi$ is defined according to \eqref{e.def.St1} (the equality of this quantity with \eqref{e.def.St2} would be lost in this case). To explain why, we start by recalling that when $\xi$ is convex over $\R^D_+$, the limit free energy is known \cite{chen2024free, chenmourrat2023cavity, chen2022hamilton, mourrat2019parisi, mourrat2020extending} to converge to the unique viscosity solution to the Hamilton--Jacobi equation
\begin{equation}  
\label{e.hj}
\begin{cases}  
\partial_t f - \int \xi(\partial_\mu f) \, \d \mu = 0 & \quad \text{ on } \R_+ \times \mcl P^\upa_2(\R_+^D) ,\\  
f(0,\cdot) = \psi & \quad \text{ on } \mcl P^\upa_2(\R_+^D).
\end{cases}
\end{equation}
The same result may also be valid in the case when $\xi$ is non-convex, and partial results in this direction were obtained in \cite{chen2024free, chenmourrat2023cavity, chen2022hamilton, mourrat2020nonconvex, mourrat2020free}.
As discussed more precisely in \cite{issa2024hopflike}, the right-hand side of \eqref{e.convexoptimization2} can be interpreted as a generalized version of the Hopf variational representation of the solution to~\eqref{e.hj}. In other words, there is a plausible heuristic argument connecting the Hamilton--Jacobi equation~\eqref{e.hj} with the variational representation in Theorem~\ref{t.hopflike}, and therefore also with that in Theorem~\ref{t.explicit.uninverted}; the idea being that since $\psi$ is concave, it can be represented as an infimum of affine functions, and we expect the solution to the PDE to be the infimum of the solutions started with these enveloping affine functions. One consequence of the results presented here is a proof that when $\xi$ is convex over $\R^D_+$, the solution to the Hamilton--Jacobi equation~\eqref{e.hj} can indeed be represented as the variational formula in \eqref{e.convexoptimization2}. We expect this link between the solution to \eqref{e.hj} and the variational representations in Theorems~\ref{t.hopflike} and \ref{t.explicit.uninverted} to remain valid even when the function $\xi$ is no longer assumed to be convex. In addition to being interesting on its own, this new representation could be helpful to complement the lower bound obtained in \cite{mourrat2020nonconvex, mourrat2020free} with the reverse bound.

\bigskip 

\noindent \textbf{Proof ideas.} For the choice of $\mu = \de_0$, and in view of \eqref{e.parisi.preferred}, the task of showing \eqref{e.convexoptimization} boils down to showing that 
\begin{equation}
\label{e.basic.task}
 \sup_{\nu \in \mcl P^\upa_\infty(\R_+^D)} \Ll\{ \psi(\nu) - \int (t\xi)^* \, \d \nu \Rr\} =  \sup_{\nu \in \mcl P_\infty(\R_+^D)} \Ll\{ \psi(\nu) - \int (t\xi)^* \, \d \nu \Rr\} .
\end{equation}
For each measure $\nu \in \mcl P(\R^D_+)$, there is a unique measure $\nu^\upa$ that has the same marginals as $\nu$ and is a monotone measure. Due to the equality of the marginals and \eqref{e.decomp.psi}, we have that $\psi(\nu) = \psi(\nu^\upa)$, so surely the most natural idea for showing \eqref{e.basic.task} would be to assert that 
\begin{equation}
\label{e.try.this}
\int (t\xi)^* \d \nu^\upa \leq \int (t\xi)^* \d \nu.
\end{equation}
We show in Proposition~\ref{p.xi^* transport inequality} that the inequality \eqref{e.try.this} is indeed valid when $D = 2$ and $\nabla \xi(0) = 0$. However, we also show in Remark~\ref{r. no higher generalization} that the inequality \eqref{e.try.this} is false in general, so another strategy must be developed towards the proof of our main results. 

Our approach consists instead of a careful analysis of the optimizers of the variational formula in \eqref{e.convexoptimization}. First, recall that 
the definition of $S_t \chi$ from \eqref{e.def.St2} implies that for every $x, y \in \R^D_+$, 
\begin{equation} 
\label{e.chi.stchi.duality}
\chi(y) - S_t \chi(x) \le (t\xi)^*(y-x),
\end{equation}
with equality when $y$ is a maximizer in \eqref{e.def.St2}. From this, we check in Lemma~\ref{l.hopf-lax with linear initial condtion} that
\begin{equation}  
\label{e.stchi.to.transport}
\int S_t \chi \, \d \mu = \sup_{\nu \in \mcl P_\infty(\R^D_+)} \Ll\{ \int \chi \, \d \nu - \mcl T_t (\mu,\nu) \Rr\} .
\end{equation}
Letting $\bar \nu  \in \mcl P_\infty(\R^D_+)$ denote a maximizer of \eqref{e.convexoptimization}, and $\chi$ denote the derivative of $\psi$ at $\bar \nu$, we observe that, thanks to the convexity of the mapping $\nu \mapsto \mcl T_t(\mu,\nu)$, we have that 
\begin{equation}  
\int \chi \, \d \bar \nu - \mcl T_t(\mu,\bar \nu ) = \sup_{\nu \in \mcl P_\infty(\R^D_+)} \Ll\{ \int \chi \, \d \nu - \mcl T_t(\mu,\nu) \Rr\} ,
\end{equation}
and thus, in view of \eqref{e.stchi.to.transport}, that 
\begin{equation}  
\label{e.kantorovich.potential}
\int \chi \, \d \bar \nu -\int  S_t \chi \, \d \mu = \mcl T_t(\mu,\bar \nu ) .
\end{equation}
Denoting by $\pi$ the optimal transport plan from $\mu$ to $\bar \nu$, the previous display tells us that the inequality \eqref{e.chi.stchi.duality} is in fact an equality for $\d \pi$-almost every $(x,y)$. Heuristically, we then want to exploit certain monotonicity properties of the mappings $\nabla \chi$ and $\nabla S_t \chi$ to deduce that if $\mu$ is monotone, then so is~$\bar \nu$. The details are a bit more subtle, as it may well be that the measure~$\bar \nu$ itself is actually not monotone, and we need to smear out the measure~$\mu$ in order to effectuate this transfer of monotonicity, but we ultimately show that we can build a maximizing measure for \eqref{e.convexoptimization} that is indeed monotone whenever $\mu$ itself is monotone. The bulk of the technical work consists in showing that $\nabla \chi$ and $\nabla S_t \chi$ satisfy the desired monotonicity properties. Once Theorem~\ref{t.convexoptimization} is shown, we obtain Theorem~\ref{t.unique_parisi} as a direct consequence, while the proofs of Theorems~\ref{t.hopflike} and \ref{t.explicit.uninverted} rely on convex-duality arguments that are inspired by \cite{issa2024hopflike, uninverting}.

\bigskip

\noindent \textbf{Related works.} 
The Parisi formula was initially predicted using the non-rigorous replica method in the case of the Sherrington-Kirkpatrick model (i.e.\ for $D = 1$, $\xi(r) = r^2$, and $P_1 = \frac 1 2 (\de_1 + \de_{-1})$) in \cite{mpv, parisi1979infinite, parisi1980order, parisi1980sequence, parisi1983order}, and was later proved rigorously in \cite{gue03, Tpaper}. A more conceptual proof based on the ultrametricity of the Gibbs measure was later developed \cite{pan.aom, pan}, and allowed all models with a single type of spins ($D = 1$) to be covered. This ultrametricity property was further leveraged in order to identify the limit free energy of spin glasses with multiple types and convex $\xi$ \cite{barra2015multi, pan.multi, pan.potts, pan.vec}. One key ingredient of the proof consists in showing that, up to a small perturbation of the energy function, the overlaps between the different types of spins must synchronize with one another; in other words, the joint law of the overlaps must be a monotone measure.
Theorem~\ref{t.convexoptimization} shows that one can in fact relax the final optimization problem into a ``desynchronized'' one, where the measure is no longer constrained to be monotone. 
Concurrent developments for spherical models include \cite{chen2013aizenman, tal.sph} for models with a single type of spins, and \cite{bates2022free} for multi-type models. For models with additional symmetries, the limit formula may in certain cases be simplified into an expression corresponding to an effective single-type model \cite{bates2024parisi, bates2025balanced, chen2024parisi, issa2024existence, mourrat2025color}. %

Models with non-convex $\xi$ are less well-understood, and even the existence of the limit free energy is not proved in general. Yet, under the assumption that the limit free energy exists, its value has been identified in \cite{subag2023tap2, subag2025tap1} for all spherical models in the case when $\xi$ is a monomial. Still for spherical models, the case of $D = 2$ and $\xi(x,y) = xy$ has been obtained unconditionally in \cite{aufchebi, baik2020free}, and the cases of $\xi(x,y) = x^p y^q$ and $\xi(x_1,\ldots, x_D) = x_1 \cdots  \, x_D$ have also been obtained unconditionally for particular choices of the parameter $\lambda_\infty$ in \cite{bates2025balanced, dartois2024injective}. 

Outside of these cases, even stating an appropriate conjecture for the limit free energy is a non-trivial task. Indeed, several possible candidate variational formulas for the limit free energy, inspired by the classical Parisi formula, were shown to be invalid in \cite[Section~6]{mourrat2020nonconvex}. A candidate for the limit free energy, defined in terms of the solution to a Hamilton-Jacobi equation, was proposed in \cite{chen2022hamilton, mourrat2020nonconvex, mourrat2019parisi, mourrat2020free}, and this candidate was shown there to be a lower bound for the limit free energy. Under the assumption that the limit free energy exists, a representation of this limit in terms of a critical point of an explicit functional was also found in \cite{chen2024free, chenmourrat2023cavity}. This description is in the spirit of the cavity fixed-point equations appearing in the physics literature. It does not completely characterize the limit however, as there may be more than one critical point in general. 

As discussed around \eqref{e.hj}, and despite the negative results of \cite[Section~6]{mourrat2020nonconvex}, we now believe that there is in fact a plausible candidate variational formula for the limit free energy for non-convex $\xi$. The point is that it does \emph{not} take a form similar to the classical Parisi formula, but rather a form as in \eqref{e.explicit.uninverted}. Such alternative ``un-inverted'' formulas first appeared in the context of models with a single type of spins in \cite{uninverting}, and were further developed in \cite{issa2024hopflike}, with results close to Theorem~\ref{t.hopflike} for vector spin glasses. A key ingredient for our analysis as well as in \cite{uninverting} is the fact that the function~$\psi$ is concave. This was first proved for models with a single type of spins in~\cite{auffinger2015parisi}; the uniqueness of the Parisi measure was also obtained there in this case. For spherical models, the uniqueness of the Parisi measure can be seen more directly, and a similar convex-duality analysis was performed in \cite{jagannath2017low, jagannath2018bounds}. We also mention that for vector spin glasses with convex $\xi$, the uniqueness of the Parisi measure was shown to be valid in \cite{chen2025uniqueness} provided that the measure argument $\mu$ does not charge any single point.  

These ``un-inverted'' formulas are in turn inspired by a series of works concerning the construction of efficient algorithms that aim to identify configurations $\sigma$ in the support of $P_N$ such that $H_N(\sigma)$ is as large as possible. For many models, an explicit quantity $\msf{ALG}$ was identified as a variational formula analogous to the Parisi formula, such that the following holds. On the one hand, there exists an efficient algorithm that outputs a configuration~$\sigma$ in the support of $P_N$ such that $H_N(\sigma)/N$ is approximately equal to $\msf{ALG}$. On the other hand, for any $\ep > 0$, if an algorithm is Lipschitz continuous with respect to its entries, then it cannot reliably output a configuration $\sigma$ in the support of $P_N$ such that $H_N(\sigma)/N$ exceeds $\msf{ALG} + \ep$. Whether or not this quantity $\msf{ALG}$ coincide with the limit of $\sup_{\sigma \in \supp P_N} H_N(\sigma)$ depends on the choice of $\xi$. At least for models with a single type, this quantity $\msf{ALG}$ can also be written as an ``un-inverted'' formula very similar to that found in \cite{uninverting}. 

An algorithm that indeed achieves the correct value $\msf{ALG}$ was first put forward in \cite{subag2021following} for spherical models with a single type of spins. The algorithm is based on a Hessian ascent. Using a different approach based on approximate message passing algorithms, this result was extended to models with $\pm 1$ spins in \cite{elalaoui2021optimization, montanari2021optimization, sellke2024optimizing}. For models with $\pm 1$ spins, a Hessian-ascent algorithm was also constructed in \cite{jekel2025potential}. 
The converse statement, that no Lipschitz algorithm can identify configurations with energy exceeding $\msf{ALG}$, is based on the notion of overlap gap property \cite{gamarnik2021overlap-survey, gamarnik2021overlap-paper, gamarnik2022disordered, huang2025tight}. 
Remarkably, the identification of this threshold $\msf {ALG}$ has been extended to spherical models with multiple types, including for models with non-convex~$\xi$ for which the limit free energy is not currently known \cite{huang2023algorithmic, huang2025tight}. 

\bigskip

\noindent \textbf{Organization of the paper.} In Section~\ref{s.def.free.energy}, we define the enriched free energy $\bar F_N(t,\mu)$ for arbitrary arguments $(t,\mu) \in \R_+ \times \mcl P_1(\R_+)$. In Section~\ref{s.initial condition}, we study properties of $\psi(\mu) = \lim_{N\to \infty} \bar F_N(0,\mu)$, including its regularity and concavity and the decomposition in \eqref{e.decomp.psi}. In particular, we show that the derivative of $\psi$ at any measure is an element of the set $\mfk X$ defined above \eqref{e.def.psi.*}. In Section~\ref{s.2d}, we focus on the two-species case ($D = 2$) with no external field ($\nabla \xi(0) = 0$), and show \eqref{e.try.this} in this particular case, leading to a simple proof of \eqref{e.basic.task}, with some long but unsurprising proof deferred to Appendix~\ref{s.xi^* ineq}. Section~\ref{s.hjcones} is focused on the proof that $\nabla S_t \chi$ satisfies the desired monotonicity properties whenever $\chi \in \mfk X$. We show that the mapping $(t,x) \mapsto S_t\chi(x)$ can be interpreted as the solution to a Hamilton-Jacobi equation, and we rely on this property to obtain the desired result. Section~\ref{s.optim} contains the proofs of \eqref{e.stchi.to.transport} and \eqref{e.kantorovich.potential}. We also show that the variational problems appearing in \eqref{e.parisi.preferred} and \eqref{e.convexoptimization} are equal. We complete the proofs of Theorems~\ref{t.parisi.basic} to \ref{t.hopflike} in  Section~\ref{s.proofs.main}. Finally, we develop a more explicit dual representation of $\psi$ and use it to prove Theorem~\ref{t.explicit.uninverted} in Section~\ref{s.explicit}.

%
%
%
%
%
%
\section{Definition of the enriched free energy}
\label{s.def.free.energy}

The goal of this section is to give a precise definition of the enriched free energy $\bar F_N(t,\mu)$. We start by giving a definition of this quantity in the case when the measure $\mu \in \mathcal P^\upa(\R_+^D)$ has a finite support. For every $x,y \in \R^{D}$, we write $x \leq y$ when $y-x \in \R_+^D$.
Let $K \in \N$ and let $\mu \in \mcl P^\upa(\R_+^D)$ be of the form
\begin{e} \label{e.piecewise}
    \mu = \sum_{k = 0}^K (\zeta_{k+1}-\zeta_k) \delta_{q_k} \quad \text{with marginals}\quad \mu_\s = \sum_{k = 0}^K (\zeta_{k+1}-\zeta_k) \delta_{q_{k,\s}}
\end{e}
for each $\s\in\sS$, where $\zeta_k$ and $q_k = (q_{k,\s})_{\s\in\sS}$ satisfy
\begin{gather}
    0 = q_{-1,\s} \leq q_{0,\s} < q_{1,\s} < \dots < q_{K,\s}, \label{e.qs=}
    \\
    0 = \zeta_0 < \zeta_1 < \dots < \zeta_K < \zeta_{K+1}=1.\label{e.zetas=}
\end{gather}
The construction of $\bar F_N(t,\mu)$ involves a random probability measure with ultrametric properties called the Poisson--Dirichlet cascade. We briefly introduce this object and refer to \cite[Section~2.3]{pan} or \cite[Section~5.6]{HJbook} for a more detailed explanation. We define
\begin{e}
    \mcl A = \N^0 \cup \N^1 \cup \dots \cup \N^K
\end{e}
with the understanding that $\N^0 = \{ \emptyset \}$. We think of $\mcl A$ as a tree rooted at $\emptyset$ such that each vertex of depth $k < K$ has countably many children. For each $k < K$ and $\alpha = (n_1,\dots,n_k) \in \N^k$, the children of $\alpha$ are the vertices of the form 
\begin{e}
    \alpha n = (n_1,\dots,n_k,n) \in \N^{k+1}.
\end{e}
The depth of $\alpha = (n_1,\dots,n_k)$ is denoted by $|\alpha| = k$ and for every $l \leq k$, we write 
\begin{e}
    \alpha_{|l} = (n_1,\dots,n_l)
\end{e}
to denote the ancestor of $\alpha$ at depth $l$. Given two leaves $\alpha,\beta \in \N^K$, we denote by $\alpha \wedge \beta$ the generation of the most recent common ancestor of $\alpha$ and $\beta$, that is
\begin{e}
    \alpha \wedge \beta = \sup \{ k \leq K:\: \alpha_{|k} = \beta_{|k} \}.
\end{e}
We attach an independent Poisson process to each non-leaf vertex $\alpha \in \mcl A$ with intensity measure
\begin{e}
    x^{-1-\zeta_{|\alpha|+1}} \d x.
\end{e}
We order increasingly the points of those Poisson processes and denote them by $u_{\alpha 1} \geq u_{\alpha 2} \geq \dots$. For every $\alpha \in \N^K$, we set $w_\alpha = \prod_{k  = 1}^K u_{\alpha_{|k}}$ and define
\begin{e}
   v_\alpha = \frac{w_\alpha}{\sum_{\beta \in \N^K} w_\beta}. 
\end{e}
The Poisson--Dirichlet cascade associated to $(\zeta_k)_{1\leq K+1}$ in \eqref{e.zetas=} is the random probability measure on $\N^K$ (the leaves of the tree $\mcl A$) whose weights are given by $(v_\alpha)_{\alpha \in \N^K}$.

Now, let $(v_\alpha)_{\alpha \in \N^K}$ be the Poisson--Dirichlet cascade associated to $(\zeta_k)_{1\leq K+1}$ in \eqref{e.zetas=}, chosen to be independent of $H_N$. 
Let $(z_{\beta,\s})_{\beta \in \mcl A, d \in [D]}$ be a family of independent standard Gaussian variables. We choose $(z_{\beta,\s})_{\beta \in \mcl A}$ independent of $(v_\alpha)_{\alpha \in \N^K}$ and $H_N$. For every $\alpha \in \N^K$, we set 
\begin{e}\label{e.w^mu_d=}
    w^{\mu_\s}(\alpha) = \sum_{k = 0}^K (q_{k,\s} - q_{k-1,\s})^{1/2} z_{\alpha_{|k},\s}
\end{e}
with $(q_{k,\s})_{0\leq k\leq K}$ given in~\eqref{e.qs=}.
The centered Gaussian process $(w^{\mu_\s}(\alpha))_{\alpha \in \N^K}$ has the following covariance structure 
\begin{e}
    \E \left[ w^{\mu_\s}(\alpha)w^{\mu_\s}(\alpha') \right] = q_{\alpha \wedge \alpha',\s}.
\end{e}
Let $(\omega^{\mu_d}_n)_{n\in\N}$ be i.i.d.\ copies of $\omega^{\mu_d}$.
For each $N\in\N$, we define
\begin{align}
    W^\mu_N(\sigma,\alpha) = \sum_{\s\in \sS}\sum_{n\in I_{N,\s}}w^{\mu_\s}_n(\alpha)\cdot \sigma_n
\end{align}
which is a centered Gaussian process with covariance
\begin{align}
    \E \Ll[ W^\mu_N(\sigma,\alpha) W^\mu_N(\sigma',\alpha')\Rr]= N q_{\alpha\wedge\alpha'}\cdot R_N(\sigma,\sigma').
\end{align}
For $t\in\R_+$ and $\mu$ given in~\eqref{e.piecewise}, we define the enriched Hamiltonian
\begin{align*}
    H^{t,\mu}_N(\sigma,\alpha)= \sqrt{2t}H_N(\sigma) - t N \xi \Ll(R_N(\sigma,\sigma)\Rr)  
    + \sqrt{2}W^\mu_N(\sigma,\alpha) - Nq_K\cdot R_N(\sigma,\sigma),
\end{align*}
where $H_N$ is given as in~\eqref{e.def H_N}. We define
\begin{e} \label{e.enriched free energy}
    \bar F_N(t,\mu) = -\frac{1}{N} \E \log \int \sum_{\alpha \in \N^K} \exp \left(H_N^{t,\mu}(\sigma,\alpha) \right) v_\alpha \d P_N(\sigma).
\end{e}
The expression in the previous display is Lipschitz with respect to $(t,\mu)$. More precisely, for every $t,t' \in \R_+$ and every $\mu, \mu'$ of the form in~\eqref{e.piecewise}, we have 
\begin{e}\label{e.F_L^1}
    |\bar F_N(t,\mu) - \bar F_N(t',\mu')| \leq \msf W_1(\mu,\mu') +|t-t'| \sup_{|a| \leq 1} |\xi(a)|,
\end{e}
where $\msf W_1$ denotes the Wasserstein $L^1$ distance between $\mu$ and $\mu'$. This is borrowed from~\cite[Proposition~4.1]{chen2024free}, where the term $\msf W_1(\mu,\mu')$ is written as $\int_0^1|\sq(s)-\sq'(s)|\d s$ with $\mu \stackrel{\d}{=} \mathrm{Law}(\sq(U))$ and $\mu' \stackrel{\d}{=} \mathrm{Law}(\sq'(U))$ for a uniform random variable $U$ over $[0,1]$. The equality between the two expressions can be obtained from a straightforward modification of~\cite[Proposition~2.5]{mourrat2020free}.

As a consequence, the free energy $\bar F_N$ admits a unique Lipschitz extension to $\R_+ \times \mcl P^\upa_1(\R_+^D)$.

Given $w^{\mu_\s}$ in~\eqref{e.w^mu_d=} and $\spindist_\s$ in~\eqref{e.P_N=}, each $\psi_\s$ appearing in~\eqref{e.decomp.psi} is given by
\begin{align}\label{e.psi_d=}
    \psi_{d}(\mu_\s)=-\E \log \int\sum_{\alpha\in\N^K}\exp\Ll(\sqrt{2}w^{\mu_\s}(\alpha)\tau- q_{K,\s}\tau^2\Rr)v_\alpha\d \spindist_\s(\tau)
\end{align}
at every $\mu$ of the form in~\eqref{e.piecewise}. A similar property as in~\eqref{e.F_L^1} holds for $\psi_\s$ and thus $\psi_\s$ can be extended to $\mcl P^\upa_1(\R_+)=\mcl P_1(\R_+)$.
By an invariance property of $(\nu_\alpha)_{\alpha\in\N^K}$ (see e.g.\ \cite[Corollary~5.26]{HJbook} and~\cite[Lemma~4.11]{chen2024free}), we have
\begin{align}\label{e.F_N(0,mu)=}
    \bar F_N(0,\mu) = \sum_{\s\in\sS}\lambda_{N,\s} \psi_\s(\mu_\s),\qquad\forall \mu\in\mcl P_1^\upa(\R_+^D).
\end{align}
Therefore, due to~\eqref{e.lambda_infty} and~\eqref{e.decomp.psi}, we have 
\begin{align}\label{e.limF_N(0,mu)=psi(mu)}
    \lim_{N\to\infty}\bar F_N(0,\mu) = \psi(\mu),\qquad\forall \mu\in\mcl P_1^\upa(\R_+^D).
\end{align}
In order to state the Parisi formula for $\bar F_N(t,\mu)$ at a general measure $\mu$, we introduce the following notation. Let $\mcl Q$ be the collection of càdlàg increasing maps $\sq:[0,1)\to\R_+^D$, and let $U$ denote a uniform random variable on $[0,1]$. For any $\sq \in \mcl Q$, we write
\begin{align}\label{e.L_q=}
    \cL_\sq = \mathrm{Law}(\sq(U)) \in \mcl P^\upa(\R_+^D).
\end{align}
We also set $\mcl Q_\infty = \mcl Q\cap L^\infty([0,1);\R^D)$. It is clear that $\sq\mapsto \cL_\sq$ is a bijection from $\mcl Q$ to $\mcl P^\upa(\R_+^D)$, and also from $\mcl Q_\infty$ to $\mcl P^\upa_\infty(\R_+^D)$. 
The following is a restatement of \cite[Proposition~6.1]{chen2024free} which adapts \cite[Proposition~8.1]{chenmourrat2023cavity} from the setting of vector spin glasses. 
It generalizes~\eqref{e.parisi.usual} in Theorem~\ref{t.parisi.basic}.

\begin{proposition}[Parisi formula]\label{p.parisi_gen}
For every $t\geq0$ and $\sq \in \mcl Q_\infty$, denoting $\mu = \mcl L_{\sq}\in\mcl P^\upa_\infty(\R_+^D)$, we have
\begin{align}\label{e.p.parisi_gen}
    \lim_{N \to +\infty} \bar F_N(t,\mu) = \sup_{\sp \in \mcl Q_\infty} \Ll\{ \psi\Ll(\cL_{\sq+t\nabla\xi(\sp)}\Rr)- t \int \theta \, \d\cL_\sp \Rr\} .
\end{align}
\end{proposition}
In~\eqref{e.p.parisi_gen}, the notation $\nabla\xi(\sp)$ stands for the path $s\mapsto \nabla\xi(\sp(s))$, and $\int\theta\d\cL_\sp$ can be explicitly written as $\int_0^1\theta(\sp(s))\d s$. The right-hand side in \eqref{e.p.parisi_gen} can be connected to the unique solution of a Hamilton--Jacobi equation \cite{chenmourrat2023cavity, mourrat2020nonconvex, mourrat2019parisi, mourrat2020free, mourrat2020extending}.
%
%
%
%
%
%
%
%
\section{Properties of the enriched free energy at infinite temperature} \label{s.initial condition}

In this section, we derive useful properties of $\psi$. Due to the decomposition in~\eqref{e.decomp.psi}, we mostly focus on the study of $\psi_\s$ for each $d \in [D]$.

\subsection{Properties of \texorpdfstring{$\psi_\s$}{psi\_d}}

We fix $\s\in\sS$, and use the following notation in this subsection:
\begin{align}\label{e.barpsi=psi_d}
    \spindist_\circ = \spindist_\s \qquad\text{and}\qquad \psi_\circ = \psi_\s.
\end{align}

For $\nu\in \mcl P(\R_+)$, we denote by $t\mapsto \nu(t)= \nu([0,t])$ its cumulative distribution function, which is a right-continuous increasing function taking values in $[0,1]$.
When $\nu \in \mcl P_\infty(\R_+)$, we denote by
\begin{align*}
    \nu^{-1}(1) = \inf\{t\in\R_+:\: 1\leq \nu(t)\}
\end{align*}
the right endpoint of the support of $\nu$.
For $\nu \in \mcl P_\infty(\R_+)$, we consider the backward parabolic equation
\begin{align}\label{e.Phi_nu_eqn}
-    \partial_t\Phi_\nu(t,x) =  \partial_x^2 \Phi_\nu(t,x) + \nu(t) \Ll(\partial_x \Phi_\nu\Rr)^2(t,x)  \qquad 
\end{align}
with $(t,x)$ ranging in $[0,\nu^{-1}(1)] \times  \R$, and with the terminal condition
\begin{align}\label{e.Psi_nu(nu^-1(1))=}
    \Phi_\nu(\nu^{-1}(1),x)= \log\int\exp\Ll(x\sigma - \nu^{-1}(1)\sigma^2\Rr)\d \spindist_\circ(\sigma),\quad\forall x \in\R.
\end{align}
We call the equation \eqref{e.Phi_nu_eqn} the \emph{Parisi PDE} associated with $\nu$. 
The fact that the function $\Phi_\nu$ is well-defined can be found in \cite{jagannath2016dynamic}. When the support of $\nu$ is finite, one can use the Cole-Hopf transform to solve~\eqref{e.Phi_nu_eqn} on each time sub-interval of $[0,\nu^{-1}(1))$ on which $\nu(\cdot)$ is constant; and then one can argue by continuity to build $\Phi_\nu$ for arbitrary $\nu$ (see for instance \cite[Section~6.5]{HJbook}).
We extend $\Phi_\nu$ to $\R_+\times \R$ by setting
\begin{align}\label{e.Phi_nu_trivial}
    \Phi_\nu(t,x) 
    = \log \int \exp\Ll(x\sigma -t\sigma^2\Rr)\d \spindist_\circ(\sigma), \quad\forall (t,x) \in [\nu^{-1}(1),\infty)\times \R.
\end{align}
A simple computation reveals that this extension also satisfies the equation~\eqref{e.Phi_nu_eqn} on $[\nu^{-1}(1),\infty)\times \R$.
We call $\Phi_\nu:\R_+\times\R\to\R$ the \emph{Parisi PDE solution} associated with $\nu \in \mcl P_\infty(\R_+)$.
We summarize the basic properties of the Parisi PDE solution in the following lemma.

\begin{lemma}[Regularity of Parisi PDE]\label{l.ext_Parisi_PDE}
For every $\nu \in \mcl P_\infty(\R_+)$ and $k\in \N = \{1,2,\ldots\}$, the derivative $\partial^k_x \Phi_\nu$ exists everywhere and is bounded uniformly over $\nu,t,x$. 
In particular, $|\partial_x\Phi_\nu|\leq 1$ everywhere.
For every $k\in \N\cup\{0\}$, there is a constant $C_k$ such that for every $\nu,\nu'\in \mcl P_\infty(\R_+)$ and $(t,x) \in \R_+ \times \R$, we have
\begin{align}\label{e.compare_Phi_nu}
    \Ll|\partial^k_x\Phi_\nu(t,x) - \partial^k_x\Phi_{\nu'}(t,x)\Rr|\leq C_k \int_t^\infty \Ll|\nu(\tau)-\nu'(\tau)\Rr|\d \tau.
\end{align}
For every $\nu \in \mcl P_1(\R_+)$, we have
\begin{align}\label{e.barpsi=Phi}
    \psi_\circ(\nu) = -\Phi_\nu(0,0).
\end{align}
\end{lemma}

\begin{proof}
First, assume $\nu,\nu'\in \mcl P_\infty(\R_+)$.
For each $\nu$, there is an explicit expression for $\Phi_\nu$ given in~\cite[Lemma~4.5 and Remark~4.6]{chen2024simultaneous} (with $D=1$ and $L(t)=t$ for $t\in\R_+$ therein, and with $\nu$ and $P_1$ substituted for $\alpha$ and $\mu$ therein). In particular, both~\eqref{e.Psi_nu(nu^-1(1))=} and~\eqref{e.Phi_nu_trivial} are satisfied. The relation to the initial condition as in~\eqref{e.barpsi=Phi} is given in~\cite[Lemma~5.4]{chen2024simultaneous}. The fact that $\Phi_\nu$ satisfies the equation and the regularity of $\Phi_\nu$ are from~\cite[Proposition~4.7]{chen2024simultaneous}. The particular case $|\partial_x\Phi_\nu|\leq 1$ follows from~\cite[Proposition~4.7~(3)]{chen2024simultaneous} and the fact that $\spindist_\circ$ is supported in $[-1,1]$.
The inequality \eqref{e.compare_Phi_nu} then allows us to extend the results to every $\nu,\nu'\in\mcl P_1(\R_+)$. To extend~\eqref{e.barpsi=Phi}, we also use the continuity of $\psi_\circ$.
\end{proof}

We now recall the concavity of $\psi_\circ$ originally due to~\cite{auffinger2015parisi}.

\begin{lemma}[Concavity of $\psi_\circ$]
\label{l.psi is concave}
The function $\psi_\circ:\mcl P_1(\R_+)\to \R$ is concave; that is, for every $\nu, \nu' \in \mcl P_1(\R_+)$ and $\lambda \in [0,1]$, we have $\psi_\circ((1-\lambda)\nu + \lambda \nu') \geq (1-\lambda)\psi_\circ(\nu) + \lambda\psi_\circ(\nu')$.

The restriction of $\psi_\circ$ to $\mcl P_\infty(\R_+)$ is strictly concave; that is, for every distinct $\nu, \nu' \in \mcl P_\infty(\R_+)$ and $\lambda \in (0,1)$, we have $\psi_\circ((1-\lambda)\nu + \lambda \nu') > (1-\lambda)\psi_\circ(\nu) + \lambda\psi_\circ(\nu')$.
\end{lemma}
\begin{proof}
By the Lipschitzness of $\psi_\circ$ in $\nu$ discussed in the previous section (or implied by~\eqref{e.compare_Phi_nu} and~\eqref{e.barpsi=Phi}), it suffices to prove the strict concavity on $\mcl P_\infty(\R_+)$. Fix any distinct $\nu$ and $\nu'\in \mcl P_\infty(\R_+)$. By~\eqref{e.Phi_nu_trivial}, we can fix any $T> \nu^{-1}(1), \nu'^{-1}(1)$ so that $\Phi_\nu$ and $\Phi_{\nu'}$ both satisfy the Parisi PDE~\eqref{e.Phi_nu_eqn} on $[0,T]\times \R$ with the same terminal condition equal to $x\mapsto\log\int\exp(x\cdot\sigma-T\sigma^2)\d \spindist_\circ(\sigma)$.
Then, the concavity follows from~\eqref{e.barpsi=Phi} and the strict convexity of $\nu\mapsto \Phi_{\nu}(0,0)$ on the set of probability measures with support in $[0,T]$, which is available from a straightforward modification of~\cite[Theorem~4~(i)]{auffinger2015parisi}.
\end{proof}

In order to determine the derivative of $\psi_\circ$, we introduce an adjoint equation to the Parisi PDE, for each $\nu \in \mcl P^1(\R_+)$. 
Let $P:\R_+\times \R\to\R$ be the one-dimensional heat kernel, namely
\begin{e}
    P(t,x) = \frac{1}{\sqrt{4\pi t}}e^{-\frac{x^2}{4t}}.
\end{e}
Since $\partial_x\Phi_\nu$ is bounded due to Lemma~\ref{l.ext_Parisi_PDE}, we can use a Picard fixed-point argument in a suitable function space to build a function $u_\nu:\R_+\times\R\to\R$ such that, for every $(t,x) \in\R_+\times \R$,
\begin{align*}
    u_\nu(t,x) = P(t,x) + \int_0^t\int_\R \partial_x P(t-s,x-y)u_\nu(s,y)\Ll(2\nu(s)\partial_x\Phi_\nu(s,y)\Rr)\d y \d s.
\end{align*}
In other words, the function $u_\nu$ is a mild solution to
\begin{align}\label{e.u_nu=eqn}
\begin{cases}
    \partial_t u_\nu = \partial_x^2 u _\nu - 2 \nu(t) \partial_x\Ll(u_\nu \partial_x \Phi_\nu\Rr),\qquad&\text{on $\R_+\times \R$},
    \\
    u_\nu(0,\cdot) = \delta_0,\qquad&\text{on $\R$},
\end{cases}
\end{align}
where $\delta_0$ is the Dirac mass at zero. One can verify that $u_\nu$ is also a weak solution to this equation. 
From the mild solution formulation, one can check that $u_\nu$ and its derivatives in $x$ have exponential decay in $x$, which allows for integration by parts. We also recognize that~\eqref{e.u_nu=eqn} is a Kolmogorov equation and thus $u_\nu(t,\cdot)$ can be interpreted as the probability density function of a stochastic process at time $t$. Therefore, $u_\nu$ also satisfies
\begin{align}\label{e.u_v>0,intu_v=1}
    u_\nu\geq0\qquad\text{and}\qquad \int u_\nu(t, x)\d x=1,\quad\forall t\in\R_+.
\end{align}
We use the properties of $u_\nu$ to prove the next two lemmas.

\begin{lemma}[Derivative of $\psi_\circ$]
\label{l.d/dpsi=}
Let $\nu,\nu'\in \mcl P_1(\R_+)$ and let $\nu_\eps = (1-\eps) \nu + \eps \nu'$ for $\eps\in [0,1]$. We have
\begin{align}\label{e.l.d/dpsi=}
    \frac{\d}{\d\eps}\psi_\circ(\nu_\eps) \big|_{\eps=0} = \int_0^\infty \chi(t) \d (\nu'-\nu)
\end{align}
where
\begin{align}\label{e.chi=}
    \chi(t) = \int_0^t \int_\R \Ll(\partial_x \Phi_\nu(\tau,x)\Rr)^2 u_\nu(\tau,x)\d x\d \tau,\quad\forall t\in\R_+.
\end{align}
\end{lemma}

Notice that the right-hand side in~\eqref{e.l.d/dpsi=} is well-defined and finite since the first moments of $\nu'$ and $\nu$ exist and $|\chi(t)|\leq t $ due to $\int u_\nu(\tau,x)\d x =1$ and the boundedness of $|\partial_x \Phi_\nu|\leq 1$ given by Lemma~\ref{l.ext_Parisi_PDE}.

\begin{proof}
For simplicity, we write $\Phi_\eps =\Phi_{\nu_\eps}$ and $u= u_\nu$. For any $T>0$, we set $C_\eps(T)= \int \Phi_\eps(T,x)u(T,x)\d x$. Using $u(0,\scdot)=\delta_0$ and integrating by parts (or more precisely appealing to the weak formulation of \eqref{e.u_nu=eqn}), we have
\begin{align*}
    \Phi_\eps(0,0) = \int_\R \Phi_\eps(0,x) u(0,x)\d x = C_\eps(T) - \int_0^T\int_\R \frac{\d}{\d t}\Ll(\Phi_\eps u\Rr)\d x \d t.
\end{align*}
Using the equations satisfied by $\Phi_\eps$ and $u$ in~\eqref{e.Phi_nu_eqn} and~\eqref{e.u_nu=eqn}, we have
\begin{align*}
    \frac{\d}{\d t}\Ll(\Phi_\eps u\Rr) = - u \partial_x^2\Phi_\eps - \nu_\eps u(\partial_x\Phi_\eps)^2 + \Phi_\eps \partial^2_x u - 2 \nu  \Phi_\eps\partial_x(\partial_x\Phi_0 u).
\end{align*}
Integrating this in $x$ and using integration by parts, we can cancel the second-order terms and get
\begin{align*}
    \Phi_\eps(0,0) = C_\eps(T) + \int_0^T\int_\R \Ll(\nu_\eps a_\eps^2-2 \nu a_\eps a_0 \Rr)u 
\end{align*}
where we used the shorthand $a_\eps = \partial_x \Phi_\eps$. Then, we have
\begin{align*}
    \Phi_\eps(0,0)- \Phi_0(0,0) -\Ll(C_\eps(T)-C_0(T)\Rr) = \int_0^T\int_\R \big(\underbrace{\nu_\eps a_\eps^2-2 \nu a_\eps a_0 + \nu a_0^2}_{\blacksquare}\big) u 
\end{align*}
We can rearrange $\blacksquare = (\nu_\eps - \nu)a_0^2 -(\nu_\eps -\nu)(a_\eps^2-a_0^2)+ \nu_0(a_\eps-a_0)^2$ and we argue that the last two terms are of order $O(\eps^2)$ uniformly in $\eps, t,x$. Indeed, this follows from the facts that $\nu_\eps -\nu = O(\eps)$ by definition, that $a_\eps - a_0 = O(|\nu_\eps-\nu|_{L^1})= O(\eps)$ due to~\eqref{e.compare_Phi_nu}, and that $a_\eps, a_0 = O(1)$ by Lemma~\ref{l.ext_Parisi_PDE}. Using this and $\nu_\eps -\nu=\eps(\nu'-\nu)$, we get $\blacksquare = \eps(\nu'-\nu) a_0^2 + O(\eps^2)$ uniformly in $\eps, t,x$. 
Inserting this to the above display we get
\begin{align*}
    \Phi_\eps(0,0)- \Phi_0(0,0) -\Ll(C_\eps(T)-C_0(T)\Rr) = \eps\int_0^T\int_\R (\nu'-\nu)a_0^2 u  + O(\eps^2)
    \\
    \stackrel{\eqref{e.chi=}}{=}\eps\int_0^T (\nu'-\nu)\dot \chi  + O(\eps^2) 
    \\
    \stackrel{\text{(IBP)}}{=} - \eps \int_0^T \chi \d (\nu'-\nu) + \eps \Ll(\nu'(T)-\nu(T)\Rr)\chi(T)++ O(\eps^2).
\end{align*}
In the last step, we used the fact that $\chi(0)=0$. Using this relation and $\psi_\circ(\nu_\eps) = -\Phi_\nu(0,0)$ (due to~\eqref{e.barpsi=Phi}), we can obtain~\eqref{e.l.d/dpsi=} by first sending $T\to\infty$ and then $\eps\to0$ provided that we can show that, for any fixed $\eps$,
\begin{align}
    \lim_{T\to\infty} \Ll(C_\eps(T)-C_0(T)\Rr) & =0, \label{e.lim_T|C(T)-C(T)|=0}
    \\
    \lim_{T\to\infty} \Ll(\nu'(T)-\nu(T)\Rr)\chi(T) & =0 . \label{e.lim_T|nu'(T)-nu(T)|chi(T)=0}
\end{align}
Hence, it remains to verify them.
Using the definition of $C_\eps(T)$, the estimate in~\eqref{e.compare_Phi_nu} (with $k=0$), and $\int u(T,x)\d x =1$, we have
\begin{align*}
    \Ll|C_\eps(T)-C_0(T)\Rr|\leq c_0\eps \int_T^\infty|\nu(\tau)-\nu'(\tau)|\d \tau
\end{align*}
for some constant $c_0$.
Since $\int_0^\infty|\nu(\tau)-\nu'(\tau)|\d\tau$ is finite as it is equal to $\msf W_1(\nu,\nu')$, the right-hand side vanishes as $T\to\infty$, which verifies~\eqref{e.lim_T|C(T)-C(T)|=0}. To see~\eqref{e.lim_T|nu'(T)-nu(T)|chi(T)=0}, we start by recalling that $|\chi(t)|\leq t$ as explained below Lemma~\ref{l.d/dpsi=}. Hence, we have
\begin{align*}
    \Ll|\Ll(\nu'(T)-\nu(T)\Rr)\chi(T)\Rr|\leq T(1-\nu'(T)) + T(1-\nu(T)).
\end{align*}
Let $X$ be a random variable with law $\nu$. We have $T(1-\nu(T)) = T\P(X>T)\leq \E [X\mathds{1}_{X>T}]$. Since $\nu$ has finite first moment, the dominated convergence theorem implies $\lim_{T\to\infty} T(1-\nu'(T)) =0$. The other term can be treated similarly and thus we obtain~\eqref{e.lim_T|nu'(T)-nu(T)|chi(T)=0}. As explained previously, now that~\eqref{e.lim_T|C(T)-C(T)|=0} and~\eqref{e.lim_T|nu'(T)-nu(T)|chi(T)=0} are verified, the proof is complete.
\end{proof}

\begin{lemma}[Properties of $\chi$]\label{l.chi_properties}
Let $\nu\in\mcl P_1(\R_+)$ and let $\chi:\R_+\to\R$ be given as in~\eqref{e.chi=}. Then, $\chi(0)=0$ and $\chi$ is increasing, $1$-Lipschitz, and convex.
\end{lemma}

\begin{proof}
Due to $\dot\chi\geq 0$ evidently from \eqref{e.chi=}, $\chi$ is increasing. By Lemma~\ref{l.ext_Parisi_PDE}, $|\partial_x \Phi_\nu|\leq 1$ everywhere. Using this and~\eqref{e.u_v>0,intu_v=1}, we have $|\dot\chi|\leq 1$ and thus $\chi$ is $1$-Lipschitz. To show that $\chi$ is convex, we compute $\ddot \chi$. For brevity of notation, we write $\Phi=\Phi_\nu$, $u=u_\nu$, and $\int = \int_\R \d x$. Using the equations satisfied by $\Phi$ and $u$ as in~\eqref{e.Phi_nu_eqn} and~\eqref{e.u_nu=eqn} and integration by parts (IBP), we can compute
\begin{align*}
    \ddot \chi 
    & = \frac{\d}{\d t}\int (\partial_x\Phi)^2 u  = \int \Ll(2\partial_{t} \partial_{x}\Phi(u \partial_{x}\Phi)  + (\partial_{x}\Phi_x)^2 \partial_{t}u\Rr)
    \\
    & \stackrel{\text{(IBP)}}{=} \int \Ll(-2\partial_{t}\Phi \partial_{x}(u\partial_{x}\Phi ) + (\partial_{x}\Phi)^2 \partial_{t}u \Rr)
    \\
    & \stackrel{\eqref{e.Phi_nu_eqn}\eqref{e.u_nu=eqn}}{=} \int \Ll(2\Ll(\partial_{x}^2\Phi+\nu(\partial_{x}\Phi)^2\Rr)\partial_{x}(u\partial_{x}\Phi ) + (\partial_{x}\Phi)^2\Ll(\partial_{x}^2u - 2 \nu\partial_{x}(u\partial_{x}\Phi )\Rr)\Rr)
    \\
    & = \int \Ll(2\partial_{x}^2\Phi \partial_{x}(u\partial_{x}\Phi ) + (\partial_{x}\Phi)^2 \partial_{x}^2u\Rr)
     \stackrel{\text{(IBP)}}{=} \int 2(\partial_{x}^2\Phi)^2 u \stackrel{\eqref{e.u_v>0,intu_v=1}}{\geq} 0.
\end{align*}
Therefore, $\chi$ is convex, completing the proof.
\end{proof}

\subsection{Properties of \texorpdfstring{$\psi$}{psi}}

Using the decomposition in~\eqref{e.decomp.psi}, we obtain the Lipschitzness, concavity, and differentiability of $\psi$ as stated below. 
For each $p\geq 1$, let $\msf W_p$ be the Wasserstein $L^p$ metric over probability measures on a Euclidean space that will be clear from the context.

Recall that for any $\nu\in\mcl P(\R^D_+)$ and $\s\in\sS$, we denote by $\nu_\s\in\mcl P(\R_+)$ its $\s$-th marginal.
Due to~\eqref{e.F_L^1} and~\eqref{e.limF_N(0,mu)=psi(mu)}, we have, for all $\nu,\nu'\in \mcl P_1(\R_+^D)$,
\begin{align}\label{e.Lip_psi}
    \Ll|\psi(\nu)-\psi(\nu')\Rr|\leq \msf W_1(\nu,\nu').
\end{align}

\begin{proposition}[Concavity of $\psi$] \label{p.almost_strict_concavity}
The function $\psi$ is concave on $\mcl P_1(\R_+^D)$. Furthermore, for every $\nu,\nu' \in \mcl P_1(\R_+^D)$, if there exists $\s\in\sS$ such that $\nu_\s \neq \nu'_\s$ and $\nu_\s,\nu'_\s \in \mcl P_\infty(\R_+)$, then we have, for every $\lambda \in (0,1)$,
    \begin{equation}
       \psi\left(\lambda\nu'+(1-\lambda)\nu\right) > \lambda\psi(\nu')+ (1-\lambda)\psi(\nu).
    \end{equation}
\end{proposition}

\begin{proof}
This follows from the decomposition in~\eqref{e.decomp.psi} and Lemma~\ref{l.psi is concave} stating that each $\psi_d$ is concave on $\mcl P_1(\R_+)$ and strictly concave on $\mcl P_\infty(\R_+)$.
\end{proof}

Note that the function $\psi$ is not strictly concave on $\mcl P_\infty(\R_+^D)$ since it is possible for $\nu,\nu' \in \mcl P_\infty(\R_+^D)$ to satisfy $\nu \neq \nu'$ while $\nu_d = \nu'_d$ for every $d \in \{1,\dots,D\}$. However, the following is true.

\begin{lemma}[Monotone measures characterized by marginals]\label{l.monotone_marginal}
Let $\nu,\nu' \in \mcl P^\upa(\R_+^D)$. Then, $\nu = \nu'$ if and only if $\nu_\s = \nu'_\s$ for all $\s \in \sS$.
\end{lemma}

\begin{proof}
By the definition of monotone measures (see the paragraph below~\eqref{e.def.FN.delta0}), there are two increasing maps $\sq,\sq':[0,1]\to \R^D_+$ such that $\nu = \mathrm{Law}(\sq(U))$ and $\nu'=\mathrm{Law}(\sq(U))$ where $U$ is a uniform random variable in $[0,1]$. Then, $\nu_\s=\nu'_\s$ is equivalent to $\sq_\s=\sq'_\s$ a.e.\ on $[0,1)$. The desired result follows from the simple observation that $\sq=\sq'$ if and only if $\sq_\s=\sq'_\s$ for all $\s\in\sS$.
\end{proof}

The set $\mcl P_\infty^\upa(\R_+^D)$ is not convex when $D > 1$, so it does not make sense to talk about concavity or strict concavity of $\psi$ on $\mcl P_\infty^\upa(\R_+^D)$. Yet we have the following result, which is an immediate corollary of Proposition~\ref{p.almost_strict_concavity} and Lemma~\ref{l.monotone_marginal}.

\begin{corollary}[``Strict concavity'' of $\psi$ on $\mcl P_\infty^\upa(\R_+^D)$]\label{c.strict_concavity}
For any two distinct $\nu,\nu' \in \mcl P^\upa_\infty(\R_+^D)$ and $\lambda\in(0,1)$, we have
\begin{equation}
       \psi\left(\lambda\nu'+(1-\lambda)\nu\right) > \lambda\psi(\nu')+ (1-\lambda)\psi(\nu).
\end{equation}
\end{corollary}

We define
\begin{align}\label{e.frakX=}
    \mathfrak X = \Ll\{\chi : \R_+^D \to \R\bigg|\begin{array}{l}\text{$\chi(x) = \sum_{d = 1}^D \lambda_{\infty,d}\chi_d(x_d)$, $\forall x\in\R_+^D$, where $\chi_d(0)= 0$, } \\ \text{ $\chi_d$ is increasing, $1$-Lipschitz, and convex on $\R_+$} \end{array} \Rr\}.
\end{align}

\begin{proposition}[Differentiability of $\psi$] 
For every $\nu\in \mcl P_1(\R_+^D)$, there is $\chi\in \mathfrak X$ such that, for every $\nu'\in\mcl P_1(\R_+^D)$, we have
\begin{align}
    \frac{\d}{\d\eps}\psi\Ll((1-\eps)\nu+\eps\nu'\Rr)\Big|_{\eps=0} =  \int_{\R^\D_+} \chi \d \Ll(\nu'-\nu\Rr).
\end{align}
\label{p.diff of psi}
\end{proposition}
\begin{proof}
This follows from the decomposition in~\eqref{e.decomp.psi} and Lemma~\ref{l.d/dpsi=} for the differentiability of each $\psi_d$, with
Lemma~\ref{l.chi_properties} ensuring the desired properties of each of the~$\chi_d$ in the decomposition of $\chi$.
\end{proof}

%
%
%
%
%
%
\section{Models with two species} \label{s.2d}

In this section, we present relatively simple proofs of Theorems \ref{t.convexoptimization} and Theorem~\ref{t.hopflike} in the special case of models with two species ($D = 2$) with no external field ($\nabla \xi(0) = 0$) and $\mu = \delta_0$. The proofs we give later of the general versions of these theorems do not make any reference to the results of this section, so one can skip this entire section if so enclined. Nevertheless, we find it interesting to also present this comparatively simpler argument, and also to explain why it does not generalize to models with more than two species.

Given $\mu \in \mcl P(\R^2_+)$, recall that we write $\mu_1$ and $\mu_2$ to denote its marginals. Letting $F_1$ and $F_2$ be the cumulative distribution functions of $\mu_1$ and $\mu_2$ respectively, and $U$ be a uniform random variable on $[0,1]$, we set $\mu^\upa = \mathrm{Law}(F_1^{-1}(U), F_2^{-1}(U))$. This defines a mapping from $\mcl P(\R^2_+)$ to $\mcl P^\upa(\R^2_+)$, and we observe that $\mu = \mu^\upa$ if $\mu \in \mcl P^\upa(\R^2_+)$. The fundamental idea is the following observation. 
\begin{proposition} \label{p.xi^* transport inequality}
Suppose that $D = 2$ and that $\nabla \xi(0) = 0$. For every $\mu \in \mcl P(\R^2_+)$, we have 
    \begin{equation} \label{e.xi^* transport inequality}
        \int \xi^* \d \mu^\upa \leq \int \xi^* \d \mu.
    \end{equation}
\end{proposition}
As will be explained in more details shortly, Proposition~\ref{p.xi^* transport inequality} allows us to rewrite the supremum in \eqref{e.parisi.preferred} as
\begin{e}  
\begin{aligned}
\sup_{\nu \in \mcl P^\upa_\infty(\R_+^2)} \Ll\{ \psi(\nu) - \int (t\xi)^* \, \d \nu \Rr\}  & = \sup_{\nu \in \mcl P_\infty(\R_+^2)} \Ll\{ \psi(\nu) - \int (t\xi)^* \, \d \nu^\upa \Rr\}  \\
& = \sup_{\nu \in \mcl P_\infty(\R_+^2)} \Ll\{ \psi(\nu) - \int (t\xi)^* \, \d \nu \Rr\} ,
\end{aligned}
\end{e}
which is key to establishing Theorems~\ref{t.convexoptimization} and \ref{t.hopflike} in this particular case. The key ingredient of the proof of Proposition~\ref{p.xi^* transport inequality} is the following.
\begin{proposition} \label{p.xi^* inequality}
Suppose that $D = 2$ and that $\nabla \xi(0) = 0$. For all real numbers $a \leq a'$ and $b \leq b'$, we have 
    \begin{equation} \label{e. xi^* inequality}
        \xi^*(a,b) + \xi^*(a',b') \leq \xi^*(a',b) + \xi^*(a,b').
    \end{equation}
\end{proposition}   
The condition \eqref{e. xi^* inequality} is an integrated version of the claim that the mixed second derivatives of $\xi^*$ are non-positive, i.e.\ $\partial_a \partial_b \xi^*(a,b) \le 0$. Heuristically, the mappings $\nabla \xi$ and $\nabla \xi^*$ are inverse of one another, and thus the Hessian of $\xi^*$ at $(a,b)$ should be the inverse of the matrix $A = \nabla \xi^2 (\nabla \xi^*(a,b))$. Moreover, the assumption that $\xi$ must allow for the existence of a Gaussian random field satisfying \eqref{e.def H_N} imposes some constraints on $\xi$ \cite[Proposition~6.6]{mourrat2020free}. In particular, the Hessian of $\xi$  must be a symmetric positive semi-definite matrix with nonnegative coefficients, and thus we should have
\begin{e*}
   \partial^2_{ab} \xi^*(a,b) = -\det(A)^{-1} A_{12} \leq 0. 
\end{e*}
Relying on this heuristic, we give a rigorous proof of Proposition~\ref{p.xi^* inequality} in Appendix~\ref{s.xi^* ineq}.
\begin{proof}[Proof of Proposition~\ref{p.xi^* transport inequality}]
This follows from Proposition \ref{p.xi^* inequality} and \cite[Proposition~2.5]{mourrat2020free} applied to $c(a,b) = \xi^*(a,b)$. (We point out that the sentence ``Examples of functions satisfying the condition (2.16) include any convex function of $x - y$.'' in \cite[Remark~2.5]{mourrat2020free} is actually false, as explained below some convex functions fail to satisfy \cite[(2.16)]{mourrat2020free}.)
\end{proof}
\begin{remark} \label{r. no higher generalization}
We now give an example with $D = 3$ of an admissible function $\xi : \R^3\to \R$ for which Proposition~\ref{p.xi^* transport inequality} is invalid. Consider $\xi(x) = \frac{x \cdot Ax}{2}$ where
    \begin{equation*}
        A = 
        \begin{pmatrix}
            6 & 2 & 1 \\ 
            2 & 5 & 3 \\
            1 & 3 & 4
        \end{pmatrix}.
    \end{equation*}
    The function $\xi : \R^3 \to \R$ thus defined satisfies 
    \begin{e*}
        \E H_N(\sigma) H_N(\tau) = \xi \left( \frac{\sigma_1 \cdot \tau_1}{N},\frac{\sigma_2 \cdot \tau_2}{N},\frac{\sigma_3 \cdot \tau_3}{N}\right)
    \end{e*}
provided that we set
    \begin{e*}
        H_N(\sigma) = \frac{1}{\sqrt{N}} \sum_{i,j=1}^N \sum_{d,d'=1}^3 (A_{dd'})^{1/2} J_{i,j}^{d,d'}  \sigma_{di} \sigma_{d'j},
    \end{e*}
    where the $J_{i,j}^{d,d'}$ are independent standard Gaussian random variables. We observe that $A$ is positive definite, as can be checked using Sylvester's criterion on the positivity of the leading principal minors, so the function $\xi$ is convex on $\R^3$. We have 
    \begin{equation*}
        A^{-1} = \frac{1}{57}
            \begin{pmatrix}
            11 & -5 & 1 \\ 
            -5 & 23 & -16 \\
            1 & -16 & 26
            \end{pmatrix}.
    \end{equation*}
    Note that $(A^{-1})_{13} > 0$. Let $\bar V = \nabla \xi(\R^3_+)$, we have that $\nabla \xi(x) = A x$ so  $\bar V = \{ Ax \big| \, x \in \R^3_+ \}$. Using that $\xi^*(\nabla \xi(x)) = x \cdot \nabla \xi(x) - \xi(x)$ (proved below in \eqref{e.xi^*(nabla_xi)=} in any dimension), we see that for every $y \in \bar V$,
    \begin{equation*}
        \xi^*(y) = \frac{y \cdot A^{-1}y}{2}.
    \end{equation*}
    In particular $\xi^*$ is differentiable on the interior of $\bar V$ and $\nabla^2 \xi^*(y) = A^{-1}$. Now let $(e_1,e_2,e_3)$ denote the canonical basis of $\R^3$ and let us fix $y_0$ in the interior of $\bar V$. For $\varepsilon > 0$ small enough we have that the points $y_0 + \varepsilon e_1$, $y_0 + \varepsilon e_3$ and $y_0 + \varepsilon e_1 +  \varepsilon e_3$ all belong to the interior of $\bar V$. Let $\mu$ be the uniform measure on $\{y_0,y_0+ \varepsilon e_1,y_0+\varepsilon e_3\}$, we have 
    \begin{equation*}
        \mu^\upa = \frac{1}{3} \left( 2\delta_{y_0} + \delta_{y_0+\varepsilon e_1+ \varepsilon e_3}\right).
    \end{equation*}
    In particular,
    \begin{align*}
        &\int \xi^* \d \mu^\upa - \int \xi^* \d \mu \\
        &= \frac{1}{3} \left( \xi^*(y_0) + \xi^*(y_0 +\varepsilon e_1 +\varepsilon e_3) - \xi^*(y_0 +\varepsilon e_1) - \xi^*(y_0 +\varepsilon e_1) \right) 
        \\
        &= \frac{1}{3} \int_0^1  \int_0^1 \partial_{13} \xi^*(y_0 + t\varepsilon e_1 + s\varepsilon e_3) \d t \d s.
    \end{align*}
    Since $\partial_{13} \xi^* > 0$ on the interior of $V$, we obtain 
    \begin{equation*} 
        \int \xi^* \d \mu^\upa > \int \xi^* \d \mu,
    \end{equation*}
    as announced. \qed
\end{remark}


We now prove Theorem~\ref{t.convexoptimization} when $D = 2$, $\nabla \xi(0) = 0$ and $\mu = \delta_0$ by relying on Proposition~\ref{p.xi^* transport inequality}. As a consequence, we will be able to deduce that Theorem~\ref{t.hopflike} also holds at $D = 2$ and $\mu = \delta_0$.

Recall from \eqref{e.F_N(0,mu)=} that every $\nu \in \mcl P^\upa_1(\R^2_+)$, we have
\begin{equation}
    \psi(\nu) = \lambda_{\infty,1}\psi_1(\nu_1) + \lambda_{\infty,2}\psi_2(\nu_2),
\end{equation}
where $\psi_d$ is the Parisi functional associated to $\spindist_\s$. The right-hand side of the previous display is well defined even when $\nu$ is not monotone, in particular this gives an extension of the functional $\psi$ to $\mcl P_1(\R^2_+)$.

\begin{proof}[Proof of Theorem~\ref{t.convexoptimization} when $D = 2$, $\nabla \xi(0) = 0$ and $\mu = \delta_0$]
    Let $\nu \in \mcl P_\infty(\R^2_+)$. Since $\nu$ and $\nu^\upa$ have the same marginals, we have $\psi(\nu) = \psi(\nu^\upa)$. Thus according to Proposition~\ref{p.xi^* transport inequality} we have 
    \begin{equation*}
        \psi(\nu) - \int \xi^* \d \nu \leq \psi(\nu^\upa) - \int \xi^* d\nu^\upa.
    \end{equation*}
    This implies 
    \begin{equation*}
        \sup_{\nu \in \mcl P_\infty(\R^2_+)} \left\{ \psi(\nu) - \int \xi^* d\nu \right\} \leq \sup_{\nu \in \mcl P^\upa_\infty(\R^2_+)} \left\{ \psi(\nu) - \int \xi^* d\nu \right\}.
    \end{equation*}
    Since the other bound is trivial, 
    \begin{equation*}
        \sup_{\nu \in \mcl P_\infty(\R^2_+)} \left\{ \psi(\nu) - \int \xi^* d\nu \right\} = \sup_{\nu \in \mcl P^\upa_\infty(\R^2_+)} \left\{ \psi(\nu) - \int \xi^* d\nu \right\}.
    \end{equation*}
    We use Theorem~\ref{t.parisi.basic} to conclude that \eqref{e.convexoptimization} holds at $\mu = \de_0$.

    From Lemma~\ref{l.psi is concave}, we have that $\psi_1$ and $\psi_2$ are (strictly) concave on $\mcl P_\infty(\R_+)$ and the map $\nu \mapsto \int \xi^* \d \nu$ is affine, thus the function inside the supremum in \eqref{e.convexoptimization} is concave in $\nu$.

    Finally, by contradiction, assume that there are two distinct monotone probability measures $\nu$ and $\nu'$ that reach the supremum in \eqref{e.convexoptimization}. Consider $\nu'' = \frac{\nu+\nu'}{2} \in \mcl P_\infty(\R^2_+)$ (note that $\nu''$ may not be monotone). Since $\nu \neq \nu'$ and $\nu,\nu' \in \mcl P^\upa(\R^2_+)$, we must have $\nu_1 \neq \nu'_1$ or $\nu_2 \neq \nu'_2$. Since $\psi_1$ and $\psi_2$ are strictly concave on $\mcl P_\infty(\R_+)$, this yields
    \begin{align*}
        \psi(\nu'') &= \psi_1 \left( \frac{\nu_1 +\nu'_1}{2}\right) + \psi_2 \left( \frac{\nu_2 +\nu'_2}{2}\right) \\
                  &>  \frac{\psi_1 (\nu_1) +\psi_1 (\nu'_1)}{2} + \frac{\psi_2 (\nu_2) +\psi_2 (\nu'_2)}{2} \\
                  &= \frac{\psi (\nu) +\psi (\nu')}{2}.
    \end{align*}
    By \eqref{e.convexoptimization} at $\mu = \de_0$, we thus have 
    \begin{align*}
        \lim_{N \to +\infty} \bar F_N(t,\delta_0) &= \frac{1}{2} \left( \psi(\nu) - \int \xi^* d\nu + \psi(\nu') - \int \xi^* d\nu' \right) \\
        &< \psi(\nu'') - \int \xi^* \d\nu'' \\
        &\leq \sup_{\rho \in \mcl P_\infty(\R^2_+)} \left\{ \psi(\rho) - \int \xi^* \d\rho' \right\} \\
        &= \lim_{N \to +\infty} \bar F_N(t,\delta_0).
    \end{align*}
    This is a contradiction. Therefore there is exactly one probability measure $\nu$ that reaches the supremum in \eqref{e.convexoptimization}. 
\end{proof}
Recall the definition of $\mfk X$ from \eqref{e.frakX=}. For every $\chi \in \mfk X$, we  define $S_t \chi$ according to \eqref{e.def.St2}.

\begin{proof}[Proof of Theorem~\ref{t.hopflike} when $D = 2$, $\nabla \xi(0) = 0$, and $\mu = \delta_0$]
    We assume, without loss of generality, that $t =1$. According to \cite[Lemma~4.1]{issa2024hopflike}, we have that
    \begin{equation} \label{e.psi_d_bidual}
        \psi_d(\nu_d) = \inf_{\chi_d} \left\{ \int \chi_d \d \nu_d - (\psi_d)_*(\chi_d) \right\},
    \end{equation}
    where the infimum is taken over all $\chi_d : \R_+ \to \R$ that are $1$-Lipschitz, increasing, convex functions that vanish at the origin and where we have defined
    \begin{e*}
        (\psi_d)_*(\chi_d) = \inf_{\nu_d \in \mcl P_1(\R_+)} \left\{ \int \chi_d \d \nu_d - \psi_d(\nu_d) \right\}.
    \end{e*} 
    It then follows that
    \begin{align*}
        \psi(\nu) &= \psi_1(\nu_1) + \psi_2(\nu_2) \\
                  &= \inf_{\chi_1} \left\{ \int \chi_1 \d \nu_1 - (\psi_1)_*(\chi) \right\} + \inf_{\chi_2} \left\{ \int \chi_2 \d \nu_2 - (\psi_2)_*(\chi) \right\} \\
                  &= \inf_{\chi \in \mathfrak X} \left\{ \int \chi \d \nu - \psi_*(\chi) \right\}.
    \end{align*}
    Plugging this into the formula of Theorem~\ref{t.convexoptimization} at $\mu = \delta_0$, we obtain 
    \begin{equation*}
        \lim_{N \to +\infty} \bar F_N(t,\delta_0) = \sup_{\mu \in \mcl P_\infty(\R^2_+)} \inf_{\chi \in \mathfrak X} \left\{ \int \chi \d \nu - \psi_*(\chi) + \int \xi^* \d \nu \right\}.
    \end{equation*}
    The sets $\mcl P_\infty(\R^2_+)$ and $\mathfrak X$ are both convex. In addition, $\mathfrak X$ is compact with respect to convergence in the local uniform topology. Furthermore, for every $\nu \in \mcl P_\infty(\R^2_+)$, the map $\chi \mapsto \int (\chi-\xi^*) \d \nu - \psi_*(\chi)$ is convex on $\mathfrak X$ and for every $\chi \in \mathfrak X$, the map $\nu \mapsto \int (\chi-\xi^*) \d \nu - \psi_*(\chi)$ is concave on $\mcl P_\infty(\R^2_+)$. According to Sion's min-max theorem \cite{sion1958minmax}, we obtain 
    \begin{equation*}
        \lim_{N \to +\infty} \bar F_N(t,\delta_0) = \inf_{\chi \in \mathfrak X} \sup_{\nu \in \mcl P_\infty(\R^2_+)} \left\{ \int \chi \d \nu - \psi_*(\chi) + \int \xi^* \d \nu \right\}.
    \end{equation*}
    According to \eqref{e.def.St2}, we have 
    \begin{equation*}
        \sup_{\nu \in \mcl P_\infty(\R^2_+)} \left\{ \int (\chi-\xi^*) \d \nu \right\} = \sup_{x \in \R^2_+} \left\{ \chi(x) - \xi^*(x) \right\}  = S_1 \chi(0).
    \end{equation*}
    we obtain 
    \begin{equation*}
        \lim_{N \to +\infty} \bar F_N(t,\delta_0) = \inf_{\chi \in \mathfrak X}  \left\{S_1 \chi(0) - \psi_*(\chi) \right\},
    \end{equation*}
    as desired.
\end{proof}

%
%
%
%
%
%
\section{Stability of \texorpdfstring{$\R^D_+$}{R\_+\^D}-convexity under HJ semigroup}
\label{s.hjcones}

We now return to the general setting where $D$ and $\nabla \xi(0)$ are arbitrary, as is the case everywhere except in the previous section. For every $\chi : \R^D_+ \to \R$, $t \ge 0$ and $x \in \R^D_+$, we define
\begin{equation}  
\label{e.real.def.Stchi}
S_t \chi(x) = \sup_{y \in \R_+^D} \left\{ \chi(y) -(t\xi)^*(y-x) \right\} ,
\end{equation}
as was announced in \eqref{e.def.St2}. The fact that this quantity is equal to that in~\eqref{e.def.St1} for every $\chi \in \mfk X$ will be a consequence of Lemma~\ref{l.St1=St2} proved below. 
The first result of this section is an analysis of the optimizers in \eqref{e.real.def.Stchi}. 

\begin{proposition}\label{p.nablaS_tchi_and_nablachi}
Assume that $\xi^*$ is differentiable on $\R^D_+$. Let $\chi:\R_+^D\to \R$ be a Lipschitz function. Let $t> 0$ and $x, y \in \R_+^D$ be such that 
\begin{equation}
\label{e.l.nablaS_tchi_and_nablachi(0)}
    \chi(y) - S_t\chi(x) -(t\xi)^*(y-x)=0.
\end{equation}
If $x\in (0,+\infty)^D$ and $x$ is a differentiable point of $S_t\chi$, then we have
\begin{align}\label{e.l.nablaS_tchi_and_nablachi(1)}
    \nabla S_t\chi(x) =\nabla(t\xi)^*(y-x).
\end{align}
If $y\in (0,+\infty)^D$ and $y$ is a differentiable point of $\chi$, then we have
\begin{align}\label{e.l.nablaS_tchi_and_nablachi(2)}
    \nabla \chi(y) = \nabla (t\xi)^*(y-x).
\end{align}
In particular, if the two conditions hold at the same time, then we have
\begin{align}\label{e.l.nablaS_tchi_and_nablachi(2)(3)}
    \nabla S_t \chi(x) = \nabla \chi(y).
\end{align}
\end{proposition}
Before showing this, we derive in the next three lemmas basic properties of $\xi$ and $\xi^*$ that will be useful throughout the rest of the paper. We say that a function $\varsigma:\R^D\to \R$ is \emph{superlinear on $\R_+^D$} provided that, for every $M>0$, there is $R>0$ such that
\begin{align}\label{e.def_superlinear}
    \inf_{x\in\R_+^D:\:|x|\geq R}\frac{\varsigma(x)}{|x|}\geq M.
\end{align} 
\begin{lemma} \label{l.xi* continuous and superlinear}
    The function $\xi^*$ is superlinear on $\R_+^D$. In addition, if we assume that $\xi$ is superlinear on $\R_+^D$, then $\xi^*$ is also continuous. The same holds for $(t\xi)^*$ for any $t>0$.
\end{lemma}
\begin{proof}
    The function $\xi$ is continuous so it is locally bounded. For every $r > 0$, let $M_r$ denote the supremum of $|\xi|$ on the ball of radius $r$ centered at $0$. Letting $y \in \R_+^D$ and setting $x = ry/|y|$ in the supremum below, we get that 
    \begin{e*}
        \xi^*(y) = \sup_{x \in \R_+^D} \left\{ x \cdot y- \xi(x)\right\} \geq r|y| - M_r.
    \end{e*}
    Then \eqref{e.def_superlinear} follows from the last display. Let us now further assume that $\xi$ is superlinear. In this case, we have that for every $y \in \R_+^D$,
    \begin{e*}
        \lim_{\substack{|x| \to +\infty \\ x \in \R_+^D}} \left(x \cdot y - \xi(x) \right) = -\infty. 
    \end{e*}
    This implies that $\xi^* < +\infty$ on $\R^D$. Hence, $\xi^*$ is a convex function on $\R^D$ taking only finite values and thus $\xi^*$ is continuous (see e.g.\ \cite[Proposition~2.9]{HJbook}).
\end{proof}
\begin{lemma} \label{l.xi^*.diff}
    If $\xi$ is superlinear on $\R^D_+$ and strictly convex on $\R^D_+$, then $\xi^*$ is differentiable on $\R^D$.
\end{lemma}

\begin{proof}
    Let $M > 0$ be arbitrary and let $R > 0$ be such that \eqref{e.def_superlinear} holds for $\xi$. Let $y \in \R^D$ be such that $|y| \leq M$. For every $|x| \geq R$, we have $x \cdot y - \xi(x) \leq |x|(|y|-M) \leq 0$. Hence,
    \begin{e*}
        \xi^*(y) = \sup_{|x| \leq R} \left\{ x \cdot y - \xi(x) \right\}.
    \end{e*}
    Since $\xi$ is assumed to be strictly convex, the variational problem in the previous display admits exactly one maximizer. It then follows from the envelope theorem \cite[Theorem~2.21]{HJbook} that $\xi^*$ is differentiable on $\R^D$.
\end{proof}

\begin{lemma} \label{l.save_the_day}
    For every $x \in \R^D_+$, we have
    \begin{e} \label{e.xi^*(nabla_xi)=}
        \xi^*(\nabla \xi(x)) = x \cdot \nabla \xi(x) - \xi(x).
    \end{e}
    Furthermore, if $\xi$ is strongly convex on $\R^D_+$, then we have for every $x,y\in \R^D_+$ that
    \begin{align}
        &\nabla \xi^*( \nabla \xi(x)) = x, \label{e.nabla_xi^*(nabla_xi)=}\\ 
        &\nabla \xi(\nabla \xi^*(y)) -y \in \R^D_+, \label{e.KKT1}\\ 
        &\nabla \xi^*(y) \cdot (\nabla \xi(\nabla \xi^*(y)) -y) =0, \label{e.KKT2}\\
        &\xi^*(\nabla \xi(\nabla \xi^*(y))) = \xi^*(y) \label{e.xi^*(nabla_xi(nabla_xi^*))}.
    \end{align}
    
\end{lemma}

\begin{remark} \label{r.bla}
    When $\xi$ is strongly convex on $\R^D_+$, note that thanks to \eqref{e.nabla_xi^*(nabla_xi)=} we have in particular that for every $y \in \nabla \xi(\R^D_+)$,
    \begin{e}
        \nabla \xi(\nabla \xi^*(y)) = y.
    \end{e}
\end{remark}

\begin{proof}
    Let $x \in \R^D_+$, we start by proving \eqref{e.xi^*(nabla_xi)=}. By definition of $\xi^*$, it is clear that the left-hand side is larger than the right-hand side. To prove the other inequality, let $x ' \in \R^D_+$, by convexity of $\xi$ we have 
    \begin{equation*}
        \xi( \lambda x' + (1-\lambda) x) \leq \lambda \xi(x') + (1-\lambda) \xi(x).
    \end{equation*}
    Rearranging and letting $\lambda \to 0$ we get $\nabla \xi(x) \cdot (x'-x) \leq \xi(x') - \xi(x)$, in particular
    \begin{equation*}
        x' \cdot \nabla \xi(x) - \xi(x') \leq x \cdot \nabla \xi(x) - \xi(x).
    \end{equation*}
    Taking the supremum over $x' \in \R^D_+$ in the previous display, we obtain $\xi^*(\nabla \xi(x)) \leq x \cdot \nabla \xi(x) - \xi(x)$, which thus completes the proof of \eqref{e.nabla_xi^*(nabla_xi)=}. 
    
    Since $\xi$ is assumed to be strongly on $\R^D_+$, it is superlinear on $\R^D_+$ and strictly convex on $\R^D_+$, so by Lemma~\ref{l.xi^*.diff}, the function $\xi^*$ is differentiable on $\R^D$. For $x \in \R^D_{++}$ we can differentiate \eqref{e.xi^*(nabla_xi)=} to obtain 
    \begin{e*}
        \nabla^2 \xi(x) \nabla \xi^*(\nabla \xi(x)) = \nabla^2 \xi(x)x.
    \end{e*}
    Since $\xi$ is assumed to be strongly convex on $\R^D_+$, the matrix $\nabla^2 \xi(x)$ is positive definite and the previous display implies \eqref{e.nabla_xi^*(nabla_xi)=} for $x \in \R^D_{++}$, which we can then extend to $x \in \R^D_+$ by continuity. 

    We now turn to proving \eqref{e.KKT1} and \eqref{e.KKT2}. Let $y \in \R^D_{++}$, by definition we have
    \begin{e*} 
        \xi^*(y)= \sup_{x \in \R^D_+} \left\{ x\cdot y - \xi(x) \right\}.
    \end{e*}
    Let $M > 0$ and assume that $|y| \leq M$. Since $\xi$ is assumed to be superlinear, there exists $R > 0$ such that \eqref{e.def_superlinear} holds for $\xi$ and we have that for $|x| > R$,
    \begin{e*}
        x \cdot y - \xi(x) \leq |x|(|y|-M) < 0.
    \end{e*}
    Therefore,
    \begin{e*} 
        \xi^*(y)= \sup_{\substack{x \in \R^D_+ \\ |x| \leq R}} \left\{ x\cdot y - \xi(x) \right\}.
    \end{e*}
    In addition since $\xi$ is assumed to be strongly convex on $\R^D_+$, there is a unique maximizer $x_\text{opt}(y) \in \R^D_+$ in the variational formula of the above display. According to the envelope theorem \cite[Theorem~2.21]{HJbook} we have $\nabla \xi^*(y) = x_\text{opt}(y)$. It follows that for every $x \in \R^D_+$,
    \begin{align*}
        \xi(x_\text{opt}(y)) &= x_\text{opt}(y) \cdot y - \xi^*(y) \\
                             &\le x_\text{opt}(y) \cdot y - (x \cdot y - \xi(x)).
    \end{align*}    
    Rearranging, we obtain 
    \begin{e*}
        \xi(x) \geq \xi(x_\text{opt}(y)) + y \cdot(x- x_\text{opt}(y)).
    \end{e*}
    Chosing $x = x_\text{opt}(y) + t e_d$ for $t > 0$ and letting $t \to 0$ in the previous display, we obtain \eqref{e.KKT1}. Furthermore, if $x_{\text{opt},d}(y) > 0$ for some $d \in [D]$, then we can also take $t < 0$ in the previous construction (provided that $t$ is small enough) and this yields $\partial_{x_d} \xi(x_\text{opt}(y)) = y_d$. Therefore for every $d \in [D]$, we have that the $d$-th coordinate of $\nabla \xi^*(y)$ is $0$ or the $d$-th coordinate of $\nabla \xi(\nabla \xi^*(y))-y$ is $0$, and this yields \eqref{e.KKT2} for $y \in \R^D_{++}$. By continuity, we obtain \eqref{e.KKT1} and \eqref{e.KKT2} as announced.

    Finally, by a similar proof as for \eqref{e.xi^*(nabla_xi)=}, we have for every $y \in \R^D_+$ that
    \begin{e*}
        \xi(\nabla \xi^*(y)) = y \cdot \nabla \xi^*(y) - \xi^*(y).
    \end{e*}
    From this, \eqref{e.xi^*(nabla_xi)=} and \eqref{e.KKT2}, we deduce that
    \begin{align*}
        \xi^*(\nabla \xi(\nabla \xi^*(y))) &= \nabla \xi^*(y) \cdot \nabla \xi(\nabla \xi^*(y)) - \xi(\nabla \xi^*(y)) \\
                                           &= \nabla \xi^*(y) \cdot \nabla \xi(\nabla \xi^*(y)) - \left( y \cdot \nabla \xi^*(y) - \xi^*(y)\right) \\
                                           &= \xi^*(y) + \nabla \xi^*(y) \cdot(\nabla \xi(\nabla \xi^*(y))-y ) \\
                                           &= \xi^*(y),
    \end{align*} 
    and thus, \eqref{e.xi^*(nabla_xi(nabla_xi^*))} is proved.
\end{proof}

In the proof of Proposition~\ref{p.nablaS_tchi_and_nablachi} as well as later on, we make use of the following simple fact: 
\begin{align}\label{e.|y-x|-|(y-x)_+|}
    x,y\in\R_+^D\quad\Longrightarrow\quad |y-x|-|(y-x)_+|\leq |x|.
\end{align}
Indeed, we have
\begin{align*}
    |y-x|-|(y-x)_+| \leq |y-x-(y-x)_+|\leq\Big(\sum_{d=1}^D(y_d-x_d)^2\mathds{1}_{0\leq y_d<x_d}\Big)^\frac{1}{2} \leq |x|.
\end{align*}
We also need the following definition.
\begin{definition}[$\R_+^D$-increasing] \label{d.R^D_+-increasing}
For any $k \in \N$, we say that a function $\chi : \R^D_+ \to \R^k$ is \emph{$\R_+^D$-increasing} when for every $x,y \in \R_+^D$ we have 
\begin{e} \label{e.R^D_+-increasing}
    \chi(x+y)  \geq  \chi(x).
\end{e}
\end{definition}
The inequality \eqref{e.R^D_+-increasing} is interpreted in the sense that $\chi(x+y) - \chi(x) \in \R^k_+$. We may at times also make use of this notion for functions that are defined almost everywhere, in which case we ask that \eqref{e.R^D_+-increasing} be satisfied for almost every $x, y \in \R^D_+$. 

We record the following last preparatory step, which will be used more crucially in the proof of Proposition~\ref{p.modified hopf-lax} below.
For every $y \in \R^D$, we set
\begin{align}\label{e.x_+=}
    y_+ = \Ll(y_d \mathds{1}_{y_d\geq 0}\Rr)_{1\leq d\leq D}.
\end{align}
\begin{lemma}\label{l.H^*(x)=H^*(x_+)}
    Let $\varsigma:\R_+^D\to\R$ be an $\R_+^D$-increasing function, and for every $y \in \R^D$, let 
    \begin{e} \label{e.convex dual R^D_+}
        \varsigma^*(y)= \sup_{x \in \R^D_+} \left\{ x \cdot y -\varsigma(x) \right\}.
    \end{e}
    For every $y \in \R^D$, we have $\varsigma^*(y) =\varsigma^*(y_+)$.
\end{lemma}
\begin{proof}
Let $y \in \R^D$.
Since the supremum in the definition of $\varsigma^*$ in~\eqref{e.convex dual R^D_+} is taken over $\R_+^D$, we have that $\varsigma^*$ is $\R_+^D$-increasing and thus $\varsigma^*(y) \leq\varsigma^*(y_+)$. For each $x\in\R_+^D$, we set $\hat x = (x_d \mathds{1}_{y_d\geq 0})_{1\leq d\leq D}$. We observe that $\hat x \cdot y = x \cdot y_+$. Moreover, since $\varsigma$ is $\R_+^D$-increasing, we have $\varsigma(x)\geq \varsigma(\hat x)$. Hence, for every $x\in\R_+^D$, we have
\begin{align*}
    x \cdot y_+ - \varsigma(x)\leq \hat x \cdot y -\varsigma(\hat x) \leq \varsigma^*(y),
\end{align*}
which implies $\varsigma^*(y_+)\leq \varsigma^*(y)$. 
\end{proof}
\begin{proof}[Proof of Proposition~\ref{p.nablaS_tchi_and_nablachi}]
Assume that $x\in (0,+\infty)^D$ and that $x$ is a differentiable point of $S_t\chi$.
Using that $(t\xi)^*$ is superlinear on $\R_+^D$ (Lemma~\ref{l.xi* continuous and superlinear}), Lemma~\ref{l.H^*(x)=H^*(x_+)}, and~\eqref{e.|y-x|-|(y-x)_+|}, we can verify that, for each fixed $x\in\R_+^D$, $(t\xi)^* (\scdot -x)$ is superlinear on $\R_+^D$.
Hence, for $x'$ sufficiently close to $x$, we can see that the variational formula~\eqref{e.real.def.Stchi} for $S_t\chi(x')$ attains its maximum in a compact set.
By~\eqref{e.l.nablaS_tchi_and_nablachi(0)}, $y$ is a maximizer of the formula for $S_t\chi(x)$.
Hence, we can apply the envelope theorem (e.g.\ see~\cite[Theorem~2.21]{HJbook}) to see that $\nabla_xS_t\chi (x) = \nabla_x \Ll(\chi(y) -(t\xi)^*(y-x)\Rr)$, which gives~\eqref{e.l.nablaS_tchi_and_nablachi(1)}.

Next, assume that $ y\in (0,+\infty)^D$ is a differentiable point of $\chi$.
Since $y$ is a maximizer of the formula~\eqref{e.real.def.Stchi} for $S_t\chi(x)$, we have
\begin{equation*}  
\nabla_y  \Ll(\chi(y) -(t\xi)^*(y-x)\Rr)=0,
\end{equation*}
 which gives~\eqref{e.l.nablaS_tchi_and_nablachi(2)}. Lastly, \eqref{e.l.nablaS_tchi_and_nablachi(2)(3)}  immediately follows from~\eqref{e.l.nablaS_tchi_and_nablachi(1)} and~\eqref{e.l.nablaS_tchi_and_nablachi(2)}.
\end{proof}

We now turn to the main objective of this section. 
As discussed around \eqref{e.kantorovich.potential}, we will be interested in the case when $\chi$ is the derivative of $\psi$ at some measure $\bar \nu$, and the main technical ingredients needed in the proofs of our main results involve certain monotonicity properties of $\nabla \chi$ and $\nabla S_t \chi$. More precisely, we are particularly interested in asserting that $\chi$ and $S_t \chi$ are $\R^D_+$-convex, where the notion of $\R^D_+$-convexity is defined as follows. 

\begin{definition}[$\R_+^D$-convexity] \label{d.R^D_+-convexity}
We say that a function $\chi : \R^D_+ \to \R$ is \emph{$\R_+^D$-convex} when for every $x,y,z \in \R_+^D$ we have 
\begin{e} \label{e.R^D_+-convexity}
    \chi(x+y+z) - \chi(x+z) \geq \chi(x+y) - \chi(x).
\end{e}
\end{definition}
When $\chi : \R^D_+ \to \R$ is Lipschitz, one can check that $\chi$ is $\R^D_+$-convex if and only if $\nabla \chi$ is $\R^D_+$-increasing.

Recalling the definition of $\mfk X$ from \eqref{e.frakX=}, we already know from Proposition~\ref{p.diff of psi} that the derivative of $\psi$ at any point in $\mcl P_1(\R^D_+)$ is a function $\chi$ that belongs to $\mfk X$, and this implies in particular that $\chi$ is indeed $\R^D_+$-convex. The main result of this section is as follows.
\begin{proposition}
 \label{p.R^D_+-convexity}
    For every $\chi \in \mathfrak X$ and $t \geq 0$, the function $S_t \chi$ is $\R_+^D$-convex.
\end{proposition}
In order to prove Proposition~\ref{p.R^D_+-convexity}, we will rely on the fact that the mapping $(t,x) \mapsto S_t \chi(x)$ can be interpreted as the solution to a Hamilton--Jacobi equation. We find it clearer to state some results in a slightly more general setting.
Let $H : \R^D_+ \to \R$ be a locally Lipschitz function and $\phi : \R^D_+ \to \R$ be a Lipschitz function. Thanks to \cite{chen2023viscosity}, we know that when $\phi$ and $H$ are $\R^D_+$-increasing, the Hamilton--Jacobi equation
\begin{e}
    \begin{cases}
        \partial_t u - H(\nabla u) = 0 \on (0,+\infty) \times \R^D_+ \\
        u(0,\cdot) = \phi
    \end{cases}
\end{e}
admits a unique viscosity solution. Moreover, by \cite[Proposition~6.2]{chen2023viscosity}, when $H$ is further assumed to be convex and bounded from below, this unique viscosity solution can be represented as the following version of the Hopf-Lax formula 
\begin{e}
\label{e.usual.hopf.lax}
    u(t,x) = \sup_{y \in \R^D_+} \{ \phi(x+y) - (tH)^*(y) \},
\end{e}
where $(tH)^*(x) = \sup_{p \in \R^D_+} \left\{ x \cdot p - tH(p) \right\}$ denotes the convex dual of $tH$ as defined in \eqref{e.convex dual R^D_+}. We show that, under exactly the same assumptions on $H$ and $\phi$, one can represent the solution using another version of the Hopf-Lax representation.
\begin{proposition} \label{p.modified hopf-lax}
    Let $\phi : \R^D_+ \to \R$ be an $\R^D_+$-increasing and Lipschitz function, and let $H : \R^D_+ \to \R$ be a locally Lipschitz function that is $\R^D_+$-increasing, convex and bounded below. The unique $\R^D_+$-increasing and Lipschitz viscosity solution $u$ of 
    \begin{align}
        \begin{cases}
            \partial_t u - H(\nabla u) = 0 \text{ on } (0,+\infty) \times \R^D_+, \\
            u(0,\cdot) = \phi,
        \end{cases}
    \end{align}
    satisfies, for every $(t,x) \in [0,+\infty) \times \R^D_+$,
    \begin{e} \label{e.modified hopf lax}
        u(t,x) = \sup_{x ' \in \R^D_+} \{ \phi(x') - (tH)^*(x'-x) \}.
    \end{e}
\end{proposition}

\begin{proof}
    We let $RHS$ denote the right-hand side in \eqref{e.modified hopf lax}. It is clear from the standard Hopf-Lax representation in \eqref{e.usual.hopf.lax} that $u \leq RHS$. Let $x' \in \R^D_+$, we let $y \in \R^{k}_+$ denote the vector whose coordinates are positive parts of the coordinates of $x'-x$. We have $(tH)^*(y)  =(tH)^*(x'-x)$ by Lemma~\ref{l.H^*(x)=H^*(x_+)}, and since $(x+y)- x' \in \R^D_+$, by monotonicity of $\phi$ we have $\phi(x') \leq \phi(x+y)$. This yields
    \begin{e*}
       u(t,x) \geq \phi(x+y) -(tH)^*(y)  \geq \phi(x') + (tH)^*(x'-x).
    \end{e*}
    Taking the supremum over $x' \in \R^D_+$ in the previous display, we obtain that $u(t,x) \geq RHS$, and thus the result.
\end{proof}

\begin{remark} \label{r.xi.monotonicity}
    The condition \eqref{e.def H_N} puts some constraint on the function $\xi$. In fact, as proved in \cite[Proposition~6.6]{mourrat2020free}, the coefficients in the power series expansion of $\xi$ must be non-negative. This implies that both $\xi$ and $\nabla\xi$ are $\R^D_+$-increasing on $\R^D_+$ and that $\xi$ is $\R^D_+$-convex on $\R^D_+$. In particular, we can indeed apply Proposition~\ref{p.modified hopf-lax} to $H = \xi_{\big|\R^D_+}$ and $\phi = \chi$ for any $\chi \in \mfk X$. \qed
\end{remark}

A first consequence of this remark and the previous proposition is a simple proof of the equality between \eqref{e.def.St1} and \eqref{e.def.St2}. 
\begin{lemma} \label{l.St1=St2}
    Let $\chi : \R^D_+ \to \R$ be a Lipschitz, $\R^D_+$-increasing and convex function. We have 
    \begin{e} \label{e.St1=St2}
       \sup_{y \in \R^D_+} \left\{ \chi(y) - (t\xi)^*(y-x) \right\} =\sup_{p \in \R^D_+} \{ p \cdot x +t \xi(p) - \chi^*(p) \}.
    \end{e}
\end{lemma}
\begin{proof}
    The Fenchel-Moreau theorem states that a function $h : \R^D \to \R$ is convex and lower semi-continuous if and only it can be written as the supremum of a family of affine functions. In \cite{chen2020fenchel}, a similar result has been proved for functions on $\R^D_+$, more precisely a function $h: \R^D_+ \to \R$ is convex, lower semi-continuous and $\R^D_+$-increasing if and only if it satisfies 
    \begin{e*}
        h(x) = \sup_{p \in \R^D_+} \left\{ x \cdot p - h^*(p)\right\},
    \end{e*}
    where $h^*$ is defined as in \eqref{e.convex dual R^D_+}. From this, we have that 
    \begin{align*}
        \xi(x) = \sup_{p \in \R^D_+} \left\{ x \cdot p - \xi^*(p)\right\}, \\
        \chi(x) = \sup_{p \in \R^D_+} \left\{ x \cdot p - \chi^*(p)\right\},
    \end{align*}
    where again $\xi^*$ and $\chi^*$ are defined as in \eqref{e.convex dual R^D_+}. Thus, we can write
    \begin{e*} 
    \begin{aligned}
        \sup_{p \in \R^D_+} \{ p \cdot x +t \xi(p) & - \chi^*(p) \} \\
                                               &= \sup_{p \in \R^D_+} \sup_{x' \in \R^D_+} \left\{ p \cdot (x+x') - (t\xi)^*(x') - \chi^*(p) \right\} \\
                                               &= \sup_{x' \in \R^D_+} \sup_{p \in \R^D_+}  \left\{ p \cdot (x+x') - (t\xi)^*(x') - \chi^*(p) \right\} \\
                                               &= \sup_{x' \in \R^D_+}  \left\{\chi(x+x') - (t\xi)^*(x') \right\}.
    \end{aligned}    
    \end{e*}
    The last line is then equal to the left side of \eqref{e.St1=St2} thanks to Proposition~\ref{p.modified hopf-lax} and \eqref{e.usual.hopf.lax}. 
\end{proof}

Another useful consequence of Proposition~\ref{p.modified hopf-lax} is the following result.

\begin{lemma}[Preservation of Lipschitzness by $S_t$]\label{l.S_tLip}
Let $\chi:\R_+^D\to\R$ be $\R_+^D$-increasing and Lipschitz. Let $(S_t)_{t\in\R_+}$ be the semigroup introduced in~\eqref{e.real.def.Stchi} and set $C=\sup_{z\in\R_+^D:\: |z|\leq \|\chi\|_\mathrm{Lip}}|\xi(z)|$. For every $t,t'\in \R_+$ and $x,x'\in\R_+^D$, we have 
\begin{align}
    \Ll|S_t\chi(x) - S_{t'}\chi(x')\Rr|\leq C\Ll|t-t'\Rr| + \|\chi\|_\mathrm{Lip}\Ll|x-x'\Rr|.
\end{align}
\end{lemma}

\begin{proof}
By Proposition~\ref{p.modified hopf-lax} and Remark~\ref{r.xi.monotonicity}, the mapping $(t,x) \mapsto S_t \chi(x)$ is the unique $\R_+^D$-increasing and Lipschitz viscosity solution of the equation described therein with the choices of $H = \xi_{\big|\R^D_+}$ and $\phi = \chi$. The desired result thus follows from~\cite[Theorem~1.2~(2a)]{chen2023viscosity}.
\end{proof}

Note that when $D = 1$ the notion of $\R_+^D$-convexity coincides with the usual notion of convexity. In this case it is clear from the alternative representation of $S_t \chi$ obtained in Lemma~\ref{l.St1=St2} that, for every $\chi \in \mfk X$, we have that $S_t \chi$ is convex in the usual sense. In particular, when $D=1$ Proposition~\ref{p.R^D_+-convexity} holds. More generally, the following characterization of $\R^D_+$-convexity holds.

\begin{lemma} \label{l.R^D_+-convexity}
    Let $\chi : \R^D_+ \to \R$ be a continuous function that is  $\mcl C^2$ on $(0,+\infty)^D$. We have that $\chi$ is $\R^D_+$-convex if and only if for every $x \in (0,+\infty)^D$, $\nabla^2 \chi(x) \in \R^{D \times D}_+$.
\end{lemma} 

\begin{proof}
    Assume that $\chi$ is $\R^D_+$-convex. Applying \eqref{e.R^D_+-convexity} to $(x,ty,sy)$ and letting $s\to0$ and $t\to 0$, we obtain that for every $x \in (0,+\infty)^D$ and every $y,z \in \R^D_+$, 
    \begin{e*}
        y \cdot \nabla^2 \chi(x)z \geq 0.
    \end{e*}
    This exactly means that $\nabla^2\chi(x) \in \R^{D \times D}_+$. Conversely, for $x \in (0,+\infty)^D$ and $y,z \in \R^D_+$ we have 
    \begin{e*}
        (\chi(x+y+z) - \chi(x+z))-(\chi(x+y) - \chi(x)) = \int_0^1 \int_0^1 y \cdot \nabla^2 \chi(x+ty +sz)z \d s \d t.
    \end{e*}
    Thus, if $\nabla^2\chi$ is $\R^{D \times D}_+$-valued, we obtain that the right-hand side of the previous display is nonnegative and thus \eqref{e.R^D_+-convexity} holds. We then obtain \eqref{e.R^D_+-convexity} for $x \in \R^D_+$ by continuity.
\end{proof}
 
To prove Proposition~\ref{p.R^D_+-convexity} we will make use of the following idea. Assume that the map $u : (t,x) \mapsto S_t \chi(x)$ is smooth and let $w = \nabla^2 u$ denote its Hessian. We have 
\begin{equation*}
    \begin{cases}
        \partial_t w = \nabla w \nabla \xi(\nabla u ) + w \nabla^2 \xi(\nabla u) w \text{ on } (0,+\infty) \times \R_+^D \\
        w(0,\cdot) = \nabla^2 \chi.
    \end{cases}
\end{equation*}
By Lemma~\ref{l.R^D_+-convexity}, we have that $w(0,\cdot)$ is $\R^{D \times D}_+$-valued. Furthermore, it can be checked that the PDE in the previous display preserves positivity of the coefficients of its solutions. Indeed, on the right-hand side the first term is a plain transport term and the second term is of the form $f(t,x,w(t,x))$ with $f(t,x,z) \in \R^{D \times D}_+$ when $z \in \R^{D \times D}_+$, this is because the function $\xi$ itself is $\R_+^D$-convex. 

To make this observation rigorous we will have to pay attention to several details. First, we will extend $\xi$ and $\chi$ in order to work with PDEs on $\R^D$ and avoid boundary problems. Second, we will add a regularizing term of the form $\varepsilon \Delta u$ on the right hand-side in the Hamilton--Jacobi equation defining $S_t \chi$, this will yield a family of approximations of $S_t \chi$ that are smooth enough to perform the computations sketched above.  

Observe that every $\chi \in \mfk X$ is $1$-Lipschitz, hence the function $S_t \chi$ is $\R_+^D$-increasing thanks to \eqref{e.def.St1} and $1$-Lipschitz thanks to Lemma~\ref{l.S_tLip}. Thus, changing $\xi$ outside the intersection of $\R_+^D$ with the ball of radius $1$ in $\R^D$ does not affect $S_t\chi$, this is proved rigorously in \cite[Theorem~1.2~(2)~(c)]{chen2023viscosity}. Therefore, we have some freedom on choosing how to define $\xi$ outside of $B(0,1) \cap \R_+^D$, and it will be convenient for our purposes to use this freedom to modify $\xi$ so that it becomes uniformly Lipschitz. 
Let $L$ denote the largest value taken by $|\nabla \xi|$ on the set of $x \in \R^D_+$ such that $\sum_{d =1}^D x_d \leq 2D$. For every $x \in \R_+^D$, let
\begin{e*}
   \xi_{\text{reg}}(x) = \begin{cases}    
                        \max\{\xi(x),\xi(0) + 2L(\sum_{d = 1}^D x_d - D)\} &\text{ if } \sum_{d=1}^D x_d \leq 2D \\
                        \xi(0) + 2L(\sum_{d = 1}^D x_d - D) &\text{ otherwise}.
                        \end{cases}
\end{e*}
It can be checked (see \cite[Proposition~6.8]{mourrat2020free} and \cite[Lemma~4.3]{chen2022hamilton}) that $\xi_{\text{reg}}$ coincides with $\xi$ on $B(0,1) \cap \R_+^D$, that $\xi_{\text{reg}}$ is convex and Lipschitz on $\R_+^D$, and that it is $\R_+^D$-convex and $\R_+^D$-increasing on $\R_+^D$ (this last property is referred to as ``proper'' in the cited references).

We define $\overline{\xi} : \R^D \to \R$ by setting 
\begin{e*}
    \overline \xi(x) = \xi_{\text{reg}}(x_+).
\end{e*}
It can be easily checked that $\overline \xi$ is locally Lipschitz and convex. We let $(\overline S_t)_{t \geq 0}$ denote the semigroup associated to the equation 
\begin{e*}
    \partial_t u - \overline \xi(\nabla u) = 0 \text{ on } (0+,\infty) \times \R^D.
\end{e*}
According to \cite{chen2023viscosity} or \cite{HJbook}, the semigroup $(\overline S_t)_{t \geq 0}$ is well-defined on the set of Lipschitz functions $\R^D \to \R$. 
\begin{lemma} \label{l.removeborder}
    Let $\chi \in \mathfrak X$, define $\overline \chi : \R^D  \to \R$ by setting $\overline \chi(x) = \chi(x_+)$. Then $\overline \chi$ is Lipschitz and convex on $\R^D$ and we have for every $(t,x) \in [0,+\infty) \times \R_+^D$,
    \begin{e}
        \overline S_t \overline \chi(x) = S_t \chi(x).
    \end{e}
\end{lemma}

\begin{proof}
    In this proof, we let $\xi_\text{reg}^*$ denote the convex dual of $\xi_\text{reg}$ with respect to $\R^D_+$, more precisely 
    \begin{e*}
        \xi^*_\text{reg}(y) = \sup_{x \in \R^D_+} \left\{ x \cdot y - \xi_\text{reg}(x) \right\}.
    \end{e*}
    We also let $\bar \xi^*$ denote the convex dual of $\bar \xi$ with respect to $\R^D$, more precisely
    \begin{e*}
         \bar \xi^*(y) = \sup_{x \in \R^D} \left\{ x \cdot y - \bar \xi(x) \right\}.
    \end{e*}
    %
    
    \medskip
    
    \noindent \emph{Step 1.} We show that $\overline \chi$ is Lipschitz, convex and $\R_+^D$-increasing on $\R^D$.

    Since $\chi \in \mathfrak X$ we have $\overline \chi(x) = \sum_{d=1}^D \overline \chi_d(x_d)$, where $\overline \chi_d$ is the function which coincides with $\chi_d$ on $\R_+$ and is constant taking the value $\chi_d(0) = 0$ otherwise. Therefore it suffices to show that $\overline \chi_d$ is Lipschitz, convex and increasing on $\R$.  Let $\lambda_1 \in \R_+$ and $\lambda_2 \in \R_-$, we have 
    \begin{e*}
        |\overline \chi_d(\lambda_1) - \overline \chi_d(\lambda_2)| = |\chi_d(\lambda_1) - \chi_d(0)| \leq |\lambda_1|.                                                                    
    \end{e*}    
    Now observe that $|\lambda_1| = \lambda_1 \leq \lambda_1 - \lambda_2 = |\lambda_1 - \lambda_2|$, so 
    \begin{e*}
         |\overline \chi_d(\lambda_1) - \overline \chi_d(\lambda_2)| \leq |\lambda_1 - \lambda_2|.
    \end{e*}
    Now observe that the previous display clearly also holds when $\lambda_1,\lambda_2 \in \R_+$ or $\lambda_1,\lambda_2 \in \R_-$, thus $\overline \chi_d$ is Lipschitz. In addition, $\chi_d$ is increasing on $\R$ as a composition of increasing functions and $\overline \chi_d$ is convex on $\R$ as the composition of a convex function and a convex increasing function. In conclusion, $\overline \chi$ is Lipschitz, convex and $\R_+^D$-increasing on $\R^D$.
    
    \medskip
    
    \noindent \emph{Step 2.} We show that for every $x \in \R^D$, $\overline \xi ^*(x) \geq \xi_{\text{reg}}^*(x_+)$ with equality when $x \in \R_+^D$. 

     Let $x \in \R^D$, using Lemma~\ref{l.H^*(x)=H^*(x_+)} in the last line we have 
    \begin{align*}
        \overline \xi^*(x) &= \sup_{y \in \R^D} \left\{ x \cdot y - \xi_{\text{reg}}(y_+) \right\} \\
                               &\geq \sup_{y \in \R_+^D} \left\{ x \cdot y - \xi_{\text{reg}}(y) \right\} \\
                               &= \xi_{\text{reg}}^*(x) \\
                               &= \xi_{\text{reg}}^*(x_+).
    \end{align*}
    This proves the first part of the statement. Now assume that $x \in \R_+^D$. Since $x \cdot y \leq x \cdot y_+$ we have 
    \begin{e*}
        \overline \xi^*(x) = \sup_{y \in \R^D} \left\{ x \cdot y - \xi_{\text{reg}}(y_+) \right\} \leq \sup_{y \in \R^D} \left\{ x \cdot y_+ - \xi_{\text{reg}}(y_+) \right\} = \xi_{\text{reg}}^*(x).
    \end{e*}

    \medskip
    
     \noindent \emph{Step 3.} We show that for every $x \in \R_+^D$, $\overline S_t \overline \chi(x) = S_t \chi(x)$.
     
     Without loss of generality we assume that $t= 1$. Since $\xi_{\text{reg}}$ and $\overline \xi$ are convex on $\R_+^D$ and $\R^D$, we are simply going to verify that the Hopf-Lax representations of $\overline S_1 \overline \chi(x)$ and $S_1 \chi(x)$ coincide. Let $x \in \R^D_+$, thanks to the equality case in the inequality from Step 2 we have 
    \begin{align*}
        \overline S_1 \overline \chi(x) &= \sup_{y \in \R^D} \left\{ \overline \chi(x+y) - \overline \xi^*(y) \right\} \\
                                        &\geq \sup_{y \in \R_+^D} \left\{ \overline \chi(x+y) - \overline \xi^*(y) \right\} \\
                                        &= \sup_{y \in \R_+^D} \left\{ \chi(x+y) - \xi_{\text{reg}}^*(y) \right\} \\
                                        &= S_1 \chi(x).
    \end{align*}
    In addition, thanks to the sub-additivity of the positive part and the inequality from Step 2, we have for every $y \in \R^D$ that
    \begin{align*}
        x+y_+ - (x+y)_+ \in \R_+^D \\
        -\bar \xi^*(y) \leq -\xi_{\text{reg}}^*(y_+).
    \end{align*}
    Since $\overline \chi$ is $\R_+^D$-increasing on $\R^D$ according to Step 1, we have 
    \begin{e*}
        \overline \chi(x+y) -\bar \xi^*(y) \leq \chi(x+y_+) - \xi_{\text{reg}}^*(y_+) \leq S_1 \chi(x).
    \end{e*}
    Taking the supremum over $y \in \R^D$ in the previous display, we obtain $\overline S_1 \overline \chi(x) \leq S_1 \chi(x)$, thus the result.
\end{proof}

We can extend the notion of $\R_+^D$-convexity to functions on $\R^D$ by saying that such a function is $\R_+^D$-convex when \eqref{e.R^D_+-convexity} holds for every $x \in \R^D$ and $y,z \in \R_+^D$.

\begin{definition}[$\R^D_+$-convexity on $\R^D$]
    We say that $\bar \chi : \R^D \to \R$ is $\R^D_+$-convex when for every $x \in \R^D$ and $y,z \in \R_+^D$, we have 
    \begin{e}
       \bar \chi(x+y+z) - \bar \chi(x+z) \geq \bar \chi(x+y) - \bar \chi(x).
    \end{e}
\end{definition}

\begin{remark} \label{r.R^D_+-convexity}
   Reproducing the proof of Lemma~\ref{l.R^D_+-convexity}, one can check that a $\mcl C^2$ function $\bar \chi : \R^D \to \R$ is $\R^D_+$-convex if and only for every $x \in \R^D$, $\nabla^2 \bar \chi(x) \in \R^{D \times D}_+$. \qed
\end{remark}

\begin{lemma} \label{l.R^D_+-convexitysmooth}
    Let $H : \R^D \to \R$ and $g : \R^D \to \R$ be $\mcl C^\infty$ such that all the derivatives of order $\geq 1$ of $H$ are in $L^\infty$ and all the derivatives of order $\geq 2$ of $g$ are in $L^\infty$. Consider the Cauchy problem
    \begin{e} \label{e.vhj}
        \begin{cases}
            \partial_t u^\varepsilon - H(\nabla u^\varepsilon) = \varepsilon \Delta u^\varepsilon \text{ on } (0,+\infty) \times \R^D \\
            u^\varepsilon(0,\cdot) = g.
        \end{cases}
    \end{e}
    There is exactly one function $u^\varepsilon$ that is $\mcl C^\infty$ smooth with bounded derivatives of order $\geq 2$ and that satisfies \eqref{e.vhj}. Furthermore, if $H$ and $g$ are $\R_+^D$-convex, then for every $t \geq 0$, the function $u^\varepsilon(t,\cdot)$ is $\R_+^D$-convex.
\end{lemma}

The well-posedness of \eqref{e.vhj} is well known in many contexts and is a basic result in the theory of viscosity solutions to Hamilton--Jacobi equations  \cite{evans1980solving,friedman1973cauchy}. For a pedagogical and detailed proof of this result under some slightly different boundedness assumptions, we refer to the lecture notes \cite{benartzi2007viscous}. In the usual approach to well-posedness, the main ingredient is Duhamel's formula which allows to show that the solutions of \eqref{e.vhj} are the fixed points of a certain functional. For the sake of completeness, we explain in the proof below how to adapt the estimates of \cite[Section~3]{benartzi2007viscous} yielding existence of a fixed point.

\begin{proof}[Proof of Lemma~\ref{l.R^D_+-convexitysmooth}]
    Up to modifying $H$ and rescaling time, we can assume without loss of generality that $\varepsilon = 1$. In this proof we will thus not explicitly make the $\varepsilon$ dependence of $\eqref{e.vhj}$ apparent and simply write $u$ instead of $u^\varepsilon$. We let $P$ denote the $D$-dimensional heat kernel, more precisely
    \begin{e*}
        P(t,x) = \frac{1}{\left(4 \pi t\right)^{\frac{D}{2}}} e^{-\frac{|x|^2}{4t}}.
    \end{e*}
    Given two functions $a,b : \R^D \to \R$, we let $a*b$ denote their additive convolution, namely
    \begin{e*}
        a* b(x) = \int_{\R^D} a(y) b(x-y) \d y.
    \end{e*}

    \medskip
    
    \noindent \emph{Step 1.} Existence and uniqueness of solutions to \eqref{e.vhj}.
    
     Given a $\mcl C^\infty$ function $u : \R_+ \times \R^D \to \R$, we let for every $t \geq 0$
    \begin{e*}
        S(u)(t,\cdot)= P(t,\cdot) * g + \int_0^t P(t-s,\cdot) * H(\nabla u(s,\cdot)) \d s.
    \end{e*}
    We let $u^{(-1)} = 0$ and for every $k \geq 0$, $u^{(k)} = S(u^{(k-1)})$. By construction of $S$, for every $k \geq 0$, we have
    \begin{e} \label{e.recurrence pde}
        \begin{cases}
                \partial_t u^{(k)} - \Delta u^{(k)} = H(\nabla u^{(k-1)})   \text{ on } (0,+\infty) \times \R^D \\
                u^{(k)}(0,\cdot) = g.
            \end{cases}
    \end{e}
    Therefore, it suffices to show that the sequence $(u^{(k)})_{k \geq 1}$ converges in an appropriate sense to a function $u$ to get the existence of a solution for \eqref{e.vhj}.
    To do so, let us fix $T > 0$. We will prove that for every $n \geq 2$, the sequence $(\nabla^n u^{(k)})_{k \geq 1}$ is uniformly bounded on $[0,T] \times \R^D$. From this we will deduce that the sequence $((\partial_t,\nabla)^2 u^{(k)})_{k \geq 1}$ is uniformly bounded on $[0,T] \times \R^D$. Since $S(u)(0,\cdot) = g$ and $\nabla S(u)(0,\cdot) = \nabla g$, thanks to the Arzelà--Ascoli theorem, those estimates will yield local uniform convergence up to extraction of the sequence $(\nabla^n u^{(k)})_{k \geq 1}$ for all $n \geq 0$, which will conlude the proof of existence of a solution with bounded derivatives of order $\geq 2$ for \eqref{e.vhj}. 
    
    We define
    \begin{e*}
        L_n(k,t) = \| \nabla^n u^{(k)}(t,\cdot) \|_{L^\infty}.
    \end{e*}
    We now explain how to prove bounds that are uniform in $k \in \N$ and $t \in [0,T]$ for $L_2(k,t)$ following the proof of \cite[Claim~3.6]{benartzi2007viscous}, those estimates can easily be generalized to any $n \geq 2$. Differentiating twice the recurrence relation $u^{(k)} = S(u^{(k-1)})$, we have 
    \begin{align*}
        &\partial_{d}\partial_{d'} u^{(k)} 
        \\
        &= P(t)*  \partial_{d}\partial_{d'} g + \int_0^t \partial_{d'} P(t-s) * \left( \nabla H(\nabla u^{(k-1)})(s,\cdot) \cdot \nabla \partial_d u^{(k-1)}(s,\cdot) \right)\d s
    \end{align*}
    There exists a constant $c > 0$ depending only on $D$ such that $\| \nabla P(t) \|_{L^1} \leq ct^{-1/2}$. Therefore it follows from Young's inequality $\|a*b\|_{L^\infty} \leq \|a\|_{L^1}\|b\|_{L^\infty}$ and the previous display that
    \begin{e*}
        L_2(k,t) \leq \|\nabla^2 g\|_{L^\infty} + c |\nabla H|_\infty \int_0^t \frac{1}{\sqrt{t-s}} L_2(k-1,s) \d s.
    \end{e*}
    After further manipulations of this inequality and induction on $k$, we obtain the existence of a constant $\Lambda > 0$ depending only on $D$ and $|\nabla H|_\infty$ such that 
    \begin{e*}
        L_2(k,t) \leq 2  \|\nabla^2 g\|_{L^\infty} e^{\Lambda t}.
    \end{e*}
    More generally, by induction on $n$ using a similar argument, we can show that for every $T > 0$ and every $n \geq 2$ there exists $C(n,T) > 0$ such that for every $k \in \N$ and every $t \in [0,T]$,
    \begin{e*}
        L_n(k,t) \leq C(n,T).
    \end{e*}
    Those are the announced estimates for $\nabla^n u^{(k)}$. To convert this into estimates for $(\partial_t,\nabla)^2 u^{(k)}$ we use \eqref{e.recurrence pde} to deduce that 
    \begin{align*}
        &\partial_t \nabla u^{(k)} = \nabla^2 u^{(k)} \nabla H(\nabla u^{(k)}) + \nabla \Delta u^{(k-1)}, \\
        &\partial_t \partial_t u^{(k)} = \partial_t \nabla u^{(k)} \cdot \nabla H(\nabla u^{(k)}) + \partial_t \Delta u^{(k-1)}.
    \end{align*}
    In addition, using \eqref{e.recurrence pde} once more we have 
    \begin{align*}
        \partial_t \Delta u^{(k-1)} = \sum_{i} \partial_i \nabla u^{(k-1)} \cdot \nabla ^2 H(u^{(k-1)}) \partial_i \nabla u^{(k-1)} + \nabla H(\nabla u^{(k-1)}) \cdot \nabla \Delta u^{(k-1)}
        \\
        + \Delta \Delta u^{(k-2)}.
    \end{align*}
    Since $\nabla H \in L^\infty$ and $\nabla^2 H \in L^\infty$, it follows from the previous two displays and our previous estimates that $(\partial_t \nabla u^{(k)})_{k \geq 1}$ and $(\partial_t \partial_t u^{(k)})_{k \geq 1}$ are uniformly bounded on $[0,T] \times \R^D$.  Therefore, $((\partial_t, \nabla)^2 u^{(k)})_{k \geq 1}$ is uniformly bounded on $[0,T] \times \R^D$, so $((\partial_t, \nabla) u^{(k)})_{k \geq 1}$ is uniformly Lipschitz on $[0,T] \times \R^D$. Since at $t = 0$, $(\partial_t,\nabla) u^{(k)}(0,\cdot)=  (\Delta g + H(\nabla g),\nabla g)$, we deduce that $((\partial_t, \nabla) u^{(k)})_{k \geq 1}$ converges uniformly on $[0,T] \times \R^D$ thanks to the Arzelà-Ascoli theorem. In addition, it follows from the uniform Lipschitzness that $((\partial_t, \nabla) u^{(k)})_{k \geq 1}$ is locally uniformly bounded on $[0,T] \times \R^D$, so $(u^{(k)})_{k \geq 1}$ is uniformly locally Lipschitz on $[0,T] \times \R^D$. Since $u^{(k)}(0,\cdot) = g$, we can apply the Arzelà-Ascoli theorem once more to obtain local uniform convergence of $(u^{(k)})_{k \geq 1}$ on $[0,T] \times \R^D$.
    This concludes the existence part of the proof. For the uniqueness part we can proceed as in the proof of \cite[Theorem~3.2]{benartzi2007viscous} but using weighted $L^p$ norms instead of the $L^\infty$ norm. 
    
    \medskip
    
    \noindent \emph{Step 2.} Let $g_\alpha(x) = g(x) + \alpha \left( \sum_{d}^D x_d \right)^2$ and let $u_\alpha$ the unique solution of \eqref{e.vhj} with initial condition $g_\alpha$. We show that for every $t \geq 0$, $\nabla^2 u_\alpha(t,\cdot) \to \nabla^2 u(t,\cdot)$ pointwise up to extraction as $\alpha \to 0$.
    
     To prove this, we perform estimates as in Step 1, but this time for 
    \begin{e*}
       \tilde L_n(k,t) = \|\nabla^n u^{(k)}_\alpha(t,\cdot) - \nabla^n u^{(k)}(t,\cdot) \|_{L^p(w)},
    \end{e*}
    with $n = 1,2$, where $L^p(w)$ denotes the $L^p$ space with weight $w(x) = e^{-|x|}$ and $\| f\|_{L^p(w)} = \| fw\|_{L^p}$. In this context we can replace the usual Young inequality by the following weighted version,
    \begin{e*}
        \| a*b \|_{L^p(w)} \leq  \| a \|_{L^1(\overline{w})}\| b \|_{L^p(w)},
    \end{e*}
    where $\overline{w}(x) = e^{|x|}$. This weighted inequality can be deduced from the usual Young inequality after observing that $w(x+y) \leq \overline{w}(x) w(y) $, we refer to \cite[Theorem~2.4]{mourrat2017wellposedness} for a full proof. As previously, this yields estimates that are uniform in $t \in [0,T]$ and $k \in \N$ and that vanish in the limit $\alpha \to 0$. From this, we obtain convergence of $\nabla^2 u_\alpha(t,\cdot) \to \nabla^2 u(t,\cdot)$ in $L^p(w)$ and thus we obtain pointwise convergence after extraction.
    
    \medskip
    
    \noindent \emph{Step 3.} We show that if $H$ and $g$ are $\R_+^D$-convex, then for every $t \geq 0$, $u(t,\cdot)$ is  $\R_+^D$-convex.

    According to Lemma~\ref{l.R^D_+-convexity} and Remark~\ref{r.R^D_+-convexity}, it suffices to show that $w= \nabla^2 u$ is $\R^{D \times D}_+$-valued. To do so, observe that the coefficients of $w$ satisfy a system of semilinear parabolic PDEs. More precisely, for every $d,d' \in \{1,\dots,D\}$ we let $w_{dd'}$ denote the coefficient of index $(d,d')$ of $w$. Then, $w$ is a classical solution of 
    \begin{e*}
        \begin{cases}
            \partial_t w_{dd'} = \Delta w_{dd'} + b(t,x) \cdot \nabla w_{dd'} + f_{dd'}(t,x,w) \text{ on } (0,+\infty) \times \R^D \\
            w(0,\cdot) = \nabla^2 g,
        \end{cases}
    \end{e*}
    where $b(t,x) = \nabla H(\nabla u(t,x))$ and $f_{dd'}(t,x,w) = \left( w \nabla^2 H(\nabla u) w \right)_{dd'}$ is the coefficient of index $(d,d')$ of the matrix $w \nabla^2 H(\nabla u) w$. The key point here is that since $H$ is $\R_+^D$-convex, for $z \in \R_+^{D \times D}$ we have $f_{dd'}(t,x,z) \geq 0$. Using this we are going to show that $w_{dd'}$ is nonnegative.

    Let $\alpha > 0$ and let $w^\alpha = \nabla^2 u_\alpha$ where $u^\alpha$ is defined as in Step 2. The function $w^\alpha$ solves a system of semilinear parabolic PDEs as in the previous display but with different $b$ and $f$ depending on $\alpha$ and with initial condition $w^\alpha(0,\cdot) = \nabla^2 g + \alpha \1$, where $\1$ is the $D \times D$ matrix with all coefficients equal to $1$. Let
    \begin{e*}
        T_\alpha = \sup \left\{t \geq 0\ \big| \  \forall x \in \R^D, \, \forall d,d' \leq D, \, w^\alpha_{dd'}(t,x) \geq 0 \right\}
    \end{e*}
    denote the last time at which all the coefficients of $w^\alpha(t,\cdot)$ are all nonnegative. All the derivatives or order $\geq 2$ of $u^\alpha$ are bounded, thus for every $T > 0$, we have that $w^\alpha(\cdot,x)$ is uniformly Lipschitz in $x$ on $[0,T]$. Thus, it follows from the fact that $w_{dd'}^\alpha(0,x) \geq \alpha$, that $T_\alpha > 0$. Then $w_{dd'}^\alpha$ solves 
    \begin{e*}
        \begin{cases}
            \partial_t w^\alpha_{dd'} \geq \Delta w^\alpha_{dd'} + b_\alpha(t,x) \cdot \nabla w^\alpha_{dd'}  \text{ on } (0,T_\alpha) \times \R^D \\
            w^\alpha_{dd'} \geq \alpha.
        \end{cases}
    \end{e*}
    In particular, we are now working with a plain linear parabolic PDEs. By the comparison principle, we obtain that for every $(t,x) \in (0,T_\alpha) \times \R^D$, $w^\alpha_{dd'}(t,x) \geq \alpha$. Arguing by contradiction, we now assume that $T_\alpha <+\infty$. By continuity, we have $w_{dd'}^\alpha(T_\alpha,x) \geq \alpha$. Using the same argument that yielded $T_\alpha > 0$, we obtain that $w_{dd'}^\alpha$ remains nonnegative for a small time after $T_\alpha$, and this contradicts the definition of $T_\alpha$. In conclusion $T_\alpha = +\infty$ and  for every $(t,x) \in (0,+\infty) \times \R^D$, we have $w^\alpha_{dd'}(t,x) \geq \alpha$. Finally, from Step 2, we have that as $\alpha \to 0$, $w^\alpha_{dd'}(t,x) \to w_{dd'}(t,x)$. Thus, $w$ is $\R^{D \times D}_+$-valued, as desired. 
\end{proof}

\begin{proof}[Proof of Proposition~\ref{p.R^D_+-convexity}]

    By Lemma~\ref{l.removeborder} we know that $S_t \chi$ is the restriction to $\R_+^D$ of $\overline{S}_t \overline{\chi}$, therefore it suffices to prove that  $\overline{S}_t \overline{\chi}$ is $\R_+^D$-convex. We will prove that this result is a consequence of Lemma~\ref{l.R^D_+-convexitysmooth} by performing several approximation procedures.

    First, observe that for every $\delta > 0$ the mollifications $H_\delta = \overline{\xi} * \eta_\delta$ and $g_\delta = \overline{\chi}*\eta_\delta$ are $\R_+^D$-convex and $\mcl C^\infty$. Since $\xi_{\text{reg}}$ is linear outside a compact set, $H_\delta$  satisfies the hypotheses of Lemma~\ref{l.R^D_+-convexitysmooth}. On the other hand, the function $g_\delta$ may have unbounded derivatives of higher order. This can be easily fixed by replacing $\chi$ by $\chi_R$ where, given $R > 0$, we let $\chi_{d,R}(\lambda) = \chi_d(\lambda)$ for $\lambda \leq R$ and $\chi_{d,R}(\lambda)=\chi_d(R) + \lambda-R$ otherwise and $\chi_R = \sum_{d} \chi_{d,R}$. This way we guarantee that outside of a large ball $\overline{\chi_R}$ is linear and thus its derivatives of all orders are bounded. Also note that since each $\chi_d$ is $1$-Lipschitz and convex, the extensions $\chi_{d,R}$ are also $1$-Lipschitz and convex. In particular $\overline \chi_R$ is $\R_+^D$-convex.

\medskip

    \noindent \emph{Step 1.} Let $H$ and $g$ satisfying the hypotheses of Lemma~\ref{l.R^D_+-convexitysmooth} and such that $g$ is Lipschitz. We show that for every $t \geq 0$, $u(t,\cdot)$ is $\R_+^D$-convex, where $u$ is the unique vicosity solution of 
    \begin{e*}
        \begin{cases}
            \partial_t u - H(\nabla u) = 0 \text{ on } (0,+\infty) \times \R^D \\
            u(0,\cdot) = g.
        \end{cases}
    \end{e*}

     It is classical \cite{evans1980solving} that in the limit $\varepsilon \to 0$, the sequence of functions $(u^\varepsilon)_\varepsilon$ from Lemma~\ref{l.R^D_+-convexitysmooth} converges to $u$. Since according to Lemma~\ref{l.R^D_+-convexitysmooth}, for every $t \geq 0$, $u^\varepsilon(t,\cdot)$ is $\R_+^D$-convex, it follows that $u(t,\cdot)$ is $\R_+^D$-convex.

\medskip

    \noindent \emph{Step 2.} We show that for every $t \geq 0$, $u_{R}(t,\cdot)$ is $\R_+^D$-convex, where $u_{R}$ is the unique vicosity solution of 
    \begin{e*}
        \begin{cases}
            \partial_t u_R - \overline{\xi}(\nabla u_R) = 0 \text{ on } (0,+\infty) \times \R^D \\
            u_R(0,\cdot) = \overline{\chi_{R}}.
        \end{cases}
    \end{e*}

     We can molify $\overline{\xi}$ and $\overline{\chi_{R}}$ into $H_\delta$ and $g_{R,\delta}$. According to Step 1, the unique viscosity solution $u_{R,\delta}$ of 
    \begin{e*}
        \begin{cases}
            \partial_t u_{R,\delta} - H_\delta(\nabla u_{R,\delta}) = 0 \text{ on } (0,+\infty) \times \R^D \\
            u_{R,\delta}(0,\cdot) = g_{R,\delta},
        \end{cases}
    \end{e*}
    is $\R_+^D$-convex at every fixed $t$. In addition $H_\delta \to \overline{\xi}$ and $g_{R,\delta} \to \overline{\chi}_R$ locally uniformly as $\delta \to 0$. Therefore, any locally uniform limit of the sequence $(u_{R,\delta})_{\delta}$ must be the unique viscosity solution of the equation defining $u_R$ \cite[Section~6]{guide}. Hence, since the sequence $(u_{R,\delta})_{\delta}$ is uniformly Lipschitz, it follows from the Arzelà--Ascoli theorem that $\delta \to 0$, $u_{R,\delta} \to u_R$ locally uniformly.  The desired result follows.
    
    \medskip
    
    \noindent \emph{Step 3.} We show that $\overline{S}_t \overline{\chi}$ is $\R_+^D$-convex.

     The initial condition $\overline{\chi}_R$ and $\overline{\chi}$ coincide on $B(0,R)$.  Since the Hamilton--Jacobi  equation has finite speed of propagation \cite[Exercise~3.8 and solution]{HJbook},
    we have that $u_R(t,\cdot)$ and $\overline{S}_t \overline{\chi}$ coincide on $B(0,R-t/L)$ where $L$ is a Lipschitz constant of $H$. Hence, we have that $u_R(t,\cdot) \to \overline S_t \overline{\chi}$ pointwise as $R \to +\infty$. This allows to deduce the desired result from Step 2.
\end{proof}

%
%
%
%
%
%

\section{Optimal transport and convex duality}
\label{s.optim}
In this section, we prove the key observations \eqref{e.stchi.to.transport} and \eqref{e.kantorovich.potential}, and use them with convex-duality arguments to show the identity between the variational problems in \eqref{e.convexoptimization} and \eqref{e.convexoptimization2}. 

We recall that $S_t \chi$ was defined in \eqref{e.real.def.Stchi},
and that in \eqref{e.T_t=} we have introduced the quantity 
\begin{align}
    \mcl T_t(\mu,\nu) = \inf_{\pi\in\Pi(\mu,\nu)}\int (t\xi)^*(y-x)\d \pi(x,y),
\end{align}
which can be interpreted as the optimal transport cost from $\mu$ to $\nu$ with cost function $(x,y) \mapsto (t\xi)^*(y-x)$. As explained in \cite[Chapter~1]{villani}, Kantorovich duality yields the dual representation 
\begin{align} \label{e.kantorovich.dual.rep}
    \mcl T_t(\mu,\nu) = \sup_{\chi} \left\{ \int \chi d \nu - \int S_t \chi d\mu \right\},
\end{align}
where the supremum on the right-hand side is taken over all functions $\chi: \R^D_+ \to \R$ that are in $L^1(\nu)$ and are such that $S_t \chi \in L^1(\mu)$. Observe that here $\mcl T_t(\mu,\nu)$ is represented as a supremum of affine functions in $(\mu,\nu)$, in particular this means that the mapping $(\mu,\nu) \mapsto \mcl T_t(\mu,\nu)$ is convex. This dual representation motivates the following definition. 
\begin{definition}[Kantorovich potential] \label{d.kantorovich.potential}
    For every $\mu, \nu \in \mcl P(\R^D_+)$, we say that a function $\chi : \R^D_+ \to \R$ is a Kantorovich potential from $\mu$ to $\nu$ (for the cost function $(x,y) \mapsto (t\xi^*)(y-x)$) when $\chi \in L^1(\nu)$, $S_t \chi \in L^1(\mu)$, and
    \begin{e}
        \mcl T_t(\mu,\nu) = \int \chi d \nu - \int S_t \chi d\mu.
    \end{e}
\end{definition}
We can thus rephrase the identity \eqref{e.kantorovich.potential} as saying that the function $\chi$ is a Kantorovich potential. A key aspect of the proofs of our main results is that the function $\chi$ appearing in \eqref{e.kantorovich.potential} and in the surrounding informal discussion is not only a Kantorovich potential, but also an element of the set~$\mfk X$ defined in \eqref{e.frakX=}.

Recall the collection $\mathfrak X$ defined in~\eqref{e.frakX=}.
For $\psi$ appearing in~\eqref{e.decomp.psi}, we define
\begin{e}\label{e.psi_*=}
    \psi_*(\chi)  = \inf_{\nu \in \mcl P_1(\R_+^D)} \left\{ \int \chi \d \nu - \psi(\nu) \right\},\quad\forall \chi \in \mathfrak X.
\end{e}
For every $(t,\mu) \in \R_+ \times \mcl P_1(\R_+^D)$, we set
\begin{align}
    g(t,\mu) &= \inf_{\chi \in \mathfrak X} \left\{ \int S_t \chi \d\mu - \psi_*(\chi) \right\} \label{e.g(t,mu)=} \\
    h(t,\mu) &= \sup_{ \nu \in \mcl P_1(\R_+^D)} \left\{ \psi(\nu) - \mcl T_t(\mu,\nu) \right\}. \label{e.h(t,mu)=}
\end{align}
\begin{proposition} \label{p.reuninverted}
We have $h = g$ on $\R_+\times \mcl P_1(\R_+^D)$. Moreover, at every $(t,\mu)\in\R_+\times \mcl P_1(\R^D)$, there is a maximizer $\nu$ of $h(t,\mu)$ in~\eqref{e.h(t,mu)=} and a minimizer $\chi$ of $g(t,\mu)$ in~\eqref{e.g(t,mu)=} such that $\chi$ is the derivative of $\psi$ at $\nu$ given by Proposition~\ref{p.diff of psi}.
\end{proposition}

To prove this, we start with some preliminary results.

\begin{lemma}[Existence of maximizers in~\eqref{e.h(t,mu)=}]\label{l.exist_max}
For every $t\in\R_+$ and $\mu\in\mcl P_1(\R_+^D)$, the variational formula of $h(t,\mu)$ in~\eqref{e.h(t,mu)=} achieves its maximum at some $\bar \nu\in \mcl P_1(\R_+^D)$.
Moreover, there is an optimal $\pi\in\Pi(\mu,\bar \nu)$ such that $\mcl T_t(\mu,\bar \nu) = \int (t\xi)^*(y-x)\d \pi(x,y)$.
\end{lemma}

\begin{proof}
\medskip

\noindent \emph{Step~1.}
We show that $\mcl T_t(\mu,\nu)$ and $\msf W_1(\mu,\nu)$ must be bounded for near maximizers $\nu$.
Recall the notation $y_+$ in~\eqref{e.x_+=} for $y\in\R^D$.
According to Lemma~\ref{l.xi* continuous and superlinear}, we have that $\xi^*$ is superlinear on $\R_+^D$. Since $(t\xi)^* = t\xi^*(\scdot/t)$, $(t\xi)^*$ is also superlinear on $\R_+^D$. 
Since $\mcl T_t(\mu,\mu) = 0$, the supremum in ~\eqref{e.h(t,mu)=} can be restricted to those measures $\nu \in \mcl P_1(\R_+^D)$ such that
\begin{align}\label{e.psi(nu)-T_t(mu,nu)>psi(mu)}
    \psi(\nu) - \mcl T_t(\mu,\nu) \geq \psi(\mu).
\end{align}
We start by showing that for every $\nu \in \mcl P_1(\R^D_+)$ satisfying \eqref{e.psi(nu)-T_t(mu,nu)>psi(mu)}, we have that $\msf W_1(\mu,\nu) \leq C$. The Lipschitzness of $\psi$ in \eqref{e.Lip_psi} together with \eqref{e.psi(nu)-T_t(mu,nu)>psi(mu)} yields
\begin{align}\label{e.T_t<W_1}
    \mcl T_t(\mu,\nu) \leq \msf W_1(\mu,\nu) .
\end{align}
For $\eps>0$, let $\pi \in \Pi(\mu,\nu)$ be a nearly optimal coupling for $\mcl T_t(\mu,\nu)$ such that $\int (t\xi)^*(y-x)\d\pi(x,y)\leq \mcl T_t(\mu,\nu)+\eps$.
Let $(X,Y)$ be a random variable with law $\pi$. Then, we can obtain from the above display that
\begin{align}\label{e.E[(txi)^*...]}
    \E \Ll[(t\xi)^*(Y-X)\Rr]\leq \E \Ll[|X-Y|\Rr]+\eps.
\end{align}
We choose $M>1$ and let $R$ be such that~\eqref{e.def_superlinear} holds with $\varsigma$ substituted with $(t\xi)^*$. By this and Lemma~\ref{l.H^*(x)=H^*(x_+)}, we have 
\begin{align}\label{e.(txi)^*(Y-X)geq M...}
    (t\xi)^*(Y-X)\geq M|(Y-X)_+|\qquad\text{on the event $|(Y-X)_+|\geq R$}.
\end{align}
Also by \eqref{e.|y-x|-|(y-x)_+|}, we have
\begin{align}\label{|X-Y|leq |Y-X|-|(}
    |X-Y| = |Y-X|-|(Y-X)_+|+|(Y-X)_+| \leq |X|+|(Y-X)_+|.
\end{align}
Using these and $|(Y-X)_+|\mathds{1}_{|(Y-X)_+|< R}\leq R$, we get
\begin{align*}
    M\E \Ll[|(Y-X)_+|\Rr]-MR &\stackrel{\eqref{e.(txi)^*(Y-X)geq M...}}{\leq}\E \Ll[(t\xi)^*(Y-X)\Rr]\stackrel{\eqref{e.E[(txi)^*...]}}{\leq}\E \Ll[|X-Y|\Rr]+\eps
    \\
    &\stackrel{\eqref{|X-Y|leq |Y-X|-|(}}{\leq} \E[|X|]+\E \Ll[|(Y-X)_+|\Rr] +\eps.
\end{align*}
In the first inequality in the display, we also used that $(t\xi)^*\geq -t\xi(0)=0$.
Rearranging, we arrive at
\begin{align}\label{e.E[|(Y-X)_+|]<...}
    \E\Ll[|(Y-X)_+|\Rr] \leq \frac{\E[|X|]+MR+\eps}{M-1}.
\end{align}
This along with~\eqref{e.|y-x|-|(y-x)_+|} implies that $\E[|Y-X|]$ is bounded by some constant~$C$.
Hence, we conclude that for every $\nu$ satisfying \eqref{e.psi(nu)-T_t(mu,nu)>psi(mu)}, we have $\mcl T_t(\mu,\nu)\leq\msf W_1(\mu,\nu) \leq C$, where we used~\eqref{e.T_t<W_1} in the first inequality.

\medskip

\noindent \emph{Step~2.}
Let $(\nu_n)_{n\in\N}$ be a maximizing sequence of~\eqref{e.h(t,mu)=}, namely,
\begin{align}\label{e.lim-psi(nu_n)...}
    \lim_{n\to\infty} \Ll(\psi(\nu_n)-\mcl T_t(\mu,\nu_n) \Rr)= h(t,\mu).
\end{align}
Our goal is to extract a maximizer from its subsequential limits.
We can choose the sequence $(\nu_n)_{n\in\N}$ such that \eqref{e.psi(nu)-T_t(mu,nu)>psi(mu)} holds for every $n$. By the previous step, we have that $\mcl T_t(\mu,\nu_n)\leq \msf W_1(\mu,\nu_n)\leq C$ for every $n$. Fix any $\eps>0$, since $\mcl T_t(\mu,\nu_n)$ has to be finite, we can choose a nearly optimal coupling $\pi_n\in \Pi(\mu,\nu_n)$ with $\int (t\xi)^*(y-x)\d \pi_n(x,y)\leq \mcl T_t(\mu,\nu_n)+\eps$.
Since $\mu\in\mcl P_1(\R_+^D)$ and $\msf W_1(\mu,\nu_n)~\leq~C$, the first moments of $\nu_n$ are bounded uniformly in $n$. 
Hence, by passing to a subsequence and invoking Skorokhod's representation theorem, we may assume that there is a sequence $(X_n,Y_n)_{n\in\N}$ of random variables satisfying $\mathrm{Law}(X_n,Y_n) = \pi_n$ and that $(X_n,Y_n)$ converges almost surely to some $(X,Y)$. Clearly, we have $\mu=\mathrm{Law}(X)$ since $\mu=\mathrm{Law}(X_n)$ for every $n$.
We need that $(X_n,Y_n)_{n\in\N}$ also converges to $(X,Y)$ in $L^1$, which together with the above follows from the uniform integrability of the second marginal: 
\begin{align}\label{e.limlimsupE[X...]=0}
    \lim_{r\to\infty} \sup_n \E \Ll[|Y_n|\one_{|Y_n|\geq r}\Rr] =0.
\end{align}
We postpone the proof of this fact and first use the $L^1$-convergence to find a maximizer. 

Let $\bar \nu=\mathrm{Law}(Y)$ and consequently, we have $\bar \nu\in\mcl P_1(\R_+^D)$. Then, the $L^1$-convergence together with the Lipschitzness of $\psi$ in \eqref{e.Lip_psi} implies
\begin{align}\label{e.lim-psi(nu_n)...2}
    \lim_{n\to\infty}\psi(\nu_n) =\psi(\bar \nu).
\end{align}
On the other hand, the function $(t\xi)^*$ is lower semi-continuous as a supremum of continuous functions, thus it follows from  Fatou's lemma that (recall that $(t\xi)^*\geq -t\xi(0)= 0$ is bounded from below)
\begin{align*}
    \E\Ll[(t\xi)^*(Y-X)\Rr] \leq \liminf_{n\to\infty} \E\Ll[(t\xi)^*(Y_n-X_n)\Rr],
\end{align*}
which further yields
\begin{align*}
    \mcl T_t(\mu,\bar \nu)\leq \E\Ll[(t\xi)^*(Y-X)\Rr] \leq \liminf_{n\to\infty}\mcl T_t(\mu,\nu_n) + \eps.
\end{align*}
This along with~\eqref{e.lim-psi(nu_n)...2} gives
\begin{align*}
    \psi(\bar \nu) - \mcl T_t(\mu,\bar \nu) 
    \geq \psi(\bar \nu) - \E\Ll[(t\xi)^*(Y-X)\Rr]
    \geq \limsup_{n\to\infty} \psi(\nu_n)- \mcl T_t(\mu,\nu_n)-\eps
    \\
    \stackrel{\eqref{e.lim-psi(nu_n)...}}{=} h(t,\mu)-\eps \stackrel{\eqref{e.h(t,mu)=}}{\geq } \psi(\bar \nu) - \mcl T_t(\mu,\bar \nu) -\eps.
\end{align*}
Since $\eps$ is arbitrary, we conclude that $\bar \nu$ maximizes~\eqref{e.h(t,mu)=} and $\mcl T_t(\mu,\bar \nu)=\E\Ll[(t\xi)^*(Y-X)\Rr]$, implying that $\pi$ is an optimal coupling for $\mcl T_t(\mu,\bar \nu)$.

\medskip

\noindent \emph{Step~3.}
It remains to verify~\eqref{e.limlimsupE[X...]=0}.
We need the following simple fact: for real numbers $r\geq 0$ and $a,b,c\geq0$ satisfying $a\leq b+c$, we have
\begin{align}\label{e.abc_tail}
    a\one_{a\geq 4r} \leq 2 b\one_{b\geq r} + 2 c\one_{c\geq r},
\end{align}
which follows from the fact that
\begin{align*}
    a\one_{a\geq 4r} \leq (b+c)\one_{a\geq 4r} &= b\one_{a\geq 4r,\, b\geq r} + b\one_{a\geq 4r,\, b< r} +  c\one_{a\geq 4r,\, c\geq r} + c\one_{a\geq 4r,\, c< r}
    \\
    &\leq b\one_{b\geq r} + \tfrac{a}{4}\one_{a\geq 4r} +  c\one_{c\geq r} + \tfrac{a}{4}\one_{a\geq 4r}.
\end{align*}
Simple applications of~\eqref{e.abc_tail} are
\begin{align*}
    |Y_n|\one_{|Y_n|\geq 4r} & \leq 2|X_n|\one_{|X_n|\geq r} +  2 |X_n-Y_n|\one_{|X_n-Y_n|\geq r},
    \\
    |X_n-Y_n|\one_{|X_n-Y_n|\geq 4r} &\leq 2|X_n|\one_{|X_n|\geq r} +  2 |(Y_n-X_n)_+|\one_{|(Y_n-X_n)_+|\geq r},
\end{align*}
where we used the triangle inequality in the first line and~\eqref{e.|y-x|-|(y-x)_+|} in the second line. Due to $\mathrm{Law}(X_n) =\mu$ independent of $n$, from the above display, we can see that~\eqref{e.limlimsupE[X...]=0} follows if we can show
\begin{align}\label{e.limlimsupE[Y-X...]=0}
    \lim_{r\to\infty} \sup_n \E \Ll[|(Y_n-X_n)_+|\one_{|(Y_n-X_n)_+|\geq r}\Rr] =0.
\end{align}
By the choice of $\nu_n$ and $\pi_n$, we have $\E\Ll[(t\xi)^*(Y_n-X_n)\Rr]\leq \mcl T_t(\mu,\nu_n)+\eps \leq C+\eps$. Recall that, for any $M>1$, we can find $R$ such that~\eqref{e.(txi)^*(Y-X)geq M...} holds. Hence, for $r>R$, we have
\begin{align*}
    \E \Ll[|(Y_n-X_n)_+|\one_{|(Y_n-X_n)_+|\geq r}\Rr] \leq M^{-1}\E\Ll[(t\xi)^*(Y_n-X_n)\Rr] \leq M^{-1}(C+\eps)
\end{align*}
uniformly in $n$. This implies~\eqref{e.limlimsupE[Y-X...]=0} and completes the proof.
\end{proof}

We now show \eqref{e.stchi.to.transport}. 
\begin{lemma} \label{l.hopf-lax with linear initial condtion}
    Let $\chi : \R_+^D \to \R$ be a Lipschitz and $\R_+^D$-increasing function and let $(S_t)_{t\in\R_+}$ be given as in~\eqref{e.real.def.Stchi}. For every $(t,\mu) \in \R_+ \times \mcl P_1(\R_+^D)$, we have
    \begin{e}\label{e.l.hopf-lax with linear initial condtion}
        \int S_t \chi \d \mu = \sup_{\nu \in \mcl P_1(\R_+^D)} \left\{ \int \chi \d \nu - \mcl T_t(\mu,\nu) \right\}.
    \end{e}
\end{lemma}
\begin{remark} \label{r.refinement} 
    When $\chi \in \mfk X$, we can also write $S_t \chi$ from \eqref{e.real.def.Stchi} as  
    \begin{e}\label{e.r.refinement}
        S_t \chi(x) = \sup_{\substack{x'\,\in\, \R^D_+ \\ x'-x \,\in\, t\nabla \xi([0,1]^D)  } } \left\{ \chi(x') - (t\xi)^*(x'-x) \right\}.
    \end{e}
    To see this, observe first that if $p \in \R^D_+$ is a maximizer in \eqref{e.def.St1}, then $p \in [0,1]^D$. This is because $\chi^*(p) = +\infty$ when $p \in \R^D_+ \setminus [0,1]^D$ (recall that $\chi^*$ is defined as in~\eqref{e.def.chi.*}). Now given $p$ a maximizer in \eqref{e.def.St1}, we have
    \begin{align*}
        S_t \chi(x) &= x \cdot p - \chi^*(p) + t \xi(p) \\
                    &= (x +t \nabla \xi(p))\cdot p - \chi^*(p) - t(\nabla \xi(p) \cdot p - \xi(p)) \\
                    &= (x +t \nabla \xi(p))\cdot p - \chi^*(p) - t \xi^*(\nabla \xi(p)) \\
                    &\leq \sup_{q} \left\{ (x +t \nabla \xi(p))\cdot q - \chi^*(q) \right\} - t \xi^*(\nabla \xi(p)) \\
                    &= \chi(x+t\nabla \xi(p)) - (t \xi)^*(t\nabla \xi(p)).
    \end{align*}
    On the third line, we used that $\xi^*(\nabla \xi(p))=\nabla \xi(p) \cdot p - \xi(p)$, which can be deduced from the definition of $\xi^*$ in~\eqref{e.def.chi.*} (see~\eqref{e.theta=} and~\eqref{e.ttheta=(txi)^*(tnablaxi)}).
    This means that $y = x + t \nabla \xi(p)$ is a maximizer in \eqref{e.real.def.Stchi}, and this proves~\eqref{e.r.refinement}.
    
    Using~\eqref{e.r.refinement} in place of \eqref{e.real.def.Stchi} in the proof of  Lemma~\ref{l.hopf-lax with linear initial condtion} below yields that \eqref{e.l.hopf-lax with linear initial condtion} remains valid when the supremum is taken over $\nu \in \mcl P_1(\R^D_+)$ with the added condition that $\nu$ is supported in $\supp(\mu)+ t\nabla \xi([0,1]^D)$. \qed
\end{remark}
\begin{proof}[Proof of Lemma~\ref{l.hopf-lax with linear initial condtion}]
Using the definition of $S_t$ in~\eqref{e.real.def.Stchi}, we can get
\begin{e*}
    \frac{1}{K} \sum_{k = 1}^K S_t \chi(x_k) = \sup_{y \in (\R^D_+)^K} \left\{ \frac{1}{K} \sum_{k = 1}^K \chi(y_k) - \frac{1}{K} \sum_{k = 1}^K  (t\xi)^*(y_k - x_k) \right\}.
\end{e*}
Let $\Pi^K(x,\scdot)$ be the set of probability measures $\pi$ on $\R_+^D\times \R_+^D$ of the form $\pi = \frac{1}{K} \sum_{k = 1}^K  \delta_{(x_k,y_k)}$ with $y \in (\R_+^D)^K$.
For $\pi \in \Pi^K(x,\scdot)$, let $\pi_2$ be the second marginal of $\pi$ on $\R_+^D$.
The previous display can be written as
\begin{e}\label{e.1/Ksm^KS_tchi=}
    \frac{1}{K} \sum_{k = 1}^K S_t \chi(x_k) = \sup_{\pi \in \Pi^K(x,\scdot)} \left\{ \int \chi d\pi_2 - \int (t\xi)^*(b-a) \d \pi(a,b) \right\}.
\end{e}

Now, we let $\mu \in \mcl P_1(\R_+^D)$ and choose a sequence $(x^K)_{K \geq 1}$ in $(\R_+^D)^K$ such that $\mu$ is limit of $ \frac{1}{K} \sum_{k = 1}^K  \delta_{x^K_k}$ in law. Setting $x = x^K$ in the previous display and letting $K \to +\infty$, we claim that we can get
\begin{e}\label{e.intS_tchi=sup_piinPi}
    \int S_t \chi \d \mu = \sup_{\pi \in \Pi_1(\mu,\scdot)} \left\{ \int \chi d\pi_2 - \int (t\xi)^*(y-x) \d \pi(x,y) \right\} ,
\end{e}
where $\Pi_1(\mu,\scdot)$ denotes the set of probability measures $\pi$ on $\R_+^D \times \R_+^D$ such that $\pi$ has a finite first moment and its first marginal satisfies $\pi_1 = \mu$. 
Assuming the validity of this for now, writing $\Pi_1(\mu,\nu)$ for the set of couplings between $\mu$ and $\nu$, and replacing ``$ \sup_{\pi \in \Pi_1(\mu,\cdot)}$'' by ``$\sup_{\nu \in \mcl P_1(\R_+^D)} \sup_{\pi \in \Pi_1(\mu,\nu)}$'' we obtain~\eqref{e.l.hopf-lax with linear initial condtion} from~\eqref{e.intS_tchi=sup_piinPi} as desired.

Now, we verify~\eqref{e.intS_tchi=sup_piinPi}. Using the definition of $S_t\chi$ in~\eqref{e.real.def.Stchi}, we can deduce that the inequality ``$\geq$'' in~\eqref{e.intS_tchi=sup_piinPi} holds. We focus on the other direction.
Let $x^K$ be given as above~\eqref{e.intS_tchi=sup_piinPi}. 
Let $\pi^K\in\Pi^K(x^K,\scdot)$ be a near maximizer of~\eqref{e.1/Ksm^KS_tchi=} satisfying
\begin{align}\label{e.pi^K_near_optimal}
    \int S_t\chi \d \pi^K_1 \leq K^{-1} +\int\chi\d\pi^K_2 - \int(t\xi)^*(y-x)\d \pi^K(x,y).
\end{align}
Let $(X_K,Y_K)$ be a pair of $\R^D$-valued random variables with law equal to $\pi^K$. Due to the choice of $x^K$, the law of $X_K$ converges to $\mu\in\mcl P_1(\R_+^D)$. By choosing $x^K$ suitably, we may assume that
\begin{align}\label{e.X_K1st_mom}
    \sup_K \E [|X_K|] <\infty.
\end{align}
In the notation of random variables, we can rewrite~\eqref{e.pi^K_near_optimal} as
\begin{align}\label{e.E[S_tchi(X_K)]...}
    \E\Ll[ S_t\chi(X_K)\Rr]\leq K^{-1}+ \E \Ll[\chi(Y_K)\Rr] - \E\Ll[(t\xi)^*(Y_K-X_K)\Rr].
\end{align}
Since $S_t\chi$ and $\chi$ are Lipschitz (see Lemma~\ref{l.S_tLip}), there is a constant $C>0$ such that, for every $x,y$,
\begin{align*}
    |S_t\chi(x)|+|\chi(y)|\leq C(1+|x|+|y-x|).
\end{align*}
This along with~\eqref{e.E[S_tchi(X_K)]...} gives
\begin{align*}
    \E\Ll[(t\xi)^*(Y_K-X_K)\Rr] \leq K^{-1}+ C(1+\E [|X_K|]+\E[|Y_K-X_K|]).
\end{align*}
We can use this condition to substitute the one in~\eqref{e.E[(txi)^*...]} and follow the ensuing steps there (using the superlinearity of $(t\xi)^*$ on $\R^D_+$ by choosing $M$ sufficiently large) to reach an upper bound on $\E[|(Y_K-X_K)_+|]$ similar to the one in~\eqref{e.E[|(Y-X)_+|]<...}. Using~\eqref{e.X_K1st_mom} and~\eqref{e.|y-x|-|(y-x)_+|}, we can deduce
\begin{align}\label{e.supKsupK<}
    \sup_K\E[|Y_K-X_K|]<\infty\qquad \text{and}\qquad \sup_K\E\Ll[(t\xi)^*(Y_K-X_K)\Rr]<\infty.
\end{align}
The first bound in~\eqref{e.supKsupK<} along with~\eqref{e.X_K1st_mom} gives the tightness of $(\pi^K)_{K\in\N}$. By passing to a subsequence and using Skorokhod's representation theorem, we may assume that $(X_K,Y_K)$ converges a.s.\ to some $(X,Y)$ with law $\pi$.

We argue that the convergence also takes place in $L^1$.
It suffices to show that the families $(X_K)_{K}$ and $(Y_K)_{K}$ are uniformly integrable (as in~\eqref{e.limlimsupE[X...]=0}). By choosing $x^K$ suitably, we can verify this for the former. For the latter, the second relation in~\eqref{e.supKsupK<} allows us to repeat the argument for~\eqref{e.limlimsupE[X...]=0} in Step~3 of the proof of Lemma~\ref{l.exist_max}. The extra input is that we need to use the uniform integrability of $(X_K)_{K}$ as their laws are no longer fixed. With this explained, we now have that $(X_K,Y_K)$ converges in $L^1$ to $(X,Y)$.

Due to the assumption on the first marginal, we have $\mathrm{Law}(X)=\mu$ and thus $\pi\in \Pi_1(\mu,\scdot)$. Since $S_t\chi$ and $\chi$ are Lipschitz and thus have linear growths, we can use the dominated convergence theorem to get
\begin{align*}
    \lim_{K\to\infty}\E[S_t\chi(X_K)] = \int S_t\chi \d \mu \quad\text{and}\quad \lim_{K\to\infty}\E[\chi(Y_K)] = \int \chi \d \pi_2.
\end{align*}
The lower semi-continuity of $(t\xi)^*$ and Fatou's lemma (recall that $(t\xi)^*\geq -t\xi(0) = 0$) yield
\begin{align*}
    \int(t\xi)^*(y-x)\d \pi(x,y) \leq \liminf_{K\to\infty} \E\Ll[(t\xi)^*(Y_K-X_K)\Rr].
\end{align*}
Applying these convergences to~\eqref{e.E[S_tchi(X_K)]...}, we get 
\begin{align*}
    \int S_t\chi \d \mu  \leq \int \chi \d \pi_2 -\int(t\xi)^*(y-x)\d \pi(x,y).
\end{align*}
This gives ``$\leq$'' in~\eqref{e.intS_tchi=sup_piinPi} and completes the proof.
\end{proof}

We now show \eqref{e.kantorovich.potential}, which can be interpreted as saying that $\chi$ is a Kantorovich potential from $\mu$ to $\bar \nu$ (see Definition~\ref{d.kantorovich.potential}).
\begin{lemma}\label{l.rel_at_max}
Let $(t,\mu)\in\R_+\times\mcl P_1(\R_+^D)$ and let $\bar \nu$ be a maximizer of $h(t,\mu)$ in~\eqref{e.h(t,mu)=} (which exists by Lemma~\ref{l.exist_max}). Let $\chi \in\mathfrak X$ be the derivative of $\psi$ at $\bar \nu$ given as in Proposition~\ref{p.diff of psi}. We have
\begin{align}\label{e.l.rel_at_max}
    \int \chi \d \bar \nu -\int S_t\chi\d\mu = \mcl T_t(\mu,\bar \nu).
\end{align}
\end{lemma}

\begin{proof}
For every $\nu \in \mcl P_1(\R_+^D)$ and $\lambda\in (0,1]$, the convexity of $\mcl T_t(\mu,\scdot)$ implies
\begin{align*}
    \mcl T_t(\mu,\nu) -\mcl T_t(\mu,\bar \nu) 
    &\geq \frac{1}{\lambda}\Ll(\mcl T_t(\mu,\bar \nu+\lambda(\nu-\bar \nu)) -\mcl T_t(\mu,\bar \nu)\Rr)
    \\
    &\geq  \frac{1}{\lambda}\Ll(\psi(\bar \nu+\lambda(\nu-\bar \nu)) -\psi(\bar \nu)\Rr)
\end{align*}
where the second inequality follows from maximality at $\bar \nu$ of the formula for $h(t,\mu)$ (see~\eqref{e.h(t,mu)=}).
Using the differentiability of $\psi$ in Proposition~\ref{p.diff of psi} and sending $\lambda$ to zero, we get
\begin{align*}
    \mcl T_t(\mu,\nu) -\mcl T_t(\mu,\bar \nu)\geq \int \chi\d (\nu - \bar \nu)
\end{align*}
for any $\nu \in \mcl P_1(\R_+^D)$, which implies
\begin{align*}
    \int \chi \d \bar \nu - \mcl T_t(\mu,\bar \nu) = \sup_{\nu \in \mcl P_1(\R_+^D)}\Ll\{\int \chi \d \nu - \mcl T_t(\mu,\nu)\Rr\}.
\end{align*}
By Lemma~\ref{l.hopf-lax with linear initial condtion}, we can recognize the right-hand side as $\int S_t\chi \d \mu$, which verifies~\eqref{e.l.rel_at_max}. 
\end{proof}

We are now ready to show that the variational problem in \eqref{e.convexoptimization2} is dual to that in \eqref{e.convexoptimization}, that is, that $g = h$. 
\begin{proof}[Proof of Proposition~\ref{p.reuninverted}]
We will proceed by double inequality. 

First, we show $h \geq g$. Allowed by Lemma~\ref{l.exist_max}, let $\bar \nu \in \mcl P_1(\R_+^D)$ be such that 
\begin{e}\label{e.t.reuninverted.1}
    h(t,\mu) = \psi(\bar \nu) - \mcl T_t(\mu,\bar \nu).
\end{e}
By Lemma~\ref{l.rel_at_max}, we have
\begin{e*}
    \mcl T_t(\mu,\bar \nu) = \int \chi \d \bar \nu - \int S_t \chi \d \mu
\end{e*}
where $\chi$ is the derivative of $\psi$ at $\bar \nu$ given as in Proposition~\ref{p.diff of psi}.
The above two displays together yield
\begin{e}\label{e.h=psi-chi+Schi}
    h(t,\mu) = \psi(\bar \nu) - \int \chi d\bar \nu + \int S_t \chi \d\mu.
\end{e}
On the other hand, the concavity of $\psi$ given in Lemma~\ref{l.psi is concave} implies that, for any $\nu\in\mcl P_1(\R_+^D)$ and $\lambda \in (0,1]$, we have 
\begin{align*}
    \frac{1}{\lambda}\Ll(\psi\Ll(\lambda\nu+(1-\lambda)\bar \nu\Rr)-\psi(\nu)\Rr) \geq \psi(\nu)-\psi(\bar \nu).
\end{align*}
Sending $\lambda\to0$ and using Proposition~\ref{p.diff of psi}, for the same $\chi$ as above, we get $\int \chi \d (\nu - \bar \nu) \geq \psi(\nu) - \psi(\bar \nu)$.
Rearranging and taking the infimum over $\nu\in\mcl P_1(\R_+^D)$, we obtain $\psi_*(\chi) = \int \chi d\bar \nu - \psi(\bar \nu)$ due to its definition in~\eqref{e.psi_*=}.
This along with~\eqref{e.h=psi-chi+Schi} gives
\begin{e}\label{e.t.reuninverted.2}
    h(t,\mu) = \int S_t \chi \d \mu - \psi_*(\chi) \stackrel{\eqref{e.g(t,mu)=}}{\geq} g(t,\mu).
\end{e}
Next, we show $g\geq h$. 
According to the Arzelà--Ascoli theorem, $\mfk X$ is compact for the topology of local uniform convergence. We can thus use Sion's min-max theorem to obtain that
\begin{align*}
    g(t,\mu) &= \inf_{\chi \in \mathfrak X} \left\{ \int S_t \chi \d \mu - \psi_*(\chi) \right\} \\
             &= \inf_{\chi \in \mathfrak X} \sup_{\nu \in \mcl P_1(\R_+^D)} \left\{ \int S_t \chi \d \mu - \int \chi \d \nu + \psi(\nu) \right\} \\ 
             &= \sup_{\nu \in \mcl P_1(\R_+^D)} \inf_{\chi \in \mathfrak X}  \left\{ \int S_t \chi \d \mu - \int \chi \d \nu + \psi(\nu) \right\} \\
             &= \sup_{\nu \in \mcl P_1(\R_+^D)} \left\{\psi(\nu) - T_t(\mu,\nu) \right\},
\end{align*}
where $T_t(\mu,\nu) = \sup_{\chi \in \mathfrak X} \left\{ \int \chi \d \nu - \int S_t \chi \d \mu \right\}$ is to be compared with $\mcl T_t(\mu,\nu)$ defined in~\eqref{e.T_t=}.
The definition of $S_t\chi$ in~\eqref{e.real.def.Stchi} implies that $\chi(x)-S_t\chi(y)\leq (t\xi)^*(y-x)$ for every $x,y$.
Then, the Kantorovich duality \eqref{e.kantorovich.dual.rep} implies $T_t(\mu,\nu) \leq \mcl T_t(\mu,\nu)$ and thus 
\begin{e*}
    g(t,\mu) \geq \sup_{ \nu \in \mcl P_1(\R_+^D)} \left\{ \psi(\nu) - \mcl T_t(\mu,\nu) \right\} \stackrel{\eqref{e.h(t,mu)=}}{=} h(t,\mu),
\end{e*}
which completes the proof of $h=g$.
The additional statements on optimizers follow from $h=g$, \eqref{e.t.reuninverted.1}, and~\eqref{e.t.reuninverted.2}. 
\end{proof}

To close this section, we record the following useful result.

\begin{lemma}[Lipschitzness of $h$]\label{l.Lip_h}
Let $h$ be given as in~\eqref{e.h(t,mu)=}. For every $t,t'\in\R_+$ and $\mu,\mu'\in \mcl P_1(\R_+^D)$, we have
\begin{align}
    \Ll|h(t,\mu) - h(t',\mu')\Rr|\leq C\Ll|t-t'\Rr| + \msf W_1(\mu,\mu'),
\end{align}
where $C=\sup_{z\in\R_+^D:\: |z|\leq 1}|\xi(z)|$.
\end{lemma}
\begin{proof}
By Proposition~\ref{p.reuninverted}, we can work with $g$ given by~\eqref{e.g(t,mu)=} instead of $h$. Fix any $t,t',\mu,\mu'$. By the same proposition, there is $\chi\in \mathfrak X$ such that $g(t,\mu) = \int S_t\chi\d\mu- \psi_*(\chi)$. From the definition of $\mathfrak X$ in~\eqref{e.frakX=}, we see that $\|\chi\|_\mathrm{Lip}\leq 1$. Let $\pi\in\Pi(\mu,\mu')$ be an optimal coupling for $\msf W_1(\mu,\mu')$ and let $(X,X')$ satisfy $\mathrm{Law}(X,X')=\pi$. We have
\begin{align*}
    g(t',\mu') -g(t,\mu) \stackrel{\eqref{e.g(t,mu)=}}{\leq}\int S_{t'}\chi \d \mu' - \int S_t\chi \d \mu
    \leq \E \Ll[\Ll|S_{t'}\chi(X')-S_t\chi(X)\Rr|\Rr] 
    \\
    \leq C|t-t'|+  \E \Ll[|X'-X|\Rr] =C|t-t'| + \msf W_1(\mu,\mu'),
\end{align*}
where the last inequality follows from Lemma~\ref{l.S_tLip}. This implies the desired result.
\end{proof}

%
%
%
%
%
%

\section{Proofs of the main results} 
\label{s.proofs.main}

In this section, we complete the proofs of Theorems~\ref{t.unique_parisi}, \ref{t.convexoptimization}, and \ref{t.hopflike}. We also spell out the validity of Theorem~\ref{t.parisi.basic}. Recall from \eqref{e.h(t,mu)=} that we have set, for every $(t,\mu) \in \R_+ \times \mcl P_1(\R_+^D)$,
\begin{equation*}  
h(t,\mu) = \sup_{ \nu \in \mcl P_1(\R_+^D)} \left\{ \psi(\nu) - \mcl T_t(\mu,\nu) \right\}.
\end{equation*}
Recall also the notation $\cL_\sq$ from \eqref{e.L_q=}, and that for every $\mu \in \mcl P^\upa(\R^D_+)$, there exists a unique $\sq \in \mcl Q$ such that $\mu = \mcl L_\sq$. 
The key step towards the derivation of our main results is as follows. 
\begin{proposition}\label{p.h=hopf-lax_on_Q}
Assume that 
$\xi$ is strongly convex on $\R^D_+$. 
Let $h$ be given by~\eqref{e.h(t,mu)=}. For every $t\in\R_+$, $\sq \in \mcl Q_\infty$, and for  $\mu =\cL_\sq\in \mcl P^\upa_\infty(\R_+^D)$, we have
\begin{align}
    h(t,\mu) &= \sup_{\nu\in\mcl P^\upa_\infty(\R_+^D)}\Ll\{\psi(\nu) - \mcl T_t(\mu,\nu)\Rr\}\label{e.p.h=hopf-lax_on_Q(1)}
    \\
    &= \sup_{\sp \in \mcl Q_\infty}\Ll\{\psi \Ll(\cL_{\sq+t\nabla\xi(\sp)}\Rr) - t\int_0^1 \theta\, \d\cL_\sp\Rr\}. \label{e.p.h=hopf-lax_on_Q(2)}
\end{align}
Moreover, the suprema in~\eqref{e.p.h=hopf-lax_on_Q(1)} and~\eqref{e.p.h=hopf-lax_on_Q(2)} are achieved.
\end{proposition}

As discussed below \eqref{e.kantorovich.potential}, the key point in showing \eqref{e.p.h=hopf-lax_on_Q(1)} is to transfer the monotonicity of $\mu$ into the monotonicity of a maximizer in the definition of $h(t,\mu)$. In order to do so, we will smear out the measure $\mu$ so that it becomes absolutely continuous with respect to the Lebesgue measure. The smeared-out measure is however no longer monotone, only approximately so, and our first task is to develop convenient tools that allow us to track approximate and exact monotonicity of a measure.

\begin{lemma}[Characterization of monotonicity]\label{l.monotone}
Let $d\in\N$ and $\mu \in \mcl P(\R_+^d)$. The measure $\mu$ is monotone if and only if, for every $i,j\in\{1,\dots, d\}$, 
\begin{align}\label{e.l.monotone}
    \mu^{\otimes 2}\Ll(\Ll\{\, (y,y')\in\R_+^d\times \R_+^d:\:\ y_i\geq y'_i\ \text{ or }\ y_j\leq y'_j\, \Rr\}\Rr) =1.
\end{align}
\end{lemma}
\begin{remark}
Let $Y$ and $Y'$ two independent random variables with law $\mu$. Then, \eqref{e.l.monotone} can be rewritten as $\P(Y_i\geq Y'_i\text{ or } Y_j\leq Y'_j)=1$. \qed
\end{remark}

\begin{proof}[Proof of Lemma~\ref{l.monotone}]
We proceed in two steps. First, we show that~\eqref{e.l.monotone} is equivalent to a stronger formulation and then we apply results from~\cite[Section~2]{mourrat2020free}.

\medskip

\noindent \emph{Step~1.}
We show that the condition specified in~\eqref{e.l.monotone} is equivalent to 
\begin{align}\label{e.l.monotone3}
    \mu^{\otimes 2}\Ll(\Ll\{\, (y,y')\in \R_+^d\times \R_+^d:\:\ a \cdot y\geq a \cdot y'\ \text{ or }\ b\cdot y\leq b\cdot y'\, \Rr\}\Rr) =1
\end{align}
for every $a,b\in\R_+^d$.
One direction is obvious and we focus on deducing~\eqref{e.l.monotone3} from~\eqref{e.l.monotone}.

It is more convenient to work with random variables and we start with some notation. For a pair $(U,V)$ of real-valued random variables, we say that $(U,V)$ is monotone if, given an independent copy $(U',V')$, we have almost surely that $\{U\geq U'$ or $V\leq V'\}$. 
It is clear that $(U,V)$ is monotone if and only if $(V,U)$ is so.
We need the following property. Given three random variables $U,V,W$, we have, for every $s,t\geq0$,
\begin{align}\label{e.pairwise_monotone}
    \text{if both $(U,W)$ and $(V,W)$ are monotone, then so is $(sU+tV, W)$.}
\end{align}
This immediately follows from the definition. Indeed, let $(U',V',W')$ be an independent copy and there is nothing to show when $W\leq W'$. Otherwise, we must have $U\geq U'$ and $V\geq V'$, which implies $sU+tV\geq sU'+tV'$.

Let $Y$ be a random vector in $\R^d$ with $\mu =\mathrm{Law}(Y)$. The condition in~\eqref{e.l.monotone} is equivalent to that $(Y_i,Y_j)$ is monotone for every $1\leq i,j\leq d$. Fix any $a,b\in\R^d_+$. For each $j$, iteratively applying~\eqref{e.pairwise_monotone} to pairs $(Y_i,Y_j)$ for $1\leq i\leq d$, we can get that $(a\cdot Y,Y_j)$ is monotone. Similarly, we can get that $(a\cdot Y, b\cdot Y)$ is monotone, which is equivalent to~\eqref{e.l.monotone3}.

\medskip

\noindent \emph{Step~2.}
We use results from~\cite[Section~2]{mourrat2020free} to conclude. We start with some definitions.
Let $S^d$ be the linear space of real symmetric matrices endowed with the Frobenius inner product, namely, $a\cdot b= \sum_{i,j=1}^da_{ij}b_{ij}$ for $a,b\in S^d$. Let $S^d_+$ be the set of positive semi-definite matrices. Let $\bar\mu$ be the image of $\mu$ through the map $\R^d\ni x\mapsto \mathrm{diag}(x_1,\dots,x_d)$. 
Since $\bar\mu$ is supported on diagonal matrices and diagonal entries of $a,b\in  S^d_+$ are nonnegative, we can see that~\eqref{e.l.monotone3} is equivalent to
\begin{align}\label{e.l.monotone2}
    \bar\mu^{\otimes 2}\Ll(\Ll\{\, (y,y')\in S_+^d\times S_+^d:\:\ a \cdot y\geq a \cdot y'\ \text{ or }\ b\cdot y\leq b\cdot y'\, \Rr\}\Rr) =1
\end{align}
for every $a,b\in S^d_+$.

Monotone probability measures on $S^d_+$ are defined in the paragraphs above~\cite[Proposition~2.3]{mourrat2020free} and~\cite[(2.1)]{mourrat2020free}. The condition in~\eqref{e.l.monotone2} exactly matches this definition and thus $\bar\mu$ is monotone on $S^d_+$. By~\cite[Proposition~2.4]{mourrat2020free}, this monotonicity is equivalent to the existence of an increasing path $\bar \sq:[0,1)\to S^d_+$ such that $\bar\mu = \mathrm{Law}(\bar\sq(U))$, where the monotonicity of $\bar\sq$ is understood as $\bar\sq(s')-\bar\sq(s)\in S^d_+$ whenever $s'\geq s$. Since $\bar\mu$ is supported on diagonal matrices, $\bar\sq$ must take values in diagonal matrices. Therefore, $\bar\mu = \mathrm{Law}(\bar\sq(U))$ is equivalent to $\mu = \mathrm{Law}(\sq(U))$ for some increasing $\sq:[0,1)\to\R^d$, where $\sq$ and $\bar\sq$ are related by $\mathrm{diag}(\bar \sq)=\sq$ a.e.\ on $[0,1)$. Combining these equivalences, we obtain the desired characterization.
\end{proof}

The next lemma is tailored for application to our future context involving a smeared-out version of the measure $\mu \in \mcl P^\upa_\infty(\R^D_+)$, where we later let the smearing-out parameter tend to zero. Recall that $U$ is a random variable with uniform distribution over $[0,1]$.

\begin{lemma}\label{l.limit_monotone}
Let $\mu = \cL_\sq\in \mcl P^\upa_\infty(\R_+^D)$ (see~\eqref{e.L_q=}) with $\sq\in\mcl Q_\infty$ strictly increasing in the sense that $\sq(s)-\sq(s')\in(0,\infty)^\D$ whenever $s>s'$. 
For each $n\in\N$, let $\varpi_n \in \mcl P(\R^{2D}_+)$ and suppose that there is an $\R_+^D$-increasing measurable map $T_n:\R_+^D\to \R_+^D$ such that
\begin{align}\label{e.l.limit_monotone}
    x' = T_n(x) \quad \text{for $\varpi_n$-a.e.\ $(x,x')$}.
\end{align}
Assume that $\varpi_n$ converges weakly to some $\varpi  \in \mcl P(\R_+^{2D})$ whose projection on the first $D$ coordinates is equal to $\mu$. Then, we have $\varpi \in \mcl P^\upa(\R_+^{2D})$ and thus there is $\sp\in\mcl Q$ such that $\varpi = \mathrm{Law}(\sq(U),\sp(U))$.
\end{lemma}

\begin{proof}
In the following, we denote by $(y,y')$ an element in $\R^{2D}_+\times \R^{2D}_+$.
For $i,j\in\{1,\dots,2D\}$, we consider the following events on $\R^{2D}_+\times \R^{2D}_+$,
\begin{align*}
    E_{i,j} = \Ll\{ y_i\geq y'_i \ \text{or}\ y_j\leq y'_j\Rr\}.
\end{align*}
We also set
\begin{align*}
    E_i = \bigcap_{j=1}^D E_{i,j},\text{ for } i \in\{1,\dots ,D\}\quad \text{ and } E =\bigcap_{i=1}^D E_i = \bigcap_{i,j=1}^D E_{i,j}.
\end{align*}
For each $n$, let $\mu_n \in \mcl P(\R_+^D)$ be the projection of $\varpi_n$ to the first $D$ coordinates. For $i,j\leq D$, $E_{i,j}$ can be measured by $\mu_n$. Then, we immediately have
\begin{align*}
    \varpi_n^{\otimes 2}\Ll(E_{i,j}\Rr)=\mu_n^{\otimes 2}(E_{i,j}) \quad \text{if $i\leq j\leq D$}.
\end{align*}
When $i \leq D<j$, by~\eqref{e.l.limit_monotone}, we have $y_j = \big(T_n(y_{[D]})\big)_j$ for $\varpi_n$-a.e.\ $y$, where $y_{[D]} = (y_k)_{1\leq k\leq D}$. Since $T_n$ is assumed to be $\R_+^D$-increasing, for $\varpi_n^{\otimes 2}$-a.e.\ $(y,y')$, we have that $y_{[D]}\leq y'_{[D]}$ implies $y_j\leq y'_j$. On the event $E_i$ we have $y_i \geq y'_i$ or $y_{[D]}\leq y'_{[D]}$. In the latter case, it follows from the previous observation that $y_j\leq y'_j$ $\varpi^{\otimes 2}_n$-almost surely. Hence, we have $E_i\subset E_{i,j}$ and thus
\begin{align*}
    \varpi_n^{\otimes 2}\Ll(E_{i,j}\Rr)\geq  \varpi_n^{\otimes 2}\Ll(E_i\Rr) = \mu_n^{\otimes 2}(E_i) \quad \text{if $i \leq D<j$}.
\end{align*}
Similarly on the event $E$ we have $y_{[D]} \leq y'_{[D]}$ or $y_{[D]} \geq y'_{[D]}$ and by~\eqref{e.l.limit_monotone}, we can also deduce
\begin{align*}
    \varpi_n^{\otimes 2}\Ll(E_{i,j}\Rr)\geq \mu_n^{\otimes 2}(E) \quad \text{if $D<i \leq j$}.
\end{align*}
Using the above three displays and the symmetry $\varpi_n^{\otimes 2}\Ll(E_{i,j}\Rr)= \varpi_n^{\otimes 2}\Ll(E_{j,i}\Rr)$, we obtain
\begin{align}\label{e.varpi^2>mu}
    \varpi_n^{\otimes 2}\Ll(E_{i,j}\Rr)\geq \mu_n^{\otimes 2}(E),\quad\forall i,j\in\{1,\dots,2D\}.
\end{align}
Notice that $\partial E \subset \bigcup_{i,j=1}^D\partial E_{i,j} = \bigcup_{i=1}^D\{y_i=y'_i\}$.
Since $\sq$ as in  $\mu=\cL_\sq$ is strictly increasing, we can get
\begin{align*}
    \mu^{\otimes2}\Ll(\{y_i=y'_i\}\Rr)= \iint_{[0,1]^2}\mathds{1}_{\sq_i(s)=\sq_i(s')}\d s \d s'=0.
\end{align*}
Hence, $E$ is a continuity set of $\mu^{\otimes2}$. Since $\mu\in \mcl P^\upa_\infty(\R_+^D)$ is monotone, by Lemma~\ref{l.monotone}, we have that $\mu^{\otimes2}(E)=1$. Also, notice that $E_{i,j}$ is a closed set.
Using these and the Portmanteau theorem, we pass to the limit in~\eqref{e.varpi^2>mu} to get
\begin{align*}
    \varpi^{\otimes2}(E_{i,j}) \geq \limsup_{n\to\infty}\varpi_n^{\otimes2}(E_{i,j}) \geq \lim_{n\to\infty}\mu_n^{\otimes2}(E) = \mu^{\otimes2}(E)=1
\end{align*}
for every $i,j$.
Therefore, by Lemma~\ref{l.monotone}, $\varpi$ is monotone. Since the first marginal of $\varpi$ is $\mu=\cL_\sq$, we must have $\varpi = \mathrm{Law}(\sq(U),\sp(U))$ for some $\sp\in\mcl Q$.
\end{proof}

We are now ready to prove Proposition~\ref{p.h=hopf-lax_on_Q}.

\begin{proof}[Proof of Proposition~\ref{p.h=hopf-lax_on_Q}]

We need approximations of $\mu \in \mcl P^\upa_\infty(\R^D_+)$ from the collection $\mcl P_\mathrm{ac}(\R_+^D)$ of probability measures on $\R_+^D$ that are absolutely continuous with respect to the Lebesgue measure.
Let $(\mu_n)_{n\in\N}$ be a sequence in $\mcl P_\mathrm{ac}(\R_+^D)$ that converges weakly to $\mu$. Since $\mu$ is compactly supported, we may assume that the measures $\mu_n$ are all supported in some fixed ball (independent of $n$). Hence, by Lemma~\ref{l.Lip_h},
\begin{align}\label{e.limh(mu_n)=h(mu)}
    \lim_{n\to\infty} h(t,\mu_n) = h(t,\mu).
\end{align}
Allowed by Lemma~\ref{l.exist_max}, for each $n\in\N$, let $\bar \nu_n$ be a maximizer of $h(t,\mu_n)$ in~\eqref{e.h(t,mu)=} and let $\pi_n$ be an optimal coupling of $(\mu_n,\bar \nu_n)$ for $\mcl T_t(\mu_n,\bar \nu_n)$.
Hence, we have
\begin{align}
    h(t,\mu_n) &= \psi(\bar \nu_n) -\mcl T_t(\mu_n,\bar \nu_n),\label{e.h(t,mu_n)=psi(nu^*_n)...}
    \\
    \mcl T_t(\mu_n,\bar \nu_n) &= \int (t\xi)^*(y-x)d \pi_n(x,y). \label{e.T_t(mu,nu)=intxi^*dpi_n}
\end{align}
The plan is to show that $\bar \nu_n$ is approximately monotone. Then, by taking limits, we expect that the formula for $h(t,\mu)$ maximizes over monotone measures. 
By Lemma~\ref{l.rel_at_max}, there is $\chi_n \in \mathfrak X$ (see~\eqref{e.frakX=}) such that
\begin{align}\label{e.intchi_n-intS_tchi_n=T}
    \int \chi_n \d \bar \nu_n -\int S_t\chi_n\d\mu_n = \mcl T_t(\mu_n,\bar \nu_n).
\end{align}
Note that according to Lemma~\ref{l.xi^*.diff}, thanks to the assumption on $\xi$, the function $\xi^*$ is differentiable on $\R^D$. We proceed in five steps. 

\noindent \emph{Step~1.}
Let $(X_n,Y_n)$ be a random variable with law $\pi_n$. We show that possibly up to a redefinition of $\pi_n$, we may assume without loss of generality that $Y_n - X_n \in \nabla \xi(\R^D_+)$ almost surely. 

Define
\begin{e*}
   Y_n' = \nabla \xi(\nabla \xi^*((Y_n-X_n)_+)) + X_n,
\end{e*}
let $\bar \nu_n'$ denote the law of $Y_n'$ and let $\pi_n'$ denote the law of $(X_n, Y_n')$. Since $x_+ - x \in \R^D_+$, according to \eqref{e.KKT1} in Lemma~\ref{l.save_the_day} we have that $Y_n'-Y_n \in \R^D_+$ almost surely. Hence, by monotonicity of $\psi_d$ (which we can borrow from \cite[Proposition~3.6]{chenmourrat2023cavity} or rederive using Proposition~\ref{p.diff of psi} and the fact that functions in $\mfk X$ are increasing), we have 
$\psi(\bar \nu_n') \geq \psi(\bar \nu_n)$. In addition, according to Lemma~\ref{l.H^*(x)=H^*(x_+)} and \eqref{e.xi^*(nabla_xi(nabla_xi^*))} in Lemma~\ref{l.save_the_day}, we have $\E[(t\xi)^*(Y'_n - X_n)]= \E[(t\xi)^*(Y_n - X_n)]$, so
\begin{equation*}
    \mcl T_t(\mu_n, \bar \nu_n') \leq \int (t\xi)^*(y-x) \d \pi_n'(x,y) = \int (t\xi)^*(y-x) \d \pi_n(x,y) = \mcl T_t(\mu_n, \bar \nu_n).
\end{equation*}
We deduce that
\begin{align*}
    \psi(\bar \nu_n') - \mcl T_t(\mu_n,\bar \nu_n') &\geq \psi(\bar \nu_n') -  \int (t\xi)^*(y-x) \d \pi_n'(x,y)  \\
                                                    &\geq \psi(\bar \nu_n) - \mcl T_t(\mu_n,\bar \nu_n) \\
                                                    &= h(t,\mu_n) \\
                                                    &\geq \psi(\bar \nu_n') -  \mcl T_t(\mu_n,\bar \nu_n'),
\end{align*}
where the last inequality follows from \eqref{e.h(t,mu)=}. The above display implies that $\bar \nu_n'$ is a maximizer in the definition \eqref{e.h(t,mu)=} of $h(t,\mu_n)$ and that $\pi_n'$ is an optimal coupling for $\mcl T_t(\mu_n,\bar \nu')$. In addition, if $(X'_n,Y_n')$ is a random variable with law $\pi_n'$, we have by definition that $Y_n' - X_n' \in \nabla \xi(\R^D_+)$ almost surely, this means that up to replacing $(\pi_n,\bar \nu_n)$ by $(\pi_n',\bar \nu_n')$ (and redefining $\chi_n$ accordingly), we may assume that $Y_n - X_n \in \nabla \xi(\R^D_+)$ almost surely.

\medskip

\noindent \emph{Step~2.}
We now show that 
\begin{align}\label{e.y=T_n(x)}
    y = x + t\nabla\xi\Ll(\nabla S_t\chi_n(x)\Rr) \quad \text{for $\pi_n$-a.e. $(x,y)$}.
\end{align}
This implicitly states that $S_t\chi_n$ is differentiable at such points.
Using~\eqref{e.T_t(mu,nu)=intxi^*dpi_n} and~\eqref{e.intchi_n-intS_tchi_n=T}, we have
\begin{align*}
    \int \Ll(\chi_n(y) - S_t\chi_n(x) -(t\xi)^*(y-x)\Rr)\d \pi_n(x,y) = 0.
\end{align*}
By the formula for $S_t\chi_n$ in~\eqref{e.real.def.Stchi}, the integrand in the above display is non-positive. Therefore, we must have
\begin{align*}
    \chi_n(y) - S_t\chi_n(x) -(t\xi)^*(y-x)=0, \quad \text{for $\pi_n$-a.e. $(x,y)$}.
\end{align*}
Since $S_t\chi_n$ is Lipschitz due to Lemma~\ref{l.S_tLip}, $S_t\chi_n$ is differentiable almost everywhere with respect to the Lebesgue measure (see e.g.\ \cite[Theorem~2.10]{HJbook}). Hence, due to $\mu_n \in \mcl P_\mathrm{ac}(\R_+^D)$ (the first marginal of $\pi_n$), we have that for $\pi_n$-a.e.\ $(x,y)$, $S_t\chi_n$ is differentiable at $x$. 
By~\eqref{e.l.nablaS_tchi_and_nablachi(1)} in Proposition~\ref{p.nablaS_tchi_and_nablachi}, we get
\begin{align*}
    \nabla S_t\chi_n(x) =\nabla(t\xi)^*(y-x), \quad \text{for $\pi_n$-a.e. $(x,y)$}.
\end{align*}
Recalling from Step 1 that $y - x \in \nabla \xi(\R^D_+)$ for $\pi_n$-a.e.\ $(x,y)$, and using Remark~\ref{r.bla} we can apply $t\nabla \xi$ to both sides in the above and obtain~\eqref{e.y=T_n(x)}.
Since $S_t\chi_n$ is $\R_+^D$-convex due to Proposition~\ref{p.R^D_+-convexity}, we have that
\begin{align}\label{e.nablaS_tchi_n_increasing}
    \text{$\nabla S_t\chi_n$ is $\R_+^D$-increasing}
\end{align}
which will be used later.

\medskip

\noindent \emph{Step~3.}
We show that $(\pi_n)_{n\in\N}$ is tight and derive the limit of~\eqref{e.h(t,mu_n)=psi(nu^*_n)...}.
Let random variables $(X_n,Y_n,Z_n)$ satisfy
\begin{align}\label{e.Z_n=nabla...}
    \mathrm{Law}(X_n,Y_n) = \pi_n;\qquad Z_n = \nabla S_t\chi_n(X_n).
\end{align}
By~\eqref{e.y=T_n(x)}, we have
\begin{align}\label{e.Y_n=X_n+...}
    Y_n = X_n + t\nabla\xi(Z_n)\quad\text{a.s.}
\end{align}
Recall that $\chi_n$ given by Proposition~\ref{p.diff of psi} belongs to the collection $\mathfrak X$ defined in~\eqref{e.frakX=}, whose definition implies that $\chi_n$ is $1$-Lipschitz. 
By Lemma~\ref{l.S_tLip}, $S_t\chi_n$ is also $1$-Lipschitz. Setting $C= t\sup_{|z|\leq 1}|\nabla\xi(z)|$, we have 
\begin{align}\label{e.X_n,Y_n,Z_n_bdd}
    |Z_n|\leq 1 \quad\text{and}\quad \Ll|Y_n-X_n\Rr|\leq C \quad \text{a.s.}
\end{align}
Since $\mu_n = \mathrm{Law}(X_n)$ and $\mu_n$ is assumed to be supported in a fixed ball, by~\eqref{e.X_n,Y_n,Z_n_bdd}, $Y_n$ and $Z_n$ take values in a fixed compact set.  
In particular, the family $(X_n,Y_n,Z_n)_{n\in\N}$ is tight. 
By passing to a subsequence and invoking Skorokhod's representation theorem, we may assume that $(X_n,Y_n,Z_n)$ converges almost surely to some $(X,Y,Z)$, and therefore also in $L^1$ since these random variables are bounded uniformly over $n$. 
Since $Y_n$ is a continuous function of $X_n$ and $Z_n$ as in~\eqref{e.Y_n=X_n+...}, we also have
\begin{align}\label{e.Y=X+...}
    Y = X + t\nabla\xi(Z)\quad\text{a.s.}
\end{align}
Since $\mu_n = \mathrm{Law}(X_n)$ and $\mu_n$ converges to $\mu$ in $\mcl P_1(\R_+^D)$, we must have $\mu=\mathrm{Law}(X)$. 
We set
\begin{align}\label{e.nu^*=Law(Y)}
    \bar \nu = \mathrm{Law}(Y) \quad\text{and}\quad \pi = \mathrm{Law}(X,Y)\in\Pi(\mu,\bar \nu). 
\end{align} 
By~\eqref{e.T_t(mu,nu)=intxi^*dpi_n} and Fatou's lemma (recall that $(t\xi)^*\geq -t\xi(0) = 0$ and $(t\xi)^*$ is lower semi-continuous), we have
\begin{align*}
    \mcl T_t(\mu,\bar \nu) \leq \E \Ll[(t\xi)^*(Y-X)\Rr]\leq \liminf_{n\to\infty} \E \Ll[(t\xi)^*(Y_n-X_n)\Rr] = \liminf_{n\to\infty}\mcl T_t(\mu_n,\bar \nu_n).
\end{align*}
This along with the continuity of $\psi$ in $\mcl P_1(\R_+^D)$ in~\eqref{e.Lip_psi} implies that
\begin{align*}
    \psi(\bar \nu) -\mcl T_t(\mu,\bar \nu)&\geq \psi(\bar \nu) - \E \Ll[(t\xi)^*(Y-X)\Rr]
    \\
    &\geq \liminf_{n\to\infty}\psi(\bar \nu_n)-\mcl T_t(\mu_n,\bar \nu_n)\stackrel{\eqref{e.limh(mu_n)=h(mu)}\eqref{e.h(t,mu_n)=psi(nu^*_n)...}}{=}h(t,\mu)
    \\
    &\stackrel{\eqref{e.h(t,mu)=}}{\geq} \psi(\bar \nu) -\mcl T_t(\mu,\bar \nu).
\end{align*}
This together with~\eqref{e.Y=X+...} yields
\begin{align}
    h(t,\mu) & = \psi(\bar \nu) - \mcl T_t(\mu,\bar \nu)\label{e.h(t,mu)=...pi...}
    \\
    &= \psi\Ll(\mathrm{Law}(Y)\Rr)-\E\Ll[(t\xi^*)(t\nabla\xi(Z))\Rr].\label{e.h(t,mu)=...pi...2}
\end{align}

\medskip

\noindent \emph{Step~4.}
We show the desired results under the additional assumption that
\begin{align}\label{e.q_strict_incr}
    \mu = \cL_\sq \quad\text{ with $\sq$ strictly increasing}
\end{align}
in the sense that $\sq(s)-\sq(s')\in(0,\infty)^\D$ whenever $s>s'$.
Assume~\eqref{e.q_strict_incr},
we want to show that there is $\sp_*\in \mcl Q_\infty$ such that
\begin{align}\label{e.(X,Y,Z)=}
    (X,Y,Z)\stackrel{\d}{=}\Ll(\sq(U),\, \sq(U)+t\nabla\xi(\sp_*(U)),\, \sp_*(U)\Rr)
\end{align}
where $U$ is uniform over $[0,1]$. Due to~\eqref{e.Y=X+...}, we only need to find the representation for $(X,Z)$. 
The assumption in~\eqref{e.q_strict_incr} together with \eqref{e.nablaS_tchi_n_increasing} and~\eqref{e.Z_n=nabla...} allow us to apply Lemma~\ref{l.limit_monotone} with $\varpi_n$, $T_n$, $\varpi$ substituted by $\mathrm{Law}(X_n,Z_n)$, $\nabla S_t\chi_n$, and $\mathrm{Law}(X,Z)$, respectively. Consequently, $\mathrm{Law}(X,Z)$ is monotone. Hence, there must be an increasing path $\sp_*:[0,1)\to \R_+^D$ such that $(X,Z)\stackrel{\d}{=}(\sq(U),\sp_*(U))$.
Since $Z$ is bounded (as a consequence of~\eqref{e.X_n,Y_n,Z_n_bdd}), we have $\sp_\star\in\mcl Q_\infty$ and more precisely
\begin{align}\label{e.|p_*|<sqrtD}
    |\sp_\star(s)|\leq 1,\quad\forall s\in[0,1).
\end{align}

Due to $\bar \nu=\mathrm{Law}(Y)$ (see~\eqref{e.nu^*=Law(Y)}) and~\eqref{e.(X,Y,Z)=}, $\bar \nu$ is a monotone measure and has bounded support. In other words, $\bar \nu=\cL_{\sq+t\nabla\xi(\sp_\star)}\in\mcl P^\uparrow_1(\R_+^D)$.
In view of the definition of $h$ in~\eqref{e.h(t,mu)=} and~\eqref{e.h(t,mu)=...pi...}, we can deduce~\eqref{e.p.h=hopf-lax_on_Q(1)} and that the supremum is achieved at $\cL_{\sq+t\nabla\xi(\sp_\star)}$. 
From the definition of $\theta$ in~\eqref{e.theta=} and \eqref{e.xi^*(nabla_xi)=} in Lemma~\ref{l.save_the_day}, we have
that for every $x \in \R^D_+$,
\begin{align}\label{e.ttheta=(txi)^*(tnablaxi)}
    t\theta(x) = (t\xi)^*(t\nabla\xi(x)).
\end{align}
Using~\eqref{e.h(t,mu)=...pi...2}, \eqref{e.(X,Y,Z)=}, and~\eqref{e.ttheta=(txi)^*(tnablaxi)}, we have
\begin{align}
     h(t,\mu) = \psi\Ll(\cL_{\sq+t\nabla\xi(\sp_\star)}\Rr) -\int \theta\,\d\cL_{\sp_\star}.
\end{align}
Due to~\eqref{e.h(t,mu)=} and the definition of the optimal transport problem $\mcl T_t$ in~\eqref{e.T_t=}, we see that $h(t,\mu)$ is always an upper bound for the term in~\eqref{e.p.h=hopf-lax_on_Q(2)}. Then, the above display implies the equality in~\eqref{e.p.h=hopf-lax_on_Q(2)} and that the supremum is achieved at $\sp_\star$. 
Hence, we have shown~\eqref{e.p.h=hopf-lax_on_Q(1)} and~\eqref{e.p.h=hopf-lax_on_Q(2)} under the assumption~\eqref{e.q_strict_incr} and the suprema are achieved at $\cL_{\sq+t\nabla\xi(\sp_\star)}$ and $\sp_\star$, respectively.

\medskip

\noindent \emph{Step~5.}
We show the results for general $\mu= \cL_\sq$ without assuming~\eqref{e.q_strict_incr}.
We employ an approximation argument.
For $k\in\N$, define a strictly increasing path $\sq_k$ by setting $\sq_k(s)= \sq(s)+k^{-1}s\mathbf{1}$ for $s\in[0,1)$, where $\mathbf{1}=(1,1,\dots,1)\in\R_+^D$ consists of ones as entries. 
So, $\sq_k$ converges to $\sq$ pointwise.
We take $\mu_k = \cL_{\sq_k}\in\mcl P^\upa_\infty(\R_+^D)$ and thus $\mu_k$ converges weakly to $\mu$. 
Let $\sp_{\star,k}$ be the corresponding path appearing in~\eqref{e.(X,Y,Z)=}. Due to~\eqref{e.|p_*|<sqrtD}, $(\sp_{\star,k})_{k\in\N}$ is a family of increasing paths that are uniformly bounded. By~\cite[Lemma~3.4]{chenmourrat2023cavity}, this allows us to assume that $\sp_{\star,k}$ converges to $\sp_\star\in\mcl Q_\infty$ a.e.\ on $[0,1)$ after passing to a subsequence.

Then, \eqref{e.p.h=hopf-lax_on_Q(1)} and~\eqref{e.p.h=hopf-lax_on_Q(2)} follow from the following convergences. We have that $h(t,\mu_k)$ converges to $h(t,\mu)$ as $k\to\infty$ by Lemma~\ref{l.Lip_h}. For each $\nu\in\mcl P^\upa_\infty(\R_+^D)$, we have $\mcl T_t(\mu_k,\nu)$ converges to $\mcl T_t(\mu,\nu)$ due to boundedness of the measures and the weak convergence of $\mu_k$. By the continuity of $\psi$ as in~\eqref{e.Lip_psi}, $\psi \Ll(\cL_{\sq_k+t\nabla\xi(\sp)}\Rr) $ converges to $\psi \Ll(\cL_{\sq+t\nabla\xi(\sp)}\Rr) $ at any $\sp\in\mcl Q_\infty$.

Recall that, at each $\mu_k$, the relevant suprema are achieved at $\cL_{\sq_k+t\nabla\xi(\sp_{\star,k})}$ and $\sp_{\star,k}$, respectively.
By similar considerations, we can show that the suprema at $\mu$ are achieved at $\cL_{\sq+t\nabla\xi(\sp_\star)}$ and $\sp_\star$, respectively.
\end{proof}

Proposition~\ref{p.h=hopf-lax_on_Q} requires $\xi$ to be strongly convex on $\R^D_+$. Since this may not necessarily always be the case, we will introduce a small perturbation of $\xi$ that enforces this property, and argue by continuity to obtain the main results. 
For every $\lambda \geq 0$, we set 
\begin{e}
    \xi_\lambda(x) = \xi(x) + \frac{\lambda}{2}|x|^2.
\end{e}
We clearly have that for every $\lambda > 0$, $\xi_\lambda$ is strongly convex on $\R^D_+$. Setting 
\begin{e}
    H_N^\lambda(\sigma)=  H_N(\sigma) + \frac{\sqrt{\frac{\lambda}{2}}}{\sqrt{N}} \sum_{d = 1}^D \sum_{i,j \in I_{N,d}} J_{ij}^d \sigma_{i} \sigma_j,
\end{e}
where $J_{ij}^d$ are independent $\mcl N(0,1)$ random variables that are each independent of $H_N$, we have that
\begin{e}
    \E H^\lambda_N(\sigma) H^\lambda_N(\sigma') = N \xi_\lambda(R_N(\sigma,\sigma')).
\end{e}
We let $\bar F_N^\lambda$, $g^\lambda$, and $h^\lambda$ denote the quantities defined in \eqref{e.enriched free energy}, \eqref{e.g(t,mu)=}, and \eqref{e.h(t,mu)=} with $\xi$ and $H_N$ replaced by $\xi_\lambda$ and $H_N^\lambda$ therein. At $\lambda = 0$, we have $\xi_0 = \xi$ and $H_N^0 = H_N$, therefore $\bar F_N^0 = \bar F_N$, $g^0= g$, and $h^0 = h$. We now show the continuity of these quantities in $\lambda$.

\begin{lemma} \label{l.lambda.continuity}
    For every $(t,\mu) \in \R_+ \times \mcl P_1(\R^D_+)$, we have that the quantities $\lim_{N \to +\infty} \bar F_N^\lambda(t,\mu)$, $g^\lambda(t,\mu)$, and $h^\lambda(t,\mu)$ are continuous functions of $\lambda \in \R_+$.
\end{lemma}

\begin{proof}
    As in \cite[Proposition~6.1]{issa2024hopflike}, we can differentiate $\bar F_N^\lambda(t,\mu)$ with respect to $\lambda$ and use Gaussian integration by parts to observe that the derivative is bounded uniformly in $N$. This yields Lipschitz continuity in $\lambda$ of $\lim_{N \to +\infty} F_N^\lambda(t,\mu)$. 

    For every $\chi \in \mfk X$, we let $S_t^\lambda \chi$ denote the quantity obtained in \eqref{e.real.def.Stchi} with $\xi$ replaced by $\xi_\lambda$ therein. By Lemma~\ref{l.St1=St2} and since $\chi^*(p) = +\infty$ when $p \in \R^D_+ \setminus [0,1]^D$, we have 
    \begin{e*}
        S^\lambda_t \chi(x) = \sup_{p \in [0,1]^D} \left\{x \cdot p - \chi^*(p) + t \xi_\lambda(p) \right\}.
    \end{e*}
    This implies that
    \begin{e*}
        |S^\lambda_t \chi(x) - S^{\lambda'}_t \chi(x)| \leq \frac{Dt}{2}|\lambda - \lambda'|,
    \end{e*}
    and thus, for every $\mu \in \mcl P_1(\R^D_+)$, we have 
    \begin{e*}
        \left| \int S^\lambda_t \chi \d \mu - \int S^{\lambda'}_t \chi \d \mu \right| \leq \frac{Dt}{2}|\lambda - \lambda'|.
    \end{e*}
    The quantity $g^\lambda(t,\mu)$ is therefore Lipschitz continuous in $\lambda$ since using \eqref{e.g(t,mu)=} it can be written as the infimum of a family of uniformly Lipschitz functions of~$\lambda$. Finally, from Proposition~\ref{p.reuninverted} we have that $h^\lambda(t,\mu) = g^\lambda(t,\mu)$, so $h^\lambda(t,\mu)$ is also Lipschitz continuous in $\lambda$.
\end{proof}

We have the following immediate corollary.

\begin{corollary}\label{c.limF_N=h}
Let $(t,\mu) \in \R_+\times \mcl P^\upa_\infty(\R_+^D)$ and let $h(t,\mu)$ be given as in~\eqref{e.h(t,mu)=}. We have
$\lim_{N\to\infty}\bar F_N(t,\mu) = h(t,\mu)$.
\end{corollary}

\begin{proof}
Let $\lambda > 0$, since $\xi_\lambda$ is strongly convex on $\R^D_+$, it follows from Proposition~\ref{p.parisi_gen} and~\eqref{e.p.h=hopf-lax_on_Q(2)} in Proposition~\ref{p.h=hopf-lax_on_Q} that 
\begin{e*}
    \lim_{N\to\infty}\bar F^\lambda_N(t,\mu) = h^\lambda(t,\mu).
\end{e*}
We can then let $\lambda \to 0$ and use Lemma~\ref{l.lambda.continuity} to get the desired result.
\end{proof}

We can now complete the proofs of Theorems~\ref{t.convexoptimization} and \ref{t.hopflike}. 

\begin{proof}[Proof of Theorem~\ref{t.convexoptimization}]
Let $\mu = \cL_\sq\in\mcl P^\upa_\infty(\R_+)$.
Due to $\mcl P^\upa_\infty(\R_+^D)\subset \mcl P_1(\R_+^D)$, the definition of $h$ in~\eqref{e.h(t,mu)=} and~\eqref{e.p.h=hopf-lax_on_Q(1)} in Proposition~\ref{p.h=hopf-lax_on_Q} yield
\begin{align}\label{e.h(t,mu)=sup_P_infty}
    h(t,\mu) = \sup_{ \nu \in \mcl P_\infty(\R_+^D)} \left\{ \psi(\nu) - \mcl T_t(\mu,\nu) \right\}.
\end{align}
The formula \eqref{e.convexoptimization} follows from~\eqref{e.h(t,mu)=sup_P_infty} and Corollary~\ref{c.limF_N=h}. The concavity of the functional follows from Proposition~\ref{p.almost_strict_concavity} and \eqref{e.kantorovich.dual.rep}. The last claim on the uniqueness of maximizers that are monotone follows from Corollary~\ref{c.strict_concavity}.
\end{proof}

\begin{proof}[Proof of Theorem~\ref{t.hopflike}]
This is a consequence of Corollary~\ref{c.limF_N=h}, the identity $h=g$ given in Proposition~\ref{p.reuninverted}, and the expression of $g$ in~\eqref{e.g(t,mu)=}.
\end{proof}

Before proceeding with the proofs of Theorems~\ref{t.parisi.basic} and \ref{t.unique_parisi}, we record the following result.
\begin{proposition}[Other forms of Parisi formula]\label{p.parisi_other}
For every $t\geq0$ and $\mu = \mcl L_q\in\mcl P^\upa_\infty(\R_+^D)$, we have
\begin{align}
    \lim_{N \to +\infty} \bar F_N(t,\mu) & = \sup_{\sq' \in \mcl Q_\infty} \Ll\{ \psi\Ll(\cL_{\sq+\sq'}\Rr)-\int  (t\xi)^* \,\d \cL_{\sq'}  \Rr\} \label{e.p.parisi_other(1)}
    \\
    &= \sup_{\sq'' \in \mcl Q_\infty} \Ll\{ \psi\Ll(\cL_{\sq''}\Rr)-\int_0^1  (t\xi)^*\Ll(\sq''(s)-\sq(s)\Rr) \,\d s  \Rr\}. \label{e.p.parisi_other(2)}
\end{align}
\end{proposition}

\begin{proof}
We denote by v(X) the variational formula in display (X) and by v$^\lambda$(X) the variational formula in display (X) where $\xi$ is replaced by $\xi_\lambda$ therein. We have
\begin{e*}
    \text{v\eqref{e.p.h=hopf-lax_on_Q(2)}}\leq\text{v\eqref{e.p.parisi_other(1)}} \leq \text{v\eqref{e.p.parisi_other(2)}} \leq \text{v\eqref{e.p.h=hopf-lax_on_Q(1)}},
\end{e*}
where the first inequality follows from the identity in~\eqref{e.ttheta=(txi)^*(tnablaxi)} and the last inequality follows from $\int_0^1  (t\xi)^*\Ll(\sq'-\sq\Rr)\geq \mcl T_t(\cL_\sq,\cL_{\sq'})$ (due to the definition of $\mcl T_t$ in~\eqref{e.T_t=}). We now argue that $\text{v\eqref{e.p.h=hopf-lax_on_Q(1)}} \leq \text{v\eqref{e.p.h=hopf-lax_on_Q(2)}}$. 

For every $\pi \in \mcl P(\R^D_+ \times \R^D_+)$ we have by monotone convergence 
\begin{e*}
    \lim_{\lambda \to 0} \int (t\xi^*_\lambda)(y-x) \d \pi(x,y) = \int (t\xi^*)(y-x) \d \pi(x,y)
\end{e*}
In particular, writing 
\begin{e*}
   \text{v$^\lambda$\eqref{e.p.h=hopf-lax_on_Q(1)}} = \sup_{\pi \in \Pi(\mu,\cdot)} \left\{\psi(\pi_2) - \int (t\xi^*_\lambda)(y-x) \d \pi(x,y)\right\},
\end{e*}
we see that $\text{v$^\lambda$\eqref{e.p.h=hopf-lax_on_Q(1)}}$ is lower semi-continuous in $\lambda$ as a supremum of continuous functions of $\lambda$. By Proposition~\ref{p.parisi_gen} we know that for every $\lambda\geq 0$, $\lim_{N\to\infty}\bar F^\lambda_N(t,\mu) = \text{v$^\lambda$\eqref{e.p.h=hopf-lax_on_Q(2)}}$, so according to Lemma~\ref{l.lambda.continuity}, \text{v$^\lambda$\eqref{e.p.h=hopf-lax_on_Q(2)}} is a continuous function of $\lambda \in \R_+$. From Proposition~\ref{p.h=hopf-lax_on_Q} we have for $\lambda > 0$, $\text{v$^\lambda$\eqref{e.p.h=hopf-lax_on_Q(2)}}  =\text{v$^\lambda$\eqref{e.p.h=hopf-lax_on_Q(1)}}$. Letting $\lambda \to 0$ in this we obtain $\text{v\eqref{e.p.h=hopf-lax_on_Q(1)}} \leq \text{v\eqref{e.p.h=hopf-lax_on_Q(2)}}$. This completes the proof since $\lim_{N\to\infty}\bar F_N(t,\mu)=\text{v\eqref{e.p.h=hopf-lax_on_Q(2)}}$ by Proposition~\ref{p.parisi_gen}.
\end{proof}

\begin{proof}[Proof of Theorem~\ref{t.parisi.basic}]
The formula in~\eqref{e.parisi.usual} is given by Proposition~\ref{p.parisi_gen} at $\mu=\delta_0$. The formula~\eqref{e.parisi.preferred} is given by~\eqref{e.p.parisi_other(1)} at $\mu=\delta_0$.
\end{proof}

\begin{proof}[Proof of Theorem~\ref{t.unique_parisi}]

We start by proving that the variational formula in \eqref{e.parisi.preferred} admits a maximizer $\bar \nu$. According to the refined version of the Parisi formula given in \cite[Proposition~6.1]{chen2024free}, we know that the supremum in \eqref{e.p.h=hopf-lax_on_Q(2)} can be taken over $\sp \in \mcl Q_\infty$ satisfying $|\sp|_{L^\infty} \leq C$ for some fixed constant $C >0$. By \cite[Lemma~3.4]{chenmourrat2023cavity}, for every $p \in [1,+\infty)$, the set of $\sp \in \mcl Q_\infty$ satisfying $|\sp|_{L^\infty} \leq C$ is compact for the convergence in $L^p$. Therefore, since by~\eqref{e.F_L^1} the function $\psi$ is Lipschitz continuous with respect to the topology of $L^1$ convergence, there exists a maximizer $\bar \sp \in \mcl Q_\infty$ for the variational formula in \eqref{e.p.h=hopf-lax_on_Q(2)}. Choosing $\bar \mu = \mcl L_{\bar \sp}$ and $\bar \nu = \nabla \xi(\bar \mu)$, we have that $\bar \mu$ and $\bar \nu$ are maximizers of the variational formulas in \eqref{e.parisi.usual} and \eqref{e.parisi.preferred} respectively.

Let $\bar\mu, \bar\nu$ be maximizers of formulas~\eqref{e.parisi.usual} and~\eqref{e.parisi.preferred} respectively. Recall the expression of $t\theta$ in~\eqref{e.ttheta=(txi)^*(tnablaxi)}. By Theorem~\ref{t.convexoptimization} at $\mu=\delta_0$ and the simple fact due to~\eqref{e.T_t=} that $\mcl T_t(\delta_0,\nu) = \int (t\xi)^*\,\d\nu$ for any $\nu$, we have that $(t\nabla\xi)(\bar\mu)$ and $\bar\nu$ are maximizers of the variational formula in~\eqref{e.convexoptimization} at $\mu=\delta_0$. The uniqueness thus follows from the said theorem. Consequently, we also have $\bar \nu = (t \nabla \xi)(\bar \mu)$.
\end{proof}

We can in fact extend the statement concerning the uniqueness of Parisi measures in the following way. 

\begin{proposition}[Uniqueness of Parisi measure II]  \label{p.unique_parisi_gen}

For every $t \ge 0$, the suprema in \eqref{e.p.parisi_other(1)} and \eqref{e.p.parisi_other(2)} are each achieved at exactly one path, say $\bar \sq'$ and $\bar \sq''$ respectively, and these paths satisfy $\sq +\bar\sq'=\bar \sq''$. 
Moreover, the supremum in~\eqref{e.p.parisi_gen} is achieved at paths $\bar \sp$ satisfying $t\nabla\xi(\bar\sp)=\bar\sq'$.

\end{proposition}

\begin{proof} 
For the existence of maximizers, as explained in the proof of Theorem~\ref{t.unique_parisi} a simple compactness and continuity argument yields the existence of a maximizer $\bar \sp \in \mcl Q_\infty$ for the variational formula in~\eqref{e.p.h=hopf-lax_on_Q(2)}. Then, the suprema in \eqref{e.p.parisi_other(1)} and \eqref{e.p.parisi_other(2)} are achieved at $\bar\sq'=t\nabla\xi(\bar\sp)$ and $\bar\sq''=\sq+\bar\sq'$ respectively.
This verifies the existence part. 
For the uniqueness part, notice that it suffices to show that for~\eqref{e.p.parisi_other(2)}.
Let $\bar \sq''$ be any maximizer of~\eqref{e.p.parisi_other(2)}, and let us write $\bar\nu= \mcl L_{\bar\sq''}\in\mcl P^\upa_\infty(\R^D_+)$ and $F=\lim_{N\to\infty}\bar F_N(t,\mu)$. Then, $\mathrm{Law}(\sq(U),\bar\sq''(U))$ is a coupling of $(\mu,\bar\nu)$ and we have
\begin{align*}
     F=\psi\Ll(\cL_{\bar\sq''}\Rr)-\int_0^1  (t\xi)^*\Ll(\bar\sq''(s)-\sq(s)\Rr) \,\d s \stackrel{\eqref{e.T_t=}}{\leq} \psi\Ll(\bar\nu\Rr) -\mcl T_t(\mu,\bar\nu)\stackrel{\eqref{e.convexoptimization}}{\leq}F.
\end{align*}
Hence, the supremum in~\eqref{e.convexoptimization} is achieved at $\bar\nu$. By Theorem~\ref{t.convexoptimization}, such $\bar\nu$ is unique in $\mcl P^\upa_\infty(\R^D_+)$ and thus $\bar \sq''$ has to be unique. Similarly if $\bar \sq'$ is a maximizer in~\eqref{e.p.parisi_other(1)}, then $\sq +\bar \sq'$ is a maximizer in~\eqref{e.p.parisi_other(2)}. Thus $\bar \sq'$ has to be unique and satisfy $\sq +\bar \sq'=\bar \sq''$.
\end{proof}

%
%
%
%
%
%

\section{Explicit representation in terms of martingales}
\label{s.explicit}

The goal of this section is to prove Theorem~\ref{t.explicit.uninverted}. One ingredient of the proof of Theorem~\ref{t.hopflike} is a representation of the function $\psi$ as, for every $\mu \in \mcl P^\upa_1(\R^D_+)$,
\begin{equation} \label{e.rep.psi}
    \psi(\mu) = \inf_{\chi \in \mfk X} \Ll\{ \int \chi \, \d \mu - \psi_*(\chi) \Rr\}.
\end{equation}
In this representation, the set $\mfk X$ and function $\psi_*$ are not uniquely determined. In order to derive Theorem~\ref{t.explicit.uninverted}, the main task compared with the proof of Theorem~\ref{t.hopflike} is to derive a more explicit representation of this function $\psi$. To do this, we will rely heavily on ideas developed in \cite{uninverting}, where very similar results are proved for models with scalar spins.

We recall that we give ourselves a probability space $\msc P = (\Omega, \mcl F, \PP)$ with associated expectation $\EE$, and $(\mcl F_1(t))_{t \ge 0}$, \ldots, $(\mcl F_D(t))_{t \ge 0}$ an independent family of complete filtrations over $\msc P$, which each comes with an adapted Brownian motion $(B_1(t))_{t \ge 0}$, \ldots, $(B_D(t))_{t \ge 0}$.  We denote by $\bmart_1$, \ldots, $\bmart_D$ the spaces of bounded martingales over $\msc P$ with respect to the filtrations $(\mcl F_1({t}))_{t \ge 0}$, \ldots, $(\mcl F_D({t}))_{t \ge 0}$ respectively, with $\bmart = \prod_{d = 1}^D \bmart_d$. 
When we want to make the dependence on the underlying probability space explicit we write $\bmart(\msc P)$ in place of $\bmart$.

We recall that the functions $\psi_d : \mcl P_1(\R_+) \to \R$, $\phi_d : \R_+ \times \R \to \R$ and $\phi_d^* : \R_+ \times \R \to \R$ are defined in \eqref{e.psi_d=}, \eqref{e.def.phid}, and \eqref{e.def.phid*} respectively. Since we have assumed that $\spindist_\s$ is not a Dirac mass, it follows that for every $t \geq 0$, $\phi_d(t,\cdot)$ is strictly convex on $\R$. Adapting the proof of \cite[Lemma~2.2]{uninverting}, we obtain the following explicit representation for $\psi_d$.
\begin{proposition}
 \label{p.rep.psid}
    For every $d \in [D]$, $T_d \in \R_+$, and $\nu_d \in \mcl P_\infty(\R_+)$ supported in $[0,T_d]$, we have
    \begin{multline}  \label{e.rep.psid}
        \psi_d(\nu_d) = \inf_{\alpha_d \in \bmart_d} \EE \bigg[   \phi^*_d(T_d, \alpha_d(T_d)) -\sqrt{2} \alpha_d(T_d) \cdot B_d(T_d) 
        \\- \int_0^{T_d} \nu_d[0,t] \alpha_d^2(t) \, \d t \bigg].
    \end{multline}
    In addition, there is a unique minimizer $\alpha_d \in \bmart_d$ in \eqref{e.rep.psid} and it must satisfy $\alpha_d(t)  = \partial_x \Phi_{\nu_d}(t,X_d(t))$ a.s.\ for almost every $t \leq T_d$, where $X_d$ is the unique strong solution to  %
    \begin{e} \label{e.def.X}       
           \left\{
    \begin{aligned}
      dX_d(t) &= 2\nu_d[0,t]\, \partial_x \Phi_{\nu_d}(t,X_d(t))\, dt + \sqrt{2}\, dB_d(t), \\
      X_d(0)  &= 0.
    \end{aligned}
    \right.
        \end{e}
\end{proposition}
We next recall that, for each $\alpha \in \bmart$, we have defined $\chi_\alpha : \R^D_+ \to \R$ in~\eqref{e.def.chi.alpha}. 
\begin{proposition} \label{p.explicit.rep.psi}
    For every $\nu \in \mcl P^\upa_\infty(\R_+^D)$ with marginals $\nu_1,\ldots, \nu_D$ and $T = (T_1, \ldots, T_D)$ such that for every $d \in [D]$, the measure $\nu_d$ is supported in $[0,T_d]$, we have
    \begin{multline}  \label{e.explicit.rep.psi}
        \psi(\nu) = \inf_{\al \in \bmart} \bigg\{ \int \chi_\al \, \d \nu - \chi_\al(T) 
        \\+  \sum_{d = 1}^D \lambda_{\infty,d} \EE \Ll[   \phi^*_d(T_d, \alpha_d(T_d)) -\sqrt{2} \alpha_d(T_d) \cdot B_d(T_d) \Rr] \bigg\}.
    \end{multline}
    In addition, there is a unique minimizer $\alpha \in \bmart$ in \eqref{e.explicit.rep.psi} and it must satisfy $\alpha_d(t)  = \partial_x \Phi_{\nu_d}(t,X_d(t))$ a.s.\  for every $d \in [D]$ and almost every $t \leq T_d$, where $X_d$ is the solution to \eqref{e.def.X}.
\end{proposition}
\begin{proof} 
    By integration by parts, we have for every $d \in [D]$ that
    \begin{equation*}  
        \int_0^{T_d} \nu_d[0,t] \EE\Ll[\alpha_d^2(t)\Rr] \, \d t = \int \int_s^{T_d} \EE\Ll[\alpha_d^2(t)\Rr] \, \d t \, \d \nu_d(s),
    \end{equation*}
    and thus, due to~\eqref{e.def.chi.alpha},
    \begin{equation*}  
        \sum_{ d = 1}^D \lambda_{\infty, d} \int_0^{T_d} \nu_d[0,t] \EE\Ll[\alpha_d^2(t)\Rr] \, \d t = \chi_\al(T) - \int \chi_\al \, \d \nu.
    \end{equation*}
    Using this, the decomposition of $\psi$ into $\psi_\s$ as in \eqref{e.decomp.psi}, and the formula~\eqref{e.rep.psid}, we obtain~\eqref{e.explicit.rep.psi}. 
\end{proof}

Given $T = (T_1,\dots,T_d)$, for every $\alpha \in \bmart$, we let
\begin{e*}
    \varphi(\alpha) = -\chi_\alpha(T)+ \sum_{d = 1}^D \lambda_{\infty,d} \EE \Ll[   \phi^*_d(T_d, \alpha_d(T_d)) -\sqrt{2} \alpha_d(T_d) \cdot B_d(T_d) \Rr],
\end{e*}
and
\begin{multline*}
    K_0 = \Big\{ (\varphi(\alpha), (\EE[\alpha^2_1(t)])_{t \leq T_1}, \dots, (\EE[\alpha^2_D(t)])_{t \leq T_D} )  \ \Big| \ \alpha \in \bmart, \\ \, \EE[\phi_d^*(T_d,\alpha_d(T_d))] < +\infty \Big\}.
\end{multline*}
We also let $K$ denote the closed convex hull of $K_0$. Since $\bmart$ depends on the underlying probability space $\msc P$, the sets $K$ and $K_0$ also depend on $\msc P$. When we want to make this dependence explicit, we will write $K(\msc P)$ and $K_0(\msc P)$. For convenience we will also let $\mcl K_T = \prod_d [0,T_d]$.
\begin{proposition} \label{p.explicit.uninverted.mixed}
    For every $\mu \in \mcl P^\upa_\infty(\R^D_+)$ and $T = (T_1,\dots,T_D)$ such that $\supp(\mu) +t\nabla\xi([0,1]^D)\subset \prod_{\s\in\sS}[0,T_\s]$, we have 
    \begin{e} \label{e.explicit.uinverted.mixed}
        \lim_{N \to +\infty} \bar F_N(t,\mu)= \inf_{(\varphi,\gamma) \in K} \left\{ \int S_t \chi_{\gamma} \d \mu - \varphi \right\},
    \end{e}
    where $\chi_\gamma$ is given by $\chi_\gamma(x) = \sum_{d=1}^D \lambda_{\infty,d} \int_0^{x_d} \gamma_d(s) \d s$.
\end{proposition}

\begin{proof}
    According to Theorem~\ref{t.convexoptimization}, we have 
    \begin{e*}
        \lim_{N \to +\infty} \bar F_N(t,\mu)=  \sup_{\nu \in \mcl P_\infty(\R^D_+)} \left\{ \psi(\nu) - \mcl T_t(\mu,\nu) \right\}.
    \end{e*}
    In view of the more precise statement of the Parisi formula for multi-species spin glasses given in \cite[Proposition~6.1]{chen2024free}, the supremum in the previous display is reached at a probability measure supported on $\supp(\mu) + t \nabla \xi([0,1]^D)$,
    the assumption on $T$ implies   
    \begin{align*}
        \supp(\mu) + t \nabla \xi([0,1]^D) \subset \mcl K_T.
    \end{align*}
    So, the supremum in the first display of this proof can be restricted to $\nu \in \mcl P(\mcl K_T)$ and thus, thanks to Proposition~\ref{p.explicit.rep.psi}, we have 
    \begin{multline*}
        \lim_{N \to +\infty} \bar F_N(t,\mu)=  \sup_{\nu \in \mcl P(\mcl K_T)} \inf_{\alpha \in \bmart} \bigg\{ \int \chi_\al \, \d \nu - \chi_\al(T) 
            \\+  \sum_{d = 1}^D \lambda_{\infty,d} \EE \Ll[   \phi^*_d(T_d, \alpha_d(T_d)) -\sqrt{2} \alpha_d(T_d) \cdot B_d(T_d) \Rr] - \mcl T_t(\mu,\nu)\bigg\}.
    \end{multline*}
    For $(\varphi,\gamma) \in K$, let $G_{t,\mu}(\nu;\varphi,\gamma)=\int \sum_{d} \lambda_{d,\infty} \Gamma_d \d \nu - \varphi - \mcl T_t(\mu,\nu)$ where we have set $\Gamma_d(x) = \int_0^{x_d} \gamma_d(s) \d s$. The previous display can be rewritten as
    \begin{e*}
        \lim_{N \to +\infty} \bar F_N(t,\mu)=  \sup_{\nu \in \mcl P(\mcl K_T)} \inf_{(\varphi,\gamma) \in K_0} G_{t,\mu}(\nu;\varphi,\gamma),
    \end{e*}
    Since $G_{t,\mu}(\nu;\cdot)$ is affine, we can replace $K_0$ by $K$ in the previous display, to obtain
    \begin{e*}
        \lim_{N \to +\infty} \bar F_N(t,\mu)=  \sup_{\nu \in \mcl P(\mcl K_T)} \inf_{(\varphi,\gamma) \in K} G_{t,\mu}(\nu;\varphi,\gamma).
    \end{e*}
    Observe that the function $G_{t,\mu}(\nu;\cdot)$ is continuous on $\R \times L^1$ and $G_{t,\mu}(\cdot \, ;\varphi,\gamma)$ is lower semicontinuous with-respect to the topology of weak convergence. In addition, $G_{t,\mu}(\nu;\cdot)$ is convex and $G_{t,\mu}(\cdot \, ;\varphi,\gamma)$ is concave because of \eqref{e.kantorovich.dual.rep}. Therefore, we can apply Sion's min-max theorem and obtain  
    \begin{e*}
        \lim_{N \to +\infty} \bar F_N(t,\mu)= \inf_{(\varphi,\gamma) \in K} \sup_{\nu \in \mcl P(\mcl K_T)}  G_{t,\mu}(\nu;\varphi,\gamma).
    \end{e*}
    From Lemma~\ref{l.hopf-lax with linear initial condtion}, Remark~\ref{r.refinement} and the fact that $\supp(\mu) + t \nabla \xi([0,1]^D) \subset \mcl K_T$, we have that for every $(\varphi,\gamma) \in K$,
    \begin{e*}
        \sup_{\nu \in \mcl P(\mcl K_T)}  G_{t,\mu}(\nu;\varphi,\gamma) = \int S_t \chi_{\gamma} \d \mu - \varphi,
    \end{e*}
    where $\chi_\gamma$ is given by $\chi_\gamma(x) = \sum_{d=1}^D \lambda_{\infty,d} \int_0^{x_d} \gamma_d(s) \d s$. This yields~\eqref{e.explicit.uinverted.mixed}.
\end{proof}

\begin{proposition} \label{p.explicit.uninverted.whenrich}
    Assume that there exists a uniform random variable $U : \msc P \to [0,1]$ that is $\mcl F_{d}(0)$-measurable for every $\s\in\sS$. Then \eqref{e.explicit.uninverted} holds with $\bmart=\bmart(\msc P)$.
\end{proposition}

\begin{proof}
    Since $K_0 \subset K$, Proposition~\ref{p.explicit.uninverted.mixed} implies 
    \begin{e}\label{e.p.explicit.uninverted.whenrich1}
        \lim_{N \to +\infty} \bar F_N(t,\mu) \leq \inf_{\alpha \in \bmart(\msc P)} \left\{ \int S_t \chi_\alpha \d \mu - \varphi(\alpha) \right\}.
    \end{e}
    So we only need to prove the converse bound. Let $\msc W$ denote the $D$-dimensional Wiener space with the canonical random variables $(W(t))_{t \geq 0}$ equipped with $D$ filtrations, each generated by a component of the canonical random variables. Let $K_1(\msc W)$ be the convex hull of $K_0(\msc W)$. We can apply Proposition~\ref{p.explicit.uninverted.mixed} to the probability space $\msc W$, and obtain from \eqref{e.explicit.uinverted.mixed} that (we can replace $K(\msc W)$ by $K_1(\msc W)$ using a continuity argument) 
    \begin{e*} 
        \lim_{N \to +\infty} \bar F_N(t,\mu)= \inf_{(\varphi,\gamma) \in K_1(\msc W)} \left\{ \int S_t \chi_{\gamma} \d \mu - \varphi \right\}.
    \end{e*}
    Now, let $(\varphi,\gamma) \in K_1(\msc W)$. By the definition of $K_1(\msc W)$ as the convex hull of $K_0(\msc W)$, there is an integer $n\in\N$ and constants $c_1,\dots,c_n \in [0,1]$ such that  $\sum_i c_i = 1$ and $\alpha^{(1)},\dots,\alpha^{(n)} \in \bmart(\msc W)$ such that 
    \begin{align*}
        &\varphi = \sum_{i = 1}^n c_i \varphi(\alpha^{(i)}), \\
        &\gamma_d(t_d) = \sum_{i =1}^n c_i \EE[(\alpha_d^{(i)}(t_d))^2].
    \end{align*}
    For every $\s \in \sS$, there is a canonical embedding from $\bmart_d(\msc W)$ to the sets of martingales in $\bmart_d(\msc P)$ that are independent of $\mcl F_{d}(0)$. This is because we can identify $W_d$ with the Brownian motion $B_d$. Thus, for every $i \leq n$ and $d \leq D$, we think of $\alpha_d^{(i)}$ as an element of $\bmart_d(\msc P)$ that is independent of $\mcl F_{d}(0)$. Now, by the assumption on $\msc P$ in the statement of Proposition~\ref{p.explicit.uninverted.whenrich}, there exists a random variable $I : \msc P \to \{1,\dots,n\}$ such that 
    \begin{e*}
        \P(I = i) = c_i,\quad\forall 1\leq i\leq n,
    \end{e*}
    and such that $I$ is $\mcl F_{d}(0)$-measurable for every $\s\in\sS$. 
    We let 
    \begin{e*}
        \beta_d(t) = \alpha_d^{(I)}(t)
    \end{e*}
    Since $I$ is $\mcl F_{d}(0)$-measurable and $\alpha_d^{(i)} \in \bmart_d(\msc P)$ is independent of $\mcl F_{d}(0)$, we have that $\beta\in \bmart(\msc P)$. Furthermore, it is easily verified that 
    \begin{align*}
        &\varphi = \varphi(\beta), \\
        &\gamma_d(t_d) =\EE[(\beta_d(t_d))^2].
    \end{align*}
    From this, it follows that 
    \begin{e*}
        \int S_t \chi_{\gamma} \d \mu - \varphi = \int S_t \chi_{\beta} \d \mu - \varphi(\beta) \geq \inf_{\alpha \in \bmart(\msc P)} \left\{ \int S_t \chi_{\alpha} \d \mu - \varphi(\alpha) \right\}.
    \end{e*}
    and thus 
    \begin{align*}
        \lim_{N \to +\infty} \bar F_N(t,\mu) &= \inf_{(\varphi,\gamma) \in K_1(\msc W)} \left\{ \int S_t \chi_{\gamma} \d \mu - \varphi \right\} 
        \\
        &\geq \inf_{\alpha \in \bmart(\msc P)} \left\{ \int S_t \chi_{\alpha} \d \mu - \varphi(\alpha) \right\}.
    \end{align*}
    This along with~\eqref{e.p.explicit.uninverted.whenrich1} implies~\eqref{e.explicit.uninverted}.
\end{proof}

\begin{proposition} \label{p.existence.optimal.martingale}
    There exists a probability space $\bar{\msc P}$ such that \eqref{e.explicit.uninverted} holds with $\bmart=\bmart(\bar{\msc P})$ and the infimum appearing there is achieved.
\end{proposition}

\begin{proof}
    We let $\msc P$ be a probability space such that Proposition~\ref{p.explicit.uninverted.whenrich} holds and let $\alpha^{(n)} \in \bmart(\msc P)$ be such that 
    \begin{e*}
        \lim_{n \to +\infty} \int S_t \chi_{\alpha^{(n)}} \d \mu - \varphi(\alpha^{(n)}) = \lim_{N \to +\infty} \bar F_N(t,\mu).
    \end{e*}
    Since Proposition~\ref{p.explicit.uninverted.whenrich} holds, for $n$ large enough, in view of~\eqref{e.explicit.uninverted}, we must have $\EE[\phi_d^*(T_d,\alpha^{(n)}_d(T_d))] < +\infty$. The definition of $\phi_\s$ in~\eqref{e.def.phid} and that of $\phi_\s^*$ in~\eqref{e.def.phid*} imply that, for each $\s\in\sS$ and for $t \leq T$, $\alpha_d(t)$ is valued in $[-1,1]$ a.s. In particular, the sequence $(((\alpha^{(n)}_\s(t))_{0\leq t\leq T_\s})_{\s\in\sS})_{n\in\N}$ (indexed by $n$) is tight. Applying Prokhorov’s theorem and Skorokhod's representation theorem, not relabeling the extraction, we obtain a probability space $\bar{\msc P} = (\bar\Omega,\bar {\mcl{F}}, \bar\P)$ and random variables $(B,(\alpha_1(t_1),\dots,\alpha_D(t_D))_{t_d \in [0,T_d] \cap \Q})$ such that, for every integer $l\in\N$, every continuous bounded function $J:\R\times (\R^D)^l \to\R$, and $t^i = (t^i_d)_{d\in\sS}\in \prod_{d\in\sS}[0,T_\s]$ for $1\leq i\leq l$, we have 
    \begin{e*}
        \lim_{n \to +\infty} \EE J\Ll(B,\alpha^{(n)}(t^1),\dots,\alpha^{(n)}(t^l)\Rr) = \bar\EE J\Ll(B,\alpha(t^1),\dots,\alpha(t^l)\Rr)
    \end{e*}
    where $\alpha^{(n)}(t^i)=(\alpha^{(n)}_d(t^i_d))_{d\in\sS}$.
    For $t_d \leq T_d$, we let $\bar{\mcl F}_d(t_d)$ denote the $\sigma$-algebra generated by $(B_d(s_d))_{s_d \in [0,t_d] \cap \Q}$ and $(\alpha_d(s_d))_{s_d \in [0,t_d] \cap \Q}$. We set $\alpha_d(t_d) = \bar\EE [\alpha_d(T_d) | \bar{\mcl F}_d(t_d) ]$, by construction $\alpha_d \in \bmart_d(\bar{ \msc P})$ with respect to the filtration $(\bar{\mcl F}_d(t_d))_{t_d \leq T_d}$. Letting $n \to +\infty$, it can then checked as in \cite[Lemma~2.5]{uninverting} that 
    \begin{e*}
        \lim_{N \to +\infty} \bar F_N(t,\mu) = \int S_t \chi_\alpha \d \mu - \varphi(\alpha).
    \end{e*}
    Therefore, we get
    \begin{align*}
        \lim_{N \to +\infty} \bar F_N(t,\mu) \geq \inf_{\alpha \in \bmart(\bar{\msc P})} \left\{ \int S_t \chi_{\alpha} \d \mu - \varphi(\alpha) \right\}.
    \end{align*}
    Since the converse inequality holds thanks to Proposition~\ref{p.explicit.uninverted.mixed}, we obtain the desired result.
\end{proof}

We are now ready to prove Theorem~\ref{t.explicit.uninverted}. Note that as a consequence of the proof below, in addition to the announced result we will observe that a condition similar to condition (1) in \cite[Theorem 2]{uninverting} (i.e. that the support of the Parisi measure is contained in the set of maximizers of $\int_t^1 (s-\EE[\bar \alpha_s^2]) \d s$) is satisfied by the optimal martingale $\bar \alpha$, see Remark~\ref{r.parisi measure in terms of optimal mart} below.

\begin{proof}[Proof of Theorem~\ref{t.explicit.uninverted}]
We let 
\begin{e*}
    \Gamma(\nu,\alpha) = \int \chi_\alpha \d \nu - \varphi(\alpha) - \mcl T_t(\mu,\nu).
\end{e*}
Let $\bar{\msc P}$ be a probability space such that Proposition~\ref{p.existence.optimal.martingale} holds. This means that we have 
\begin{e*}
    \lim_{N \to +\infty} \bar F_N(t,\mu) = \inf_{\alpha\in\bmart(\bar{\msc P})} \sup_{\nu \in \mcl P(\mcl K_T)} \Gamma(\nu,\alpha),
\end{e*}
and that the infimum in the previous display is reached at some $\bar \alpha$. Now let $\bar \nu$ denote the unique maximizer in \eqref{e.convexoptimization}. As discussed at the beginning of the proof of Proposition~\ref{p.explicit.uninverted.mixed}, thanks to \cite[Proposition 6.1]{chen2024free}, we have $\bar \nu \in \mcl P(\mcl K_T)$, hence 
\begin{e*}
    \lim_{N \to +\infty} \bar F_N(t,\mu) = \sup_{\nu \in \mcl P(\mcl K_T)} \inf_{\alpha\in\bmart(\bar{\msc P})}  \Gamma(\nu,\alpha).
\end{e*}
%
In particular, for every $\nu \in \mcl P(\mcl K_T)$ and $\alpha \in \bmart(\bar{\msc P})$, we have 
\begin{e*}
    \Gamma(\bar \nu,\alpha) \geq \lim_{N \to +\infty} \bar F_N(t,\mu) \geq \Gamma(\nu,\bar\alpha).
\end{e*}
Choosing $\alpha = \bar \alpha$ and $\nu = \bar \nu$, we obtain $\lim_{N \to +\infty} \bar F_N(t,\mu) = \Gamma(\bar \nu,\bar \alpha)$. In particular the previous display reads  
\begin{e*}
    \Gamma(\bar \nu,\alpha) \geq \Gamma(\bar \nu,\bar \alpha) \geq \Gamma(\nu,\bar\alpha).
\end{e*}
In particular $\bar \alpha$ is a minimizer of $\Gamma(\bar \nu,\cdot)$, this means that $\bar \alpha$ is a minimizer in the variational formula \eqref{e.explicit.rep.psi} with $\nu = \bar \nu$. According to the optimality condition in Proposition~\ref{p.explicit.rep.psi}, we have that $\bar \alpha$ is uniquely characterized and satisfies $\bar \alpha_d(t) = \partial_x \Phi_{\bar \nu_d}(t,X_d(t))$, where $X_d$ solves \eqref{e.def.X}. In particular, $\bar \alpha_d$ is measurable with respect to the filtration generated by $B_d$. This means that $\bar \alpha$ lies in the image of $\bmart(\msc W)$ in $\bmart(\bar{\msc P})$ through the canonical injection (recall that $\msc W$ denotes the Wiener space and that the canonical injection is obtained by replacing the canonical Brownian motion $W$ by $B$). From this, we obtain that for any probability space $\msc P$, there is a copy of $\bar \alpha$ in $\bmart(\msc P)$ that we still denote $\bar \alpha$ for convenience. Thus we have 
\begin{e*}
    \lim_{N \to +\infty} \bar F_N(t,\mu) = \int S_t \chi_{\bar \alpha} \d \mu - \varphi(\alpha) \geq \inf_{\alpha\in\bmart(\msc P)} \left\{ \int S_t \chi_{ \alpha} \d \mu - \varphi(\alpha) \right\}.
\end{e*}
The other bound follows immediately from Proposition~\ref{p.explicit.uninverted.mixed} recalling that $K_0 \subset K$. 
\end{proof}

\begin{remark} \label{r.parisi measure in terms of optimal mart}
    Also observe that from $\Gamma(\bar \nu,\bar \alpha) \geq \Gamma(\nu,\bar\alpha)$, we have that 
    \begin{e*}
        \int \chi_{\bar \alpha} \d \bar \nu - \mcl T_t(\mu,\bar \nu) \geq \int \chi_{\bar \alpha} \d \nu - \mcl T_t(\mu, \nu).
    \end{e*}
    Taking the supremum over $\nu \in \mcl P(\mcl K_T)$ on the right-hand side, using Lemma~\ref{l.hopf-lax with linear initial condtion} and rearranging, we obtain that
    \begin{e*}
        \int \chi_{\bar \alpha} \d \bar \nu - \int S_t \chi_{\bar \alpha} \d \mu \geq \mcl T_t(\mu, \bar \nu).
    \end{e*}
    This means that $\chi_{\bar \alpha}$ is a Kantorovich potential from $\mu$ to $\bar \nu$ (see Definition~\ref{d.kantorovich.potential}). This condition replaces condition (1) in \cite[Theorem 2]{uninverting} (i.e. that the support of the Parisi measure is contained in the set of maximizers of $\int_t^1 (s-\EE[\bar \alpha_s^2]) \d s$). \qed
\end{remark}

%

%
%
%
%
%
%

\appendix

\section{Proof of the \texorpdfstring{$2$}{2}-dimensional inequality for \texorpdfstring{$\xi^*$}{xi\^*}} \label{s.xi^* ineq}

In this section, we give a proof of Proposition~\ref{p.xi^* inequality}.

\begin{proof}[Proof of Proposition~\ref{p.xi^* inequality}]
    Let us first show that without loss of generality we may assume that $\xi$ is strongly convex on $\R^2_+$. Assume that Proposition~\ref{p.xi^* inequality} holds when $\xi$ is further assumed to be strongly convex. Applying Proposition~\ref{p.xi^* inequality} to the function  $\xi_\lambda= \xi + \frac{\lambda}{2} |\cdot|^2$, we obtain for all real numbers $a \leq a'$ and $b \leq b'$,
    \begin{equation*} 
        \xi_\lambda^*(a,b) + \xi_\lambda^*(a',b') \leq \xi_\lambda^*(a',b) + \xi_\lambda^*(a,b').
    \end{equation*}
    Thus, to show that \eqref{e. xi^* inequality} holds for $\xi$, it is enough to show that for every $y \in \R^2_+$, $\xi^*_\lambda(y) \to \xi^*(y)$ as $\lambda \to 0$. We fix $y \in \R^2_+$, the sequence $(\xi^*_\lambda(y))_{\lambda}$ increases as $\lambda$ decreases to $0$ and is upper-bounded by $\xi^*(y)$, we denote by $\ell \in [0,\xi^*(y)]$ its limit. In addition, for every $x \in \R^2_+$, we have $\xi_\lambda^*(y) \geq x \cdot y - \xi(x) - \frac{\lambda}{2}|x|^2$ letting $\lambda \to 0$ in this inequality yields
    \begin{equation*}
        \ell = \lim_{\lambda \to 0} \xi_\lambda^*(y) \geq x \cdot y - \xi(x).
    \end{equation*}
    Taking the supremum over $x \in \R^2_+$ in the last display, we obtain $\xi^*(y) \leq \ell$. Thus $\lim_{\lambda \to 0} \xi_\lambda^*(y) = \xi^*(y)$, as desired. 
    
    In view of the previous argument, we will assume that $\xi$ is $\lambda$-strongly convex for some $\lambda > 0$ for the rest of this proof.
    
    \medskip
    
    \noindent \emph{Step 1}. We show that for every $y \in \R^2_+$, the supremum in the definition of $\xi^*(y)$ is reached at a unique point $x_{\mathrm{opt}}(y) \in\R^2_+$ and that we have $|x_{\mathrm{opt}}(y)| \leq |y|/\lambda$.
    
     We fix $y \in \R^2_+$. Since $\xi$ is strongly convex, the map $ x \mapsto x \cdot y - \xi(x)$ is strongly concave on $\R^2_+$; so it has at most one maximizer. In addition, we have for every $x \in \R^2_+$ that $\xi(x) \geq \frac{\lambda}{2}|x|^2$. Hence for $|x| > 2|y|/\lambda$, we have 
    \begin{equation*}
        x \cdot y - \xi(x) \leq |x||y| - \frac{\lambda}{2} |x|^2 \leq |x|(|y| - \frac{\lambda}{2} |x|) < 0 \leq \xi^*(y).
    \end{equation*}
    It follows that
    \begin{equation*}
        \xi^*(y) = \sup_{\substack{x \in \R^2_+ \\ |x| \leq \frac{y}{\lambda}}} \left\{ x \cdot y - \xi(x) \right\}.
    \end{equation*}
    The optimization problem in the last display has an optimizer $x_{\mathrm{opt}}(y) \in\R^2_+$ satisfying $|x_{\mathrm{opt}}(y)| \leq |y|/\lambda$.
    
    \medskip
    
    \noindent \emph{Step 2.} We show that $x_{\mathrm{opt}}(y) = 0$ if and only if $y = 0$.
    
     If $y= 0$, the condition $|x_{\mathrm{opt}}(y)| \leq |y|/\lambda$ imposes $x_{\mathrm{opt}}(y) = 0$. Conversely, if we have $x_{\mathrm{opt}}(y) = 0$, then $\xi^*(y) = 0$. We will show that $\xi^*(y) > 0$ for $y \neq 0$. Let $r > 0$, the function $\xi$ is $\mcl C^{1,1}$ on the ball $0$ of center and radius $r$ in $\R^2_+$. In particular, since we have assumed that $\nabla \xi(0)=0$, there exists $\theta=  \theta(r) > 0$ such that for every $|x| < r$
    \begin{e*}
        \xi(x) \leq \xi(0) + \nabla \xi(0) \cdot x + \frac{\theta}{2} |x|^2 = \frac{\theta}{2} |x|^2.
    \end{e*}
    For $|y| < \lambda r$, chosing $x = \varepsilon y$ for $\varepsilon > 0$ small enough, we have $|x| < r$ and $|x| \leq |y|/\lambda$, this yields
    \begin{equation*}
        \xi^*(y) \geq x \cdot y - \xi(x) \geq \varepsilon|y|^2 - \frac{\theta \varepsilon^2}{2} |y|^2 = \varepsilon|y|^2 \left( 1- \frac{\theta \varepsilon}{2}\right) > 0.
    \end{equation*}

    \medskip
    
    \noindent \emph{Step 3.} We show that $\xi^*$ is $\mcl C^{1,1}$ on $\R^2$. In particular, $\xi^*$ is differentiable on $\R^2$ and its gradient is a Lipschitz function which is itself differentiable almost everywhere.
    
     We know that $\xi^*$ is convex, so it is enough to show that the function $y \mapsto \xi^*(y) - \frac{1}{2\lambda} |y|^2$ is concave on $\R^2$. We have
    \begin{align*}
        \xi^*(y) - \frac{1}{2\lambda} |y|^2 &= \sup_{x \in \R^2_+ } \left\{ x \cdot y  - \frac{1}{2\lambda} |y|^2 - \xi(x) \right\} \\
                                            &= \sup_{x \in \R^2_+ } \left\{ - \frac{1}{2} \left| \frac{y}{\sqrt{\lambda}} - \sqrt{\lambda} x \right|^2 - \left(\xi(x)- \frac{\lambda}{2}|x|^2 \right) \right\}.
    \end{align*}
    Thus, $y \mapsto \xi^*(y) - \frac{1}{2\lambda} |y|^2$ is the supremum of a jointly concave functional on $\R^2 \times \R^2_+$ so it is concave on $\R^2$.

\medskip

    \noindent \emph{Step 4.} We show that for $y \in \R^2_{++}$, $x_{\mathrm{opt}}(y) = \nabla \xi^*(y)$.

     We fix $y \in \R^2_{++}$, we have for every $y' \in \R^2_{++}$ and $x \in \R^2_{++}$,
    \begin{align*}
       \xi^*(y') &\geq  x_{\mathrm{opt}}(y) \cdot y' - \xi(x_{\mathrm{opt}}(y)) \\
                &=  x_{\mathrm{opt}}(y) \cdot y - \xi(x_{\mathrm{opt}}(y)) + x_{\mathrm{opt}}(y) \cdot(y'-y) \\
                &= \xi^*(y) +  x_{\mathrm{opt}}(y) \cdot(y'-y).
    \end{align*}
    This means that for every $v \in \R^2$ such that $|v|=1$, we have for $\varepsilon > 0$ small enough $y + \varepsilon v \in \R^2_{++}$ and
    \begin{equation*}
        \frac{\xi^*(y+\varepsilon v) - \xi^*(y)}{\varepsilon} \geq x _{opt}(y) \cdot v.
    \end{equation*}
    Letting $\varepsilon \to 0$, we obtain
    \begin{equation*}
        \nabla \xi^*(y) \cdot v \geq x_{\mathrm{opt}}(y) \cdot v.
    \end{equation*}
    Since the last display is valid for all $v$ in the sphere of radius $1$ in $\R^2$, it imposes that $\nabla \xi^*(y) = x_{\mathrm{opt}}(y)$.
    
    \medskip
    

    
    
    \noindent \emph{Step 5.} We show that $\partial_{a} \partial_{b} \xi^* \leq 0$ almost everywhere on $\R^2_+$.

\medskip

     \emph{Step 5.1.} Letting $V = \nabla \xi(\R^2_{++})$, we show that $\R^2_{++} \setminus \bar V = U_a \cup U_b$ where $U_a = \{ y \in \R^2_{++} \setminus \bar{V}  \big| \, x_{\mathrm{opt},a} = 0 \}$ and $U_b = \{ y \in \R^2_{++} \setminus \bar{V}  \big| \, x_{\mathrm{opt},b} = 0 \}$ are disjoint open sets.

    Clearly $U_a \cup U_b \subset \R^2_{++} \setminus \bar V$. Conversely, let $y \in \R^2_{++} \setminus \bar V$, and by contradiction assume that $x_{\mathrm{opt}}(y) \in \R^2_{++}$. Perturbing slightly around $x_{\mathrm{opt}}(y)$, we get the first order condition and $\nabla \xi(x_{\mathrm{opt}}(y)) = y$, this is a contradiction. Hence, we  have $x_{\mathrm{opt},a}(y) = 0$ or $x_{\mathrm{opt},b}(y) = 0$. This proves that $\R^2_{++} \setminus \bar V = U_a \cup U_b$. By Step 2, we must have $U_a \cap U_b = \emptyset$. Finally, let us show that $U_a$ is open. Given $y \in U_a$, we have $x_{\mathrm{opt},a}(y) = 0$ and $x_{\mathrm{opt},b}(y) > 0$. By Step 4, $x_{\mathrm{opt},b}(y)$ is a continuous function of $y$ so for every $y' \in B_r(y) \subset \R^2_{++} \setminus \bar V$ with $r > 0$ small we have $x_{\mathrm{opt},b}(y') > 0$. This imposes $x_{\mathrm{opt},a}(y') = 0$ as otherwise $y' \in V$ and thus $B_r(y) \subset U_a$. This proves that $U_a$ is an open set and by symmetry we also have that $U_b$ is an open set.
    
    \medskip
    
    \noindent \emph{Step 5.2.} We show that $\partial_{a}\partial_{b} \xi^* \leq 0$ on $V = \nabla \xi(\R^2_{++})$.

According to Step 3, $\xi^*$ is differentiable with Lipschitz gradient on $\R^2$. According to \eqref{e.xi^*(nabla_xi)=} in Lemma~\ref{l.save_the_day}, we have that for every $x \in \R^2_{++}$, 
    \begin{equation*}
        \xi^*(\nabla \xi(x))= x \cdot \nabla \xi(x) - \xi(x).
    \end{equation*}
    Differentiating the expression in the previous display, we obtain $\nabla \xi^*(\nabla \xi(x)) = x$. Now let $y \in V$, there exists $x \in \R^2_{++}$ such that $y = \nabla \xi(x)$, and we have
    \begin{equation*}
        \nabla \xi(\nabla \xi^*(y)) = \nabla \xi(\nabla \xi^*(\nabla \xi(x))) = \nabla \xi(x) = y.
    \end{equation*}
    Finally, since $\xi$ is strongly convex on $\R^2_+$, the function $\nabla \xi$ is continuous and injective on $\R^2_+$, so by invariance of domain, $V$ is an open set. Differentiating the last display with respect to $y \in V$, we obtain that for almost all $y \in V$,
    \begin{e*}
        \nabla^2 \xi^*(y) = [\nabla ^2 \xi(\nabla \xi^*(y))]^{-1}.
    \end{e*}
    The $2 \times 2$ matrix $A = \nabla ^2 \xi(\nabla \xi^*(y))$ is symmetric positive definite and has non-negative coefficients, this is because the condition \eqref{e.def H_N} imposes some positivity constraint on the coefficients of the power series expansion of~$\xi$ as explained in Remark~\ref{r.xi.monotonicity}. Since $(A^{-1})_{ab} = -\det(A)^{-1} A_{ab}$, we obtain that $(A^{-1})_{ab} \leq 0$, as desired. 
    
    \medskip
    
     \noindent \emph{Step 5.3.} We show that $\partial_{a}\partial_{b} \xi^* \leq 0$ almost everywhere on $U_a \cup U_b$.
    
Let $y \in U_a$, we have $x_{\mathrm{opt},a}(y) = 0$ and $x_{\mathrm{opt},b}(y) > 0$, so 
     \begin{equation*}
         \xi^*(y) = \sup_{x_b > 0} \left\{ x_by_b - \xi(0,x_b) \right\}.
     \end{equation*}
    The value of the unique optimizer in the previous display is $x_{\mathrm{opt},b}(y)$. In particular, $x_{\mathrm{opt},b}(y)$ is independent of the value of $y_a$ and by Step 4, we have $\partial_b \xi^*(y) = x_{\mathrm{opt},b}(y)$. Thus, for almost all $y \in U_a$, $\partial_{a}\partial_{b} \xi^*(y) = 0$. By symmetry, this also holds for every $y \in U_b$.

\medskip

    \noindent \emph{Step 5.4.} We show that the boundary of $V$ in $\R^2_{++}$ is a Lebesgue negligible set.

 Let us show that the boundary of $V$ is $\nabla \xi(\R^2_+ \setminus \R^2_{++})$. Since $\R^2_+ \setminus \R^2_{++}$ is Lebesgue negligible and $\nabla \xi$ is a smooth function, the desired result will follow. Let $\bar V$ denote the closure of $V$ in $\R^2_{++}$ and let $y \in \bar V$. There exists a sequence $(x_n)_n$ in $\R^2_{++}$ such that $\nabla \xi(x_n) \to y$. The sequence $(\nabla \xi(x_n))_n$ is bounded and $\nabla \xi(x_n) - \lambda x_n \in \R^2_+ $ so the sequence $(x_n)_n$ is bounded. Let $x \in \R^2_+$ denote a subsequential limit of $(x_n)_n$, up to extraction we have $y = \lim_n \nabla \xi(x_n) = \nabla \xi(x)$. Hence $\bar V \subset \nabla \xi (\R^2_+)$. Conversely given $y \in \nabla \xi (\R^2_+)$, for some $x \in \R^2_+$, we have $y = \nabla \xi(x) = \lim_n \nabla \xi(x_a+\frac{1}{n},x_b  +\frac{1}{n})$, so $y \in \bar V$. 
Thus $\bar V = \nabla \xi (\R^2_+)$ and $\bar V \setminus V = \nabla \xi(\R^2_+) \setminus \nabla \xi(\R^2_{++})$. Since $\xi$ is assumed to be strongly convex on $\R^2_+$, $\nabla \xi$ is injective on $\R^2_+$ and we have $\nabla \xi(\R^2_+) \setminus \nabla \xi(\R^2_{++}) =\nabla \xi(\R^2_+ \setminus \R^2_{++})$. Thus $\bar V \setminus V = \nabla \xi(\R^2_+ \setminus \R^2_{++})$ as announced.
     
     \medskip
     
    \noindent \emph{Step 6.} We show that \eqref{e. xi^* inequality} holds when $\xi$ is $\lambda$-strongly convex.

We let $a \leq a'$, $b \leq b'$, such that $a \neq a'$ and $b \neq b'$. This way $a' = a +h$ and $b' = b +k$ for $h,k \in \R_{++}^2$. We then have 
    \begin{align*}
        \big( \xi^*(a,b)- \xi^*(a',b) \big) -  \big( \xi^*(a,b')- \xi^*(a',b') \big)
        \\
        = \int_0^1 \int_0^1 hk \partial_{a}\partial_{b} \xi^*(a+th,b+sk) \d s \d t.
    \end{align*}
    According to Step 5, the right-hand side in the previous display is the integral of an almost everywhere nonnegative function, so it is nonnegative and
    \begin{equation*}
        \big( \xi^*(a,b)- \xi^*(a',b) \big) -  \big( \xi^*(a,b')- \xi^*(a',b') \big) \geq 0.
    \end{equation*}
    Rearranging the terms in the previous display, we obtain \eqref{e. xi^* inequality} when $a \neq a'$ and $b \neq b'$. Finally, by continuity of $\xi^*$, the inequality \eqref{e. xi^* inequality} holds even  when $a = a'$ or $b = b'$. 
\end{proof}

\vspace{-0.5cm}

\small
\bibliographystyle{plain}
\bibliography{ref}

\end{document}